\documentclass{mod}
\usepackage{url}
\usepackage{setspace}
\usepackage{scrextend}
\usepackage{cite}
\usepackage{multirow}

\usepackage{amsmath}
\usepackage{amsthm}
\usepackage{amssymb}

\usepackage{mathtools}
\usepackage{mathrsfs}
\usepackage[all]{xy}
\usepackage[toc,page]{appendix}
\usepackage{titlesec}
\usepackage[nottoc]{tocbibind}
\usepackage[titles]{tocloft}

\usepackage{indentfirst}

\usepackage{subcaption}
\usepackage[shortlabels]{enumitem}

 \newcounter{mainthm}
 \newtheorem{main}[mainthm]{Theorem}

 \newtheorem{thm}{Theorem}[section]
 
 \theoremstyle{plain}
 \newtheorem{lem}[thm]{Lemma}
 \newtheorem{prop}[thm]{Proposition}
 \newtheorem{cor}[thm]{Corollary}
 \newtheorem{defn-thm}[thm]{Definition-Theorem}
 \newtheorem{sublemma}[thm]{Sublemma}
 \newtheorem{defn-lem}[thm]{Definition-Lemma}

\newtheorem{conjecture}
[thm]{Conjecture}

 \theoremstyle{definition}
 \newtheorem{defn}[thm]{Definition}

\newtheorem{assumption}[thm]{Assumption}

\newtheorem{convention}
[thm]{Convention}

\newtheoremstyle{rmk}
{5pt}
{5pt}
{}
{}
{\itshape}
{}
{.5em}
{}

\theoremstyle{rmk}
\newtheorem{rmk}[thm]{Remark}

\newtheoremstyle{note}
{8pt}
{5pt}
{\itshape}
{10pt}
{\bfseries}
{}
{.5em}
{}

\theoremstyle{note}
\newtheorem*{note}{}

\newcommand{\id}{\mathrm{id}}

\newcommand{\mathds}[1]{\text{\usefont{U}{dsrom}{m}{n}#1}}
\newcommand{\one}{\mathds {1}}

\newcommand{\UD}{\mathscr{UD}}
\newcommand{\simud}{\stackrel{\mathrm{ud}}{\sim}}
\newcommand{\con}{\mathsf{con}}

\newcommand{\pr}{\mathrm{pr}}
\newcommand{\ev}{\mathrm{ev}}
\newcommand{\Ev}{\mathrm{Ev}}
\newcommand{\forget}{{\mathfrak {forget}}}
\newcommand{\uu}{\mathbf u}
\newcommand{\ia}{{\mathfrak a}}
\newcommand{\e}{\mathbf e}
\newcommand{\B}{\mathsf B}
\newcommand{\mi}{\mathfrak i}

\newcommand{\mg}{\mathbf g}

\newcommand{\tri}{{\mathrm{tri}}}

\newcommand{\m}{\mathfrak m}

\newcommand{\M}{\mathfrak M}

\newcommand{\mC}{\mathfrak C}

\newcommand{\oi}{{[0,1]}}

\newcommand{\f}{\mathfrak f}
\newcommand{\F}{\mathfrak F}

\newcommand{\h}{\mathfrak h}

\newcommand{\g}{\mathfrak g}
\newcommand{\G}{\mathfrak G}

\newcommand{\q}{\mathfrak q}

\newcommand{\HL}{H^*(L)}
\newcommand{\OL}{\Omega^*(L)}

\DeclareMathOperator{\Sp}{Sp}
\DeclareMathOperator{\incl}{Incl}
\DeclareMathOperator{\eval}{Eval}

\DeclareMathOperator{\val}{\mathsf{v}}
\DeclareMathOperator{\trop}{\mathfrak {trop}}
\DeclareMathOperator{\Hom}{Hom}
\DeclareMathOperator{\DA}{DA}
\DeclareMathOperator{\CU}{CU}

\DeclareMathOperator{\Obj}{Obj}
\DeclareMathOperator{\Mor}{Mor}
\DeclareMathOperator{\CC}{\mathbf{CC}}

\DeclareMathOperator{\CF}{CF}
\DeclareMathOperator{\HF}{HF}

\DeclareMathOperator{\rCC}{\mathbf{\widetilde{CC}}}
\DeclareMathOperator{\rHH}{\widetilde{HH}}
\DeclareMathOperator{\Corr}{Corr}
\DeclareMathOperator{\CO}{{\mathbb {CO}}}

\newcommand*{\Scale}[2][4]{\scalebox{#1}{$#2$}}%

\titleformat{\paragraph}[runin]{\small\sffamily\bfseries
}{}{}{}[]

\titleformat{\subsubsection}[runin]{\itshape\normalsize}{\thesubsubsection \ }{0em}{}[\mbox{. } ]
\titlespacing{\subsubsection}
{0pt}
{2.5ex plus 1ex minus .2ex}
{0pt}

\usepackage{enumitem}
\setlist[description]{font=\normalfont\itshape\textbullet\space}

\usepackage{titletoc}

\titlecontents{section}
[0.72em]
{\footnotesize\bfseries\vspace{3pt}}%
{\thecontentslabel.\enspace}
{}
{\titlerule*[0.38pc]{.}\contentspage}%

\titlecontents{subsection}
[0em]
{\footnotesize \vspace{0pt}}%
{\qquad \quad \thecontentslabel.\enspace}
{}
{\titlerule*[0.5pc]{.}\contentspage}%

\titlecontents{subsubsection}
[0em]
{\footnotesize \vspace{1pt}}%
{\qquad \qquad \thecontentslabel.\enspace}
{}
{\titlerule*[0.5pc]{.}\contentspage}%

\begin{document}
\title[\small Family Floer superpotential's critical values are eigenvalues of quantum product by $c_1$]{\large Family Floer superpotential's critical values are eigenvalues of quantum product by $c_1$}
\author[Hang Yuan]{\small Hang Yuan}
\begin{abstract} {\sc Abstract:}  
	 In the setting of the non-archimedean SYZ mirror construction in \cite{Yuan_I_FamilyFloer}, we prove the folklore conjecture that the critical values of the mirror superpotential are the eigenvalues of the quantum multiplication by the first Chern class.
	 Our result relies on a weak unobstructed assumption, but it is usually ensured in practice by Solomon's results \cite{Solomon_Involutions} on anti-symmetric Lagrangians.
	 Lastly, we note that some explicit examples are presented in the recent work \cite{Yuan_local_SYZ}.
\end{abstract}
\maketitle
%
%

\tableofcontents

\setlength{\parindent}{5.5mm}	\setlength{\parskip}{0em}

\section{Introduction}

The mirror symmetry phenomenon does not merely focus on the Calabi-Yau setting.
Assume that a compact symplectic manifold $(X,\omega)$ is Fano (or more generally that $-K_X$ is nef). The mirror of $X$ is expected to be a Landau-Ginzburg model $(X^\vee, W^\vee)$ consisting of a space $X^\vee$ equipped with a global function $W^\vee$ called the superpotential.
In the context of the homological mirror symmetry \cite{KonICM}, a celebrated folklore conjecture states that: (known to Kontsevich and Seidel, compare also \cite[\S 2.9]{Sheridan16})

\vspace{-0.2em}

\begin{conjecture}
	\label{conjecture_folklore}
	The critical values of the mirror Landau-Ginzburg (LG) superpotential $W^\vee$ are the eigenvalues of the quantum multiplication by the first Chern class of $X$.
\end{conjecture}

\vspace{-0.2em}

The conjecture is first proved by Auroux \cite{AuTDual} in the Fano toric case, concerning the Landau-Ginzburg mirror associated to a toric Lagrangian fibration in $X$ \cite{Cho_Oh, FOOOToricOne}.

\begin{thm}
	Conjecture \ref{conjecture_folklore} holds for the family Floer mirror Landau-Ginzburg superpotential.
\end{thm}

See \S \ref{ss_main_thm} for the precise statement. In brief, we aim to prove the conjecture with nontrivial Maslov-0 quantum correction, and actually we don't need to be limited to the family Floer scope.
It looks like a ``small'' step to include Maslov-0 disks, but it is indeed difficult and requires some new ideas (\S \ref{ss_new_operator}).
We also find new examples recently in \cite{Yuan_local_SYZ} to justify and amplify the value of this paper.

The symplectic-geometric ideas we use are not original, but it is \textit{not} our main point at all. Instead, a major point in this paper is to demonstrate that the classical ideas fit perfectly into the framework of the family Floer mirror construction in \cite{Yuan_I_FamilyFloer}. 
For example, a key new ingredient for our proof is a preliminary version of ``\textit{{Fukaya category over affinoid coefficients}}'' (i.e. the self Floer cohomology for $\UD$ in \S \ref{s_self_HF}).
Roughly, the coefficient we use is not $\mathbb C$, not the Novikov field $\Lambda$, but an affinoid algebra over $\Lambda$.
The merit is that we can ignore all the affects of Maslov-0 disk obstruction \textit{up to affinoid algebra isomorphism on the coefficients} using the ideas in the author's thesis \cite{Yuan_I_FamilyFloer}.
We expect the affinoid coefficients should be important to develop Abouzaid's family Floer functor \cite{AboFamilyFaithful,AboFamilyWithout} with Maslov-0 corrections. A work in progress is \cite{Yuan_affinoid_coeff}.
For the above speculative version of Fukaya category, one might study many topics like splitting generations \cite{AFOOO} and Bridgeland stability conditions (cf. \cite{smith2017stability} \cite{Kon_website_stab}) in some similar ways in future studies.
We also expect that it is related to the non-archimedean geometry, like the coherent analytic sheaves, on the family Floer mirror space, potentially being the analytification of an algebraic variety over the Novikov field.
But unfortunately, the precondition of all the story is to allow the Maslov-0 disks in the course. Otherwise, we cannot even realize a simple algebraic variety like $Y=\{x_0x_1=1+y\}$ as a family Floer mirror space and thus miss the new examples of Conjecture \ref{conjecture_folklore} in \cite{Yuan_local_SYZ}.


\begin{figure}
	\centering
	\includegraphics[scale=0.35]{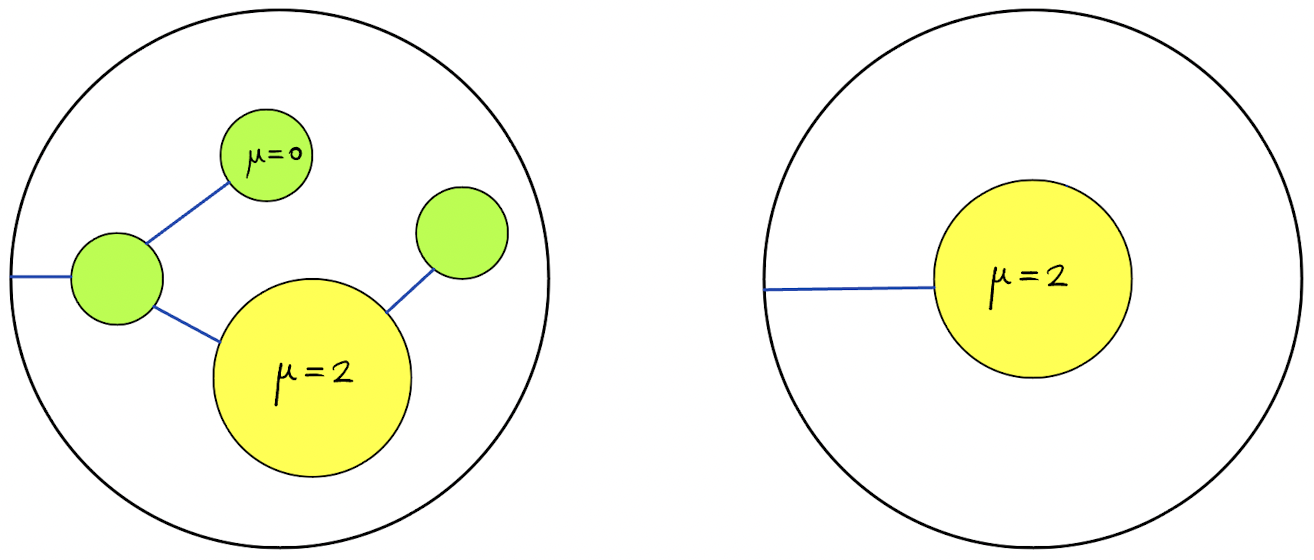}
	\caption{
		\small A piece of contribution to the LG superpotential \textbf{with / without} the Maslov-0 quantum correction. When we perturb $J$, all the Maslov-0 disks (green) in the Feynman-diagrams of the minimal model $A_\infty$ algebras are continually created and annihilated for the wall-crossing. This is more like a sort of ``quantum fluctuation''. This is why the LG superpotential is only well-defined up to affinoid algebra isomorphism.
}
	\label{figure_hp_maslov_0}
	\vspace{-0.5em}
\end{figure}

\subsubsection*{\textit{(1) What is the difficulty of allowing Maslov-0 corrections?}}
\hfill

First and foremost, the family Floer mirror superpotential is only well-defined \textbf{\textit{up to affinoid algebra isomorphisms}}, or equivalently, up to family Floer analytic transition maps \cite{Yuan_I_FamilyFloer}.
This is because the Maslov-0 quantum correction give rise to 
certain mirror non-archimedean analytic topology.
Note that the idea of uniting Maurer-Cartan sets, adopted by Fukaya and Tu \cite{FuBerkeley,FuCounting,Tu}, cannot achieve the family Floer mirror non-archimedean \textit{analytic} structure. We must go beyond their Maurer-Cartan picture; otherwise, although Fukaya and Tu did tell the correct {local} analytic charts, the local-to-global gluing was only \textit{set-theoretic}.
In fact, we never think of bounding cochains but work with formal power series in the corresponding affinoid algebra directly \cite{Yuan_I_FamilyFloer}.
One of many crucial flaws in the previous work concerns the lack of studies of the \textit{divisor axiom} for $A_\infty$ structures\footnote{The earliest literature we find for open-string divisor axiom is due to Seidel \cite{Seidel2006biased}. See also \cite{AuTDual, FuCyclic}.}.
The \textit{ud-homotopy theory} developed in \cite{Yuan_I_FamilyFloer} is indispensable for the mirror \textit{analytic} gluing.
Remarkably, the ud-homotopy theory plays the key role for the folklore conjecture once again (cf. \S \ref{ss_new_operator}).

Putting another way, we must work with the minimal model $A_\infty$ algebras, while paying attention to the wall-crossing phenomenon simultaneously (Figure \ref{figure_hp_maslov_0}).
Without minimal model $A_\infty$ algebras, the superpotential would be only $\Omega^0(L)$-valued and was not well-defined. This issue was not so fatal if the Maslov indices $>0$, but it is indeed fatal if we allow Maslov-0 correction.
As another evidence, a minimal model $A_\infty$ algebra is intuitively like counting \textit{\textbf{holomorphic pearly-trees}}  (cf. \cite{Fukaya2009canonical_Morse}) which are also used in Sheridan's proof of HMS of Calabi-Yau projective hypersurfaces \cite[\S 6.2]{Sheridan15}.
The conventional ideas of the folklore conjecture cannot work or need to be improved for holomorphic pearly-trees or minimal model $A_\infty$ algebras, if nontrivial Maslov-0 disks exist (Figure \ref{figure_hp_maslov_0}).
A daunting trouble is to handle the {minimal model} $A_\infty$ algebra's Hochschild cohomology (see \S \ref{ss_new_operator}).

\subsubsection*{
\textit{(2) Why does the family Floer method give a convincing approach to construct the mirror?}}
\hfill

The answer in short is simply that we can prove the Conjecture \ref{conjecture_folklore} supported by examples below. 

Gross and Siebert proposed a method of mirror construction in the algebro-geometric framework \cite{GS11} based on the SYZ proposal \cite{SYZ}. Gross-Hacking-Keel \cite{GHK15} also give a synthetic algebro-geometric construction of the mirror family for local Calabi-Yau surfaces.
These approaches turn out to be successful in many aspects, such as \cite{Hacking_Keating_2020homological,GHK_birational,GHKK18}.
But, in our biased opinion, the folklore Conjecture \ref{conjecture_folklore} inevitably relies on Floer-theoretic or symplectic-geometric ideas in the end.
The family Floer mirror construction \cite{Yuan_I_FamilyFloer} can naturally include the Landau-Ginzburg superpotential. 
Such a Floer-theoretic mirror construction succeeded in rigorizing the T-duality of smooth Lagrangian torus fibration with full quantum correction considered, and leads to 
a mathematical precise statement of SYZ conjecture with singularities in \cite{Yuan_local_SYZ,Yuan_conifold,Yuan_A_n}.

\subsection{Main result}
\label{ss_main_thm}

Let $(X,\omega)$ be a symplectic manifold which is closed or convex at infinity. Let $\dim_{\mathbb R}X=2n$. Let $L$ be an oriented, spin, and compact Lagrangian submanifold satisfying the \textit{semipositive} condition that prohibits any nontrivial negative Maslov index holomorphic disk.
In practice, a graded or special Lagrangian submanifold is always semipositive by \cite[Lemma 3.1]{AuTDual}.
Note that $L$ is \textit{not} necessarily a torus.
Given an $\omega$-tame almost complex structure $J$, we can associate to $L$ a superpotential function
\begin{equation}
	\label{superpotential_L}
	W_L:= W_L^J=\sum_{\beta\in \pi_2(X,L), \mu(\beta)=2} T^{E(\beta)} Y^{\partial\beta} \m_{0,\beta}
\end{equation}
where we use the \textit{minimal model $A_\infty$ algebra} 
\[
\m=(\m_{k,\beta})_{k\in\mathbb N, \beta\in\pi_2(X,L)}
\]
of the cochain-level $A_\infty$ algebra associated to $L$.
Here we work over the \textit{Novikov field}
$
\Lambda=\mathbb C((T^{\mathbb R}))$, $Y$ is a formal symbol, so $W_L\in \Lambda[[\pi_1(L)]]$.
The reverse isoperimetric inequality \cite{ReverseI} ensures that $W_L$ converges on an analytic open neighborhood of $H^1(L;U_\Lambda)$ in $H^1(L;\Lambda^*)$.

	Although the summation in (\ref{superpotential_L}) runs only over Maslov-2 topological disks $\beta$, it actually can involve \textbf{\textit{all}} the possible Maslov-0 obstruction in the same time. This is simply the nature of minimal model $A_\infty$ algebras, and the idea is illustrated in Figure \ref{figure_hp_maslov_0}.
In general, allowing the Maslov-0 disks, the superpotential function $W_L^J$ is \textit{not} well-defined and rely on the various choices like $J$. But, as in \cite{Yuan_I_FamilyFloer} (see Lemma \ref{wall_crossing_coordiante_change_lem}), we can show that the $W_L$ is well-defined \textit{{up to affinoid algebra isomorphism}}, and we can also show that the critical points and critical values are preserved.

\begin{assumption}
	\label{assumption_weak_MC}
	The formal power series\footnote{It is related to the weak Maurer-Cartan equations in the literature, but we forget the latter and only think $Q_L^J$.} $
	Q_L^J:=\sum_{\mu(\beta)=0} T^{E(\beta)} Y^{\partial\beta} \m_{0,\beta}
	$
	vanishes identically.
\end{assumption}

Notably, the recent results in \cite{Yuan_unobs} provide strong support for considering the above assumption reasonable, and it is likely that the assumption holds in many situations.
It does not say that the counts of Maslov-0 disks vanish. These disks still contribute to both the minimal model $A_\infty$ algebras and the wall-crossing phenomenon, but their contributions ultimately cancel out in the coefficient affinoid algebra.
Moreover, it is \textit{not} a restrictive assumption. Thanks to the work of Solomon \cite{Solomon_Involutions}, a very useful sufficient condition of the above Assumption \ref{assumption_weak_MC}, which we often adopt in practice, is given as follows:

\begin{assumption}
	\label{assumption_weak_MC_anti_involution}
There is an anti-symplectic involution $\varphi$ that preserves $L$.
\end{assumption}

Indeed, such a $\varphi$ gives a pairing on $\pi_2(X,L)$ via $\beta\xleftrightarrow{} -\varphi_*\beta$. Then, the Maslov-0 obstruction terms $\m_{0,\beta}$ in $Q_L^J$ are then canceled pairwise.
In practice, such an involution $\varphi$ is often not hard to find (see also \cite{Solomon_Symmetry_Lag}).
Note that the vanishing of $Q_L^J$ does not rely on $J$ due to \cite{Yuan_I_FamilyFloer}.
Beware that we have not discuss the family Floer setting or the SYZ setting at this moment.
Note also that we do not require $L$ is a \textit{torus} here.
Now, we state the main result.

\begin{main}	[Theorem \ref{eigenvalue_general_thm}]
	\label{Main_thm_c_1_general_thm}
	Under Assumption \ref{assumption_weak_MC}, if the cohomology ring $H^*(L)$ is generated by $H^1(L;\mathbb Z)$, then any critical value of $W_L$ is an eigenvalue of the quantum product by $c_1$. Moreover, the same result still holds if we use different choices.
\end{main}

The main new progress and difficulty is that we allow Maslov-0 holomorphic disks.
It is conceivable that if the above $L$ lives in a family of Lagrangian submanifolds $\{L_q\}$, then one may extend the domain of $W_L$ to a larger open analytic space.
Up to the Fukaya's trick, this amounts to study how $W_L$ depends on the choice of $J$. Recall that $W_L$ is only invariant up to affinoid algebra isomorphisms.
This is because of a sort of `quantum fluctuation', when we think of the minimal model $A_\infty$ algebra together with the wall-crossing phenomenon in the meantime (see Figure \ref{figure_hp_maslov_0}).

\subsubsection{Family Floer mirror construction setting}
As a byproduct of the above Theorem \ref{Main_thm_c_1_general_thm}, we can verify the family Floer mirror superpotential in \cite{Yuan_I_FamilyFloer} respects Conjecture \ref{conjecture_folklore}.
We first briefly review it.
Suppose $(X,\omega)$ is a symplectic manifold which is closed or convex at infinity. Suppose there is a smooth Lagrangian torus fibration $\pi_0:X_0\to B_0$ in an open subset $X_0$ in $X$ over a half-dimensional integral affine manifold $B_0$.
Suppose that $\pi_0$ is \textit{semipositive} in the sense that any holomorphic disk bounding a fiber $L_q:=\pi_0^{-1}(q)$ has a non-negative Maslov index.
By \cite{AuTDual}, a sufficient condition is to assume $\pi_0$-fibers are graded or special Lagrangians.
For simplicity, we also require that the $\pi_0$-fibers satisfy Assumption \ref{assumption_weak_MC} or \ref{assumption_weak_MC_anti_involution}.
Then, we state the result of SYZ duality with quantum correction considered:

\begin{thm}[{\cite{Yuan_I_FamilyFloer}}]
	\label{Main_theorem_thesis}
		We can naturally associate to the pair $(X,\pi_0)$ a triple
$
	\mathbb X^\vee:=(X_0^\vee,W_0^\vee, \pi_0^\vee)
$
	consisting of a non-archimedean analytic space $X_0^\vee$ over $\Lambda$, a global analytic function $W_0^\vee$, and an affinoid torus fibration $\pi_0^\vee: X_0^\vee\to B_0$
such that the following properties hold:
	
\begin{enumerate}[i)]
		\item The analytic structure of $X_0^\vee$ is unique up to isomorphism.
		\item The integral affine structure on $B_0$ from $\pi_0^\vee$ coincides with the one from the fibration $\pi_0$.
		\item The set of closed points in $X_0^\vee$ coincides with $\bigcup_{q\in B_0} H^1(L_q;U_\Lambda)$.
	\end{enumerate}
\end{thm}

Notice that this mirror construction uses holomorphic disks in $X$ rather than just in $X_0$, and we should place $\pi_0$ inside $X$ to understand the statement.
For simplicity, we often omit the subscripts $0$'s, so we often write $(X^\vee, W^\vee,\pi^\vee)$, if there is no confusion.
The \textit{ud-homotopy theory} and the category $\UD$ in \cite{Yuan_I_FamilyFloer} is indispensable for the non-archimedean analytic topology in Theorem \ref{Main_theorem_thesis}.
	The usual homotopy theory can only achieve a gluing in the set-theoretic level.
	Very intuitively, the family Floer theory is `algebrized' by the category $\UD$, and the choice issues are all controlled by the ud-homotopy relations in $\UD$.

\begin{main}\label{Main_thm_c_1}
Conjecture \ref{conjecture_folklore} holds for the mirror Landau-Ginzburg model in Theorem \ref{Main_theorem_thesis}. Namely, any critical value of $W^\vee$ must be an eigenvalue of the quantum product by the first Chern class $c_1$ of $X$.
\end{main}

Roughly, by allowing the wall-crossing invariance, we can enlarge the domain of the superpotential and so potentially cover more eigenvalues. But, we cannot conversely conclude that all the eigenvalues of $c_1$ are realized by a critical point of the superpotential.
Note that the $\Lambda$-coefficients in $W^\vee$ include the data of symplectic areas and can tell which chambers the critical points are situated in.
In particular, the locations of critical points depend on the symplectic form $\omega$.

\subsection{A new operator for Hochschild cohomology and ud-homotopy theory}
\label{ss_new_operator}

	The obstructions in Lagrangian Floer theory cause issues for both analysis and algebra.
	The analysis part should be of most importance and difficulty, concerning the complicated transversality issues of moduli spaces of pseudo-holomorphic curves. This is studied and developed by \cite{FOOOBookOne,FOOOBookTwo,Kuranishi_book} and so on.
	In comparison, the algebra part should be minor.
	But, in our biased opinion, the intricacies and subtleties of the algebra part are somewhat underestimated.

	The open string Floer theoretic invariants often depend on choices. Thus, we often must talk about certain equivalence relations, depending on which the information we miss can vary.
	For example, in the literature like \cite{FOOOBookOne}, we usually talk about $A_\infty$ algebras up to homotopy equivalence.
	It is fine in most cases, but it may be too coarse in case we need more information (e.g. non-archimedean analytic structure).
	The homotopy equivalence relation cannot distinguish between a minimal model $A_\infty$ algebra and its original.
	This is why the usual Maurer-Cartan picture \cite{FuBerkeley,FuCounting,Tu} cannot be enough for a local-to-global gluing for the family Floer mirror analytic structure, as explained before.

	In our specific situation here for the folklore Conjecture \ref{conjecture_folklore}, we must carefully distinguish between a minimal model $A_\infty$ algebra, denoted by $(H^*(L),\m)$, and its original, denoted by $(\Omega^*(L),\check \m)$.
	This is because the classic ideas for the folklore conjecture can only achieve connections between the quantum cohomology and the cochain-level Hochschild cohomology $HH(\OL, \check \m)$.
	Meanwhile, the family Floer Landau-Ginzburg superpotential comes from the minimal model $A_\infty$ algebra $(\HL, \m)$.
	
	There are some new challenges.
	For instance, the Hochschild cohomology is just not functorial, and in pure homological algebra, it is basically impossible to make any meaningful connections between the Hochschild cohomologies of a minimal model $A_\infty$ algebra and its original.
	Fortunately, in our specific geometric situation (with Assumption \ref{assumption_weak_MC} or \ref{assumption_weak_MC_anti_involution}), we magically make an effective connection by the following operation (see (\ref{map_Theta_eq}))
	\[
	\Theta: \varphi \mapsto (\mi^{-1}\{ \varphi\} ) \diamond \mi = \sum (-1)^* \  \mi^{-1} \big( \mi ,\dots, \mi, \varphi(\mi,\dots, \mi), \mi, \dots,\mi  \big) 
	\]
	where $\diamond$ is the composition (see \cite{Sheridan15}), the $\{\}$ is the celebrated \textit{brace operation} (see \cite{Getzler_1993cartan}), and the $\mi^{-1}$ is the \textbf{\textit{ud-}}homotopy inverse of the natural $A_\infty$ homotopy equivalence $\mi:\m \to \check \m$ in the course of homological perturbation.
A difficult part of the paper is to find the above operation $\Theta$ and establish the corresponding computations. As far as we know, such $\Theta$ does not appear in any literature and seems to be exclusive for our ud-homotopy theory in \cite{Yuan_I_FamilyFloer}.

Now, the ud-homotopy theory shows its necessity again.
Moreover, it is interesting to note that the family Floer mirror analytic gluing in \cite{Yuan_I_FamilyFloer} does not exploit the full information of ud-homotopy relations, while the techniques in the present paper (e.g. the self Floer cohomology for $\UD$ in \S \ref{s_self_HF}) use more information of ud-homotopy relations.

\subsection{Outline and major difficulties}
\label{ss_outline_may_difficulties}

If the Maslov indices $\ge 2$, the approach to Conjecture \ref{conjecture_folklore} is well-known to experts for a long time: see \cite[\S 6]{AuTDual} or \cite[\S 12.1]{Ritter_Smith} in toric cases and \cite[\S 23]{FOOOSpectral} in general; see also \cite[Definition 2.3]{Sheridan16}.
Roughly speaking, a key idea concerns a \textit{unital ring homomorphism}
\begin{equation}
	\label{classical_idea_eq}
	QH^*(X) \to HH^*(CF(L,L)) \to HF(L,L)
\end{equation}
from the quantum cohomology to the self Floer cohomology through the Hochschild cohomology.
But, when nontrivial Maslov-0 disks are allowed, we cannot obtain such a unital ring homomorphism as usual.
It is even necessary to find a suitable way to generalize the definitions of $HF(L,L)$ and $HH^*(CF(L,L))$ in (\ref{classical_idea_eq}).
The major difficulties are outlined as follows. 

%
%

\begin{note}[\textbf{i}]
	We want the self Floer cohomology $HF(L,L)$ to be defined in a `wall-crossing invariant' way.
\end{note}

This is especially crucial for the global result, and we note that $L$ is not necessarily a torus here.

First, thanks to the Fukaya-Oh-Ohta-Ono's foundation works, we can very generally define the self Floer cohomology $HF((L,b),(L,b))$ by choosing any odd-degree weak bounding cochain $b$.
However, it may be too general for our purpose. We want to work with a more restrictive situation, highlighting the semipositive condition and the wall-crossing phenomenon.
By utilizing the related ideas about the category $\UD$, we define a modified self Floer cohomology $\HF(L,\m)$ for any cohomology-level $A_\infty$ algebra $(\HL,\m)$ in $\UD$ (\S \ref{s_self_HF}).
It is a unital $\Lambda$-algebra, and it does not refer to a specific weak bounding cochain. Instead, whenever $\mathbf y\in H^1(L; U_\Lambda)$ admits a lift of weak  bounding cochain $b$ in $H^1(L;\Lambda_0)$ via $\exp: \Lambda_0\to U_\Lambda$, we may define $\HF(L,\m, \mathbf y)$ in almost the same way as the above $HF((L,b),(L,b))$; moreover, it may be viewed as the restriction of $\HF(L,\m)$ at the point $\mathbf y$, accompanied by a natural unital ring homomorphism $\mathcal E_{\mathbf y}: \HF(L,\m) \to \HF(L,\m,\mathbf y)$.

Our modified Floer cohomology are wall-crossing invariant in the following sense. Recall that an analytic transition map $\phi$ is determined by an $A_\infty$ homotopy equivalence $\mC:\m\to \m'$ in $\UD$. Then, the $\phi$ and $\mC$ can induce an isomorphism $\HF(L,\m)\cong \HF(L,\m')$ as well.
Further, we have the following: \textit{When the de Rham cohomology ring $H^*(L)$ is generated by $H^1(L)$ (e.g. $L$ is topologically a torus), the $\HF(L,\m,\mathbf y)$ is non-vanishing if and only if $\mathbf y$ is a critical point of $W$, where $W$ is the local superpotential defined by $\m$ as before.}
This result also has the wall-crossing invariance.
Set $\mathbf y'=\phi(\mathbf y)$ and $W'=\phi^*W$; not only $\HF(L,\m, \mathbf y)\cong \HF(L,\m', \mathbf y')$ but also we know that if $\mathbf y$ is a critical point of $W$, then $\mathbf y'$ is a one of $W'$ and vice versa.
Finally, we remark that the non-archimedean analysis plays an important role for this result:
Let $\hbar>0$ be a lower bound of the energies of nontrivial holomorphic disks, then we should first show the result modulo $T^{\hbar}$ and inductively refine it to $T^{k\hbar}$ for $k\ge 1$. The limit for the adic topology in the Novikov field $\Lambda$ as $k\to \infty$ works in the end.

\begin{note}[\textbf{ii}]
	The definition of $HH^*(CF(L,L))$ in the middle of (\ref{classical_idea_eq}) becomes problematic and ambiguous.
\end{note}

This issue is actually more difficult than (\textbf{i}).
The moduli spaces of holomorphic disks first only give the {chain-level} $A_\infty$ algebra $(\OL,\check \m)$. In contrast, the self Floer cohomology and the superpotential all resort to the {cohomology-level} $(\HL,\m)$, i.e. the minimal model $A_\infty$ algebra.
Previously, when the Maslov indices were $\ge 2$, the homological perturbation largely degenerates so that the difference between $(\OL,\check \m)$ and $(\HL,\m)$ is mild enough to achieve (\ref{classical_idea_eq}).
Unfortunately, when the Maslov-0 disks are involved, such difference cannot be ignored, and the subsequent discrepancy between the two Hochschild cohomology rings $HH^*(\OL,\check \m)$ and $HH^*(\HL,\m)$ becomes highly nontrivial.
For instance, there is basically no chance to make a ring homomorphism between the two of them, as the Hochschild cohomology is neither covariant nor contravariant.
This obstacle leads to a peculiar mismatch between the algebra and geometry.
The moduli spaces with interior markings can produce a ring homomorphism from the quantum cohomology $QH^*(X)$ only to the chain-level Hochschild cohomology $HH^*(\OL,\check \m)$ but not to the cohomology-level one $HH^*(\HL,\m)$.

In a word, the solution can be realized by the relevant technologies about the category $\UD$ introduced both in this paper and in \cite{Yuan_I_FamilyFloer}. Not only the homological algebra but also the non-archimedean analysis and the geometric semipositive condition in $\UD$ are used in a significant way.
Briefly,
we will carefully design a map $\Theta$ between the underlying two Hochschild cochain complexes such that: although it does not precisely induce a ring homomorphism, the gap to become so is controlled by the ud-homotopy relations in $\UD$.
Next, applying a natural projection map $\mathbb P$, this gap is then controlled by the weak Maurer-Cartan formal power series and eliminated under Assumption \ref{assumption_weak_MC}.
We can ultimately show that the composition map $\mathbb P\circ \Theta$ can induce an honest unital ring homomorphism $\Phi$ (\S \ref{s_closed_open}).
\[
\xymatrix{
	(\text{cochain-level})& QH^*(X) \ar[r]  & HH^*(\OL,\check\m) \ar@{.>}[dd]_{ }  \ar[ddrr]^{\Phi:=[\mathbb P\circ \Theta]} \\
	& &	& \\
	(\text{cohomology-level}) &  & HH^*(\HL,\m) \ar@{.>}[rr]_{  } & & \HF (L,\m) \ar[rr]_{\mathcal E_{\mathbf y}} & & \HF(L,\m,\mathbf y)
}
\]

Another technical point is that we need the so-called \textit{reduced Hochschild cohomologies} (c.f. \cite[\S 25]{FOOOSpectral}). Roughly speaking, it is designed to keep track of units (\S \ref{s_reduced_Hochschild}).
Finally, the composition gives a \textit{unital} ring homomorphism $\mathbb {CO}_{\mathbf y}: QH^*(X)\to \HF(L,\m,\mathbf y)$, despite of all the troubles caused by Maslov-0 disks (\S \ref{s_closed_open}).

\section{Preliminaries and reviews}\label{s_preliminaries_reviews}

The Floer theory usually works over the \textit{Novikov field} $\Lambda= \mathbb C(( T^{\mathbb R}))$.
It has a valuation map $\val: \Lambda \to \mathbb R\cup \{\infty\}$, sending a nonzero series $\sum a_i T^{\lambda_i}$ to the smallest $\lambda_i$ with $a_i\neq 0$ and sending the zero series to $\infty$.
The Novikov ring is the valuation ring $\Lambda_0:=\val^{-1}[0,\infty]$ whose maximal ideal is $\Lambda_+ := \val^{-1} (0,\infty]$.
The multiplicative group of units is $U_\Lambda:= \val^{-1}(0)$.
Note that $U_\Lambda= \mathbb C^*\oplus \Lambda_+$ and $\Lambda_0 =\mathbb C\oplus \Lambda_+$. The standard isomorphism $\mathbb C^*\cong \mathbb C/ 2\pi i \mathbb Z$ naturally extends to $U_\Lambda\cong \Lambda_0/ 2\pi i \mathbb Z$. In particular, for any $y\in U_\Lambda$, there exists some $x\in \Lambda_0$ with $y=\exp(x)$.

Let $\Lambda^* =\Lambda\setminus\{0\}$, and we consider the tropicalization map
\begin{equation}
	\label{trop_eq}
	\trop \equiv \val^{\times n} : (\Lambda^*)^n \to \mathbb R^n
\end{equation}
The total space $(\Lambda^*)^n$ should be understood as not just an algebraic variety over $\Lambda$ but also a non-archimedean analytic space.
Roughly, a non-archimedean analytic space is a locally ringed space whose structure sheaf is locally given by affinoid algebras.
Affinoid algebras are quotients of the Tate algebras.
The spectrum of an affinoid algebra is called an affinoid space or affinoid domain.
Let $\Delta$ be a rational polyhedron in $\mathbb R^n$, then the preimage $\trop^{-1}(\Delta)$ is an affinoid (sub)domain in $(\Lambda^*)^n$.
We refer to the appendix in \cite{Yuan_I_FamilyFloer} or the standard textbook \cite{BoschBook,Berkovich_2012spectral} for further readings.

Finally, we note a useful observation in \cite{Yuan_I_FamilyFloer} or \cite{Yuan_unobs} as follows:

\begin{lem}
	\label{val=0-lem}
	$f\in \Lambda[[Y_1^\pm,\dots, Y_n^\pm]]$ vanishes identically if and only if $f|_{U_\Lambda^n}\equiv 0$.
\end{lem}

\begin{proof}
	Let $f=\sum_{\nu\in\mathbb Z^m} c_\nu Y^\nu$.
	Note that $\mathrm{val}(c_\nu)\to\infty$ i.e. $|c_\nu|\to 0$. 
	Arguing by contraction, suppose the sequence $|c_\nu|$ was nonzero, say, it had a maximal value $|c_{\nu_0}|=1$ for some $\nu_0\in\mathbb Z^n$. May further assume $c_{\nu_0}=1$. 
	Then, $|c_\nu|\leqslant 1$ for all $\nu$, so $f\in \Lambda_0[[Y^\pm]]$.
	Modulo the ideal of elements with norm $<1$, we get a power series $\bar f=\sum_{\nu} \bar c_\nu Y^\nu$ over the residue field $\mathbb C$.
	As $c_\nu \to 0$, we have $|c_\nu|<1$ and $\bar c_\nu=0$ for $\nu\gg 1$.
	Hence, this $\bar f$ is just a Laurent polynomial over $\mathbb C$ with $\bar c_{\nu_0}=1$. Meanwhile, the condition also tells that $\bar f(\mathbf {\bar y})$ vanishes for all $\mathbf {\bar y}\in (\mathbb C^*)^n$; thus, $\bar f$ must be identically zero. This is a contradiction.
\end{proof}

\subsection{Gapped $A_\infty$ algebras and the category $\UD$}
\label{ss_gapped_UD}
We revisit the notations and terminologies from \cite[\S 2]{Yuan_I_FamilyFloer}, along with introducing some new concepts where necessary. While there are slight differences in the formulations, similar notions have been previously introduced in \cite{FOOOBookOne} and \cite{FuCyclic}.

\subsubsection{Gappedness}
\label{sss_label_group}
A \textit{label group} is a triple $(\G,E,\mu)$ consisting of an abelian group $\G$ and two group homomorphisms $E:\G\to \mathbb R$ and $\mu: \G\to 2 \mathbb Z$.
For instance, let $X$ and $L$ denote a symplectic manifold and a Lagrangian submanifold, then we may take $\G=\pi_2(X,L)$ (or more precisely, $\G$ is the image of Hurewicz map $\pi_2(X,L)\to H_2(X,L)$); the $E$ and $\mu$ are the energy and Maslov index.

Assume $C,  C'$ are graded $\mathbb R$-vector spaces. Given $k\in\mathbb N$ and $\beta\in\G$, we use the notation
$
\CC_{k,\beta}(  C,  C')
$
to denote a copy of $\Hom(  C^{\otimes k},  C')$ where $\beta$ is just an extra index.
Then, consider the space 
\begin{equation}
	\label{CC_eq}
	\CC( C,  C') \subseteq  \prod_{k\in\mathbb N} \prod_{\beta\in\G} \CC_{k,\beta}(  C,  C')
\end{equation}
consisting of the operator systems $\mathfrak t=(\mathfrak t_{k,\beta})$ satisfying the following \textit{gappedness conditions}:
\begin{itemize}
	\itemsep 2pt
	\item [(a)] $\mathfrak t_{0,0}=0$; 
	\item [(b)] if $E(\beta)<0$ or $E(\beta)=0$, $\beta\neq 0$, then $\mathfrak t_\beta:=(\mathfrak t_{k,\beta})_{k\in\mathbb N}$ vanishes identically; 
	\item [(c)] for any $E_0>0$, there are only finitely many $\beta$ such that $\mathfrak t_{\beta}\neq 0$ and $E(\beta)\le E_0$.
\end{itemize}

\vspace{0.5em}

Here a collection of multilinear maps is called an operator system $\mathfrak t$. If it is further contained in $\CC(C,C')$, namely, it satisfies the above three conditions, the $\mathfrak t$ is called ($\G$-)\textit{gapped}. Unless we further specify, everything is gapped from now on.
We often abbreviate it to $\CC$.

\subsubsection{Signs and degrees}
\label{sss_sign_degree}

Define the \textit{shifted degree} $\deg' x:=\deg x-1$ for $x\in C$.
It is convenient to set
\begin{equation}
	\label{sign_operator_eq}
	x^\#=\id^\#(x)=\id_\#(x)=(-1)^{\deg x -1 } x  =(-1)^{\deg' x} x
\end{equation}
Given $s\in\mathbb N$ and a multi-linear map $\phi$, we put 
\[
\phi^{\# s}=\phi\circ \id^{\# s}
\]
where $\id^{\# s}$ denotes the $s$-iteration $\id^\#\circ \cdots \circ \id^\#$ of $\id^\#$.
If $s=1$, we set $\phi^\#=\phi^{\# 1}$.
Let $\deg \phi$ be the usual degree as a homogeneously-graded $k$-multilinear operator among graded vector spaces. Then, the shifted degree is $\deg' \phi=\deg \phi+k-1$. Note that 
$
 \phi^\# = (-1)^{\deg' \phi} \ \id^\#\circ \phi$.

We define the \textit{composition} (see \cite[Definition 2.39]{Sheridan15}) and the \textit{Gerstenhaber product} (see \cite{Gerstenhaber_1963cohomology})
\begin{equation}
	\label{composition_Gerstenhaber_eq}
	\begin{aligned}
		(\g\diamond\f)_{k,\beta}=
	\sum_{\ell \ge 1}
	\sum_{k_1+\dots+k_\ell=k}
	\sum_{\beta_0+ \beta_1+\cdots +\beta_\ell=\beta}
	\g_{\ell,\beta_0} 
	\circ ( \f_{k_1,\beta_1}\otimes \cdots \otimes \f_{k_\ell,\beta_\ell} ) \\
	(\g\{ \h\})_{k,\beta} = \sum_{\lambda+\mu+\nu=k}\sum_{\beta'+\beta''=\beta} \g_{\lambda+\mu+1,\beta'} \circ (\id_{\#\deg'\h}^\lambda \otimes \h_{\nu,\beta''}\otimes \id^\mu)
	\end{aligned}
\end{equation}

The gappedness ensures that they are finite sums. 
For instance, if $\f_{0,0}$ could be non-zero, the sum in $\g\diamond \f$ would have infinite terms and was not well-defined.

\subsubsection{$A_\infty$ structures}

By a \textit{$\G$-gapped $A_\infty$ algebra} $(C,\m)$, we mean an operator system $\m=(\m_{k,\beta})$ in $\CC(C,C)$ such that $\deg \m_{k,\beta}=2-\mu(\beta)-k$ and $\m\{ \m\}=0$, namely
\[
\sum_{\beta'+\beta''=\beta} \sum_{\lambda+\mu+\nu=k} \m_{\lambda+\mu+1,\beta'}   (\id_\#^\lambda \otimes \m_{\nu,\beta''}\otimes \id^\mu)=0
\]
A \textit{$\G$-gapped $A_\infty$ homomorphism} from $(C',\m')$ to $(C,\m)$ is an operator system $\f=(\f_{k,\beta})\in \CC(C',C)$ such that $\deg \f_{k,\beta}=1-\mu(\beta)-k$ and $\m\diamond \f=\f\{\m'\}$, namely
\begin{align*}
	\sum_{\ell \ge 1} \sum_{\substack{
			0=j_0\le\cdots\le j_\ell=k
			\\
			\beta_0+\beta_1+\cdots +\beta_\ell=\beta
	}}
	\m_{\ell,\beta_0}  (\f_{j_1-j_0,\beta_1}\otimes \cdots \otimes \f_{j_\ell-j_{\ell-1},\beta_\ell}) = \sum_{\substack{\lambda+\mu+\nu=k \\ \beta'+\beta''=\beta}} \f_{\lambda+\mu+1,\beta'}  (\id_\#^\lambda\otimes \m'_{\nu,\beta''} \otimes \id^\mu)
\end{align*}
As $\mu(\beta)$ is always an even integer, we know $\deg' \m=1$ and $\deg' \f=0$ in $\mathbb Z_2$.
Restricting the $\m$ or $\f$ to the energy-zero part $\prod_k\CC_{k,0}$ (i.e. $\beta=0$), what we obtain is called the \textit{reduction} $(C,\bar {\m})$ (resp. $\bar{\f}:\bar{\m}'\to \bar{\m}$) of $\m$ (resp. $\f$).
In the energy-zero part, we always have $\m_{1,0}\circ \m_{1,0}=0$ and $\m_{1,0}\circ \f_{1,0}= \f_{1,0}\circ \m'_{1,0}$. Hence, an $A_\infty$ algebra $(C,\m)$ naturally induces a cochain complex $(C,\m_{1,0})$ and an $A_\infty$ homomorphism $\f$ also induces a cochain map $\f_{1,0}$.
Applying the homological perturbation to $\m$ yields a new $A_\infty$ algebra $\hat\m$ with $\hat\m_{1,0}=0$, often called the minimal model.

Let $P$ be the convex hull of finite points in a Euclidean space such like a $d$-simplex $\Delta^d$. Denote by $C^\infty(P,C)$ the set of all smooth maps from $P$ to $C$, and we define
\[
C_P:= \Omega^*(P)\otimes_{C^\infty(P)} C^\infty(P,C)
\]
It naturally comes with the evaluation map 
\[
\eval^s:C_P\to C
\]
for any $s\in P$ and the inclusion map 
\[
\incl: C\to C_P
\]
Alternatively, it can be regarded as the space of $C$-valued differential forms on $P$.
For instance, when $C=\Omega^*(L)$ for a manifold $L$, we have a natural identification $C_P\cong \Omega^*(P\times L)$ via $\eta\otimes x\xleftrightarrow{} \eta\wedge x$. Also, the $\eval^s$ and $\incl$ are the pullbacks for the inclusion $L\xhookrightarrow{} \{s\}\times L$ and the projection $P\times L \to L$.

A \textit{$P$-pseudo-isotopy} (of gapped $A_\infty$ algebras) on $C$ is simply a special sort of a gapped $A_\infty$ algebra structure $\M$ on $C_P$ such that the $\M$ is $\Omega^*(P)$-linear (up to sign) and the $\M_{1,0}$ is the canonical differential on $C_P$. In the special case $C=\OL$, the $\M_{1,0}$ coincides with the exterior derivative $d_{P\times L}$. 
The notion of pseudo-isotopy is introduced in \cite{FuCyclic}.
See also \cite{Yuan_I_FamilyFloer} for the details.

\subsubsection{Unitality and divisor axiom}
\label{sss_unitality_divisor_axiom}

An arbitrary operator system $\mathfrak t\in \CC$ is called \textit{cyclically unital} if, for any degree-zero element $\e$ and any $(k,\beta)\neq (0,0)$, we have
\[
\textstyle
\CU[\mathfrak t]_{k,\beta}(\e; x_1,\dots, x_k) := \sum_{i=1}^{k+1}\mathfrak t_{k+1,\beta} (x_1^\#,\dots, x_{i-1}^\#, \e, x_i, \dots, x_k) =0
\]
On the other hand, a gapped $A_\infty$ algebra $(C,\m)$ is called (strictly) \textit{unital} if there exists a degree-zero element $\one$, called a \textit{unit}, such that

\begin{enumerate}
	\itemsep 2pt
	\item $\m_{1,0}(\one)=0$
	\item $\m_{2,0}(\one, x)=(-1)^{\deg x}\m_{2,0}(x,\one)=x$
	\item $\m_{k,\beta} (\dots, \one,\dots)=0$ for $(k,\beta)\neq (1,0), (2,0)$.
\end{enumerate}

Assume $(C', \m')$ and $(C, \m)$ have the units $\one'$ and $\one$; then, a gapped $A_\infty$ homomorphism $\f: \m' \to\m$ is called \textit{unital} if $\f_{1,0}(\one')=\one$ and \[
\f_{k,\beta}(\dots,\one',\dots)=0
\]
for $(k,\beta)\neq (1,0)$.

A gapped $A_\infty$ algebra $(C,\m)$ is called a \textit{quantum correction to de Rham complex}, or in abbreviation, a \textit{q.c.dR}, if $C$ is $\Omega^*(M)$ for some manifold $M$ so that $\m_{1,0}=d_M$, $\m_{2,0}(x_1,x_2)=(-1)^{\deg x_1} x_1\wedge x_2$, and $\m_{k,0}=0$ for $k\ge 3$.
Geometrically, an $A_\infty$ algebra obtained by the moduli spaces or any of its minimal model can satisfy the q.c.dR conditions.

From now on, we always assume $\G=\pi_2(X,L)$ and $C=\HL_P$ or $\OL_P$ for some $L$ and $P$. In either of two cases, we have a natural differential $d$ on $C$ for which there is a well-defined cap product $\partial\beta\cap \cdot $ on \[
Z^1(C):=\{ b\in C\mid d b=0, \deg b=1\}
\]
for any $\beta\in \pi_2(X,L)$.
Now, we say an operator system $\mathfrak t=(\mathfrak t_{k,\beta})$ in $\CC(C,C')$ satisfies the \textit{divisor axiom} if for any $b\in Z^1(C)$ and $(k,\beta)\neq (0,0)$,
\[\textstyle
\DA [\mathfrak t]_{k,\beta}(b;x_1,\dots, x_k):= \sum_{i=1}^{k+1} \mathfrak t_{k+1,\beta} (x_1,\dots, x_{i-1}, b,x_i, \dots, x_k)
=
\partial\beta\cap b\cdot \mathfrak t_{k,\beta}(x_1,\dots, x_k)
\]

\subsubsection{The category $\UD$}
\label{sss_UD}

Given a symplectic manifold $(X,\omega)$, let $L$ be an oriented spin Lagrangian submanifold. We remark that $L$ is not necessarily a Lagrangian torus.
Set $\G=\pi_2(X,L)$.
For simplicity, we use a uniform notation $\one$ to denote all various constant-one functions in either $\HL_P$ or $\OL_P$.
In \cite[\S 2]{Yuan_I_FamilyFloer}, we introduced a category
$\UD$ as follows:

\begin{enumerate}
	\item[(I)] An object $\m$ in $\UD$ is a $\G$-gapped $A_\infty$ algebra with the following properties:
	\begin{itemize}
		\item[(I-0)] it extends $H^*(L)_P$ or $\Omega^*(L)_P$;
		\item[(I-1)] the constant-one function is a (strict) unit;
		\item[(I-2)] the cyclical unitality;
		\item[(I-3)] the divisor axiom;
		\item[(I-4)] it is a $P$-pseudo-isotopy;
		\item [(I-5)] every $\beta\in\pi_2(X,L)$ in the set $\mathsf G_\m:=\{\beta \mid \m_{\beta}\neq 0\}$ satisfies $\mu(\beta)\ge 0$.
	\end{itemize}
	\item[(II)] A morphism $\f$ in ${\UD}$ is a $\G$-gapped $A_\infty$ homomorphism with the following properties:
	\begin{itemize}
		\item[(II-1)] the strict unitalities with the constant-one functions as units;
		\item[(II-2)] the cyclically unitality;
		\item[(II-3)] the divisor axiom;
		\item[(II-4)] $\partial \beta \cap \f_{1,0} (b) = \partial \beta \cap b$ for any divisor input $b\in Z^1(C)$;
		\item[(II-5)] every $\beta\in\pi_2(X,L)$ in the set $\mathsf G_\f :=\{\beta\in\G\mid \f_{\beta}\neq 0\}$ satisfies $\mu(\beta)\ge 0$.	
	\end{itemize}
\end{enumerate}

We also call (I-5) and (II-5) the \textbf{semipositive conditions}.
We write $\Obj\UD$ and $\Mor\UD$ for the collections of objects and morphisms in $\UD$ respectively.

\subsubsection{Homotopy theory and Whitehead theorem}
\label{sss_whitehead_homotopy_theory_UD}
For $\UD$, there is the notion of \textbf{ud-homotopy}:
We call 
\[
\f_0,\f_1\in \Hom_\UD((C',\m'), (C,\m) )
\]
are \textit{ud-homotopic}
if there is 
\[
\F\in \Hom_\UD ( (C',\m'), (C_\oi,\M^\tri))
\]
such that 
\[
\eval^0  \F=\f_0 \quad , \quad \eval^1  \F =\f_1
\]
where $\M^\tri$ is the trivial pseudo-isotopy about $\m$ and can be proved to be an object in $\UD$ as well. In this case, we write $\f_0\simud \f_1$. Moreover, $\f_0\simud \f_1$ if and only if there exists $(\f_s)$ and $(\h_s)$ in $\CC(C', C)$ such that: (see \cite[\S 2]{Yuan_I_FamilyFloer})

\begin{enumerate}[label=(\alph*)]
	\itemsep 2pt
	\item Every $\f_s\in \Mor \UD$;
	\item $ \frac{d}{ds} \circ \f_s = \sum \h_s\circ (\id_\#^\bullet\otimes \m'\otimes \id^\bullet) + \sum \m\circ (\f_s^\#\otimes \cdots \otimes \f_s^\#\otimes  \h_s\otimes \f_s\otimes \cdots\otimes \f_s)$;
	\item The $\h_s$ satisfies the divisor axiom, the cyclical unitality, and $(\h_s)_{k,\beta}(\cdots \one \cdots )=0$ for all $(k,\beta)$;
	\item $\deg (\h_s)_{k,\beta}= -k-\mu(\beta)$. For every $\beta$ with $\h_\beta\neq 0$, we have $\mu(\beta)\ge 0$.
\end{enumerate}

In the context of $\UD$, we also have the Whitehead theorem \cite[\S 3]{Yuan_I_FamilyFloer}:
\begin{thm}[Whitehead]
	\label{whitehead_thm}
	Fix $\f\in \Hom_\UD ((C',\m'), (C,\m))$ such that $\f_{1,0}$ is a quasi-isomorphism of cochain complexes. Then, there exists $\g\in \Hom_\UD ( (C,\m) , (C',\m'))$, unique up to ud-homotopy, such that $\g\diamond \f\simud \id_{C'}$ and $\f\diamond \g\simud \id_C$. We call $\g$ a \textbf{ud-homotopy inverse} of $\f$.
\end{thm}

\subsection{An observation about weak Maurer-Cartan equations}
\label{ss_weak_MC}

The semipositive condition (I-5) says that for $\m\in \Obj\UD$, we have $\m_\beta\neq 0$ only if $\mu(\beta)\ge 0$. Hence, $\deg \m_{0,\beta}= 2-\mu(\beta) \le 2$.
If the $\m$ is defined on $H^*(L)$, then we may write as in \cite{Yuan_I_FamilyFloer}:
\begin{equation}
	\label{W_Q_eq}
	W\cdot \one + Q:=\sum_{\mu(\beta)=2} T^{E(\beta)}Y^{\partial\beta} \m_{0,\beta}+\sum_{\mu(\beta)=0} T^{E(\beta)}Y^{\partial\beta} \m_{0,\beta}
\end{equation}

\begin{rmk}
	\label{W_minimal_model_rmk}
	The moduli space of holomorphic disks bounding a Lagrangian submanifold $L$ gives an $A_\infty$ algebra $\check\m$ on $\OL$ at first. But, we need to use its minimal model $(\HL, \m)$.
	Otherwise, the first summation in (\ref{W_Q_eq}) would live in $\Lambda[[\pi_1(L)]]\hat\otimes \Omega^0(L)$ and may not be in the form of $W\cdot \one$. 
\end{rmk}

Let $\ia$ denote the ideal in $\Lambda[[\pi_1(L)]]$ generated by the components of $Q$.
The ideal $\ia$ vanishes under Assumption \ref{assumption_weak_MC}, but let us keep it for a moment to highlight the ideas.
Basically, one may view (\ref{W_Q_eq}) as the (weak) Maurer-Cartan equations with certain differences.
We review the following lemma whose proof is also implicitly given in \cite{Yuan_I_FamilyFloer}:

\begin{lem}
	\label{weakMC_observation_lem}
	If $\f_0\simud \f_1$ from $(\HL,\m')$ to $(\HL,\m)$ and $(\f_0)_{1,0}=(\f_1)_{1,0}$, then
	\[
	\textstyle \sum_\beta T^{E(\beta)} Y^{\partial\beta} ((\f_1)_{0,\beta}-(\f_0)_{0,\beta})\in \ia
	\]
	In particular, under Assumption \ref{assumption_weak_MC}, we have
	\[
	\textstyle \sum_\beta T^{E(\beta)} Y^{\partial\beta}(\f_0)_{0,\beta}
	=
	\textstyle \sum_\beta T^{E(\beta)} Y^{\partial\beta}(\f_1)_{0,\beta}
	\]
\end{lem}

\begin{proof}
	Let $(\f_s)$ and $(\h_s)$ be as in \S \ref{sss_whitehead_homotopy_theory_UD} with the conditions (a) (b) (c) (d) therein.
	For a basis of $\pi_1(L)$, we have $\Lambda[[\pi_1(L)]]\cong \Lambda[[ Y_1^\pm, \dots, Y_m^\pm]]$.
	For a basis of $H^2(L)$, we write $Q=\sum_\eta Q_\eta\cdot \eta$, and then by definition, the ideal $\ia$ is generated by all the $Q_\eta$.
	By Lemma \ref{val=0-lem}, it suffices to restrict 
	\[
	S:= \sum T^{E(\beta)} Y^{\partial\beta} ((\f_1)_{0,\beta}-(\f_0)_{0,\beta})
	\]
	to an arbitrary point $\mathbf y=(y_1,\dots, y_m)$ in $U_\Lambda^m\cong H^1(L;U_\Lambda)$.
	We can write $y_i=e^{x_i}$ for $x_i\in\Lambda_0$. We set $b=(x_1,\dots, x_m)$ for the basis.
	Applying the divisor axiom of $\f_0,\f_1$ and the condition (b) implies that
	\begin{align*}
		S(y_1,\dots, y_m) &= \sum_{k,\beta} T^{E(\beta)}  \big((\f_1)_{k,\beta} (b,\dots, b) -(\f_0)_{k,\beta}(b,\dots, b) \big) \\
		&= \sum T^{E(\beta)}  \int_0^1 ds \cdot  \frac{d}{ds} \circ (\f_s)_{k,\beta}(b,\dots,b)\\
		&=  \sum T^{E(\beta_1)} \int_0^1 ds\cdot  (\h_s)_{\lambda+\mu+1,\beta_1} (b,\dots,b , T^{E(\beta_2)} \m'_{\nu,\beta_2} (b,\dots,b) , b ,\dots b ) \\
		&+ \sum T^{E(\beta)} \int_0^1ds\cdot \m \big( (\f_s)(b,\dots,b) ,\dots, (\h_s)_{\ell,\beta_0} (b,\dots, b) ,\dots, (\f_s)(b,\dots,b)\big)
	\end{align*}
	Since $\deg (\h_s)_{\ell,\beta_0}(b,\dots,b)=-\mu(\beta_0)$, the semipositive condition and the cyclical unitality of $\m$ deduce that the second summation vanishes. Hence, for the operator $\mathfrak H:=\int_0^1ds \cdot \h_s$, it follows from the condition (c) that
	\[
	S(y_1,\dots, y_m)= \sum_\beta T^{E(\beta)} \mathbf y^{\partial\beta} \mathfrak H_{1,\beta} \big( W(\mathbf y) \cdot \one + \sum_\eta Q_{\eta}(\mathbf y) \cdot  \eta \big) = \sum_\eta Q_\eta(\mathbf y) \cdot \sum_\beta T^{E(\beta)} \mathbf y^{\partial\beta } \mathfrak H_{1,\beta}(\eta)
	\]
	Now that the above equation holds for any $\mathbf y=(y_i)\in U_\Lambda^m$, it also holds identically by Lemma \ref{val=0-lem}.
\end{proof}

\subsection{Wall-crossing from family Floer viewpoint: Review}
\label{ss_wall_cross_review}

Here we review the wall-crossing aspect of the family Floer framework in \cite{Yuan_I_FamilyFloer}. In this section, we provide a summary of the basic points and refer to \cite{Yuan_I_FamilyFloer} for the details.

\subsubsection{$A_\infty$ structures}
\label{sss_moduli_space_A_infinity}

Let $J$ be an $\omega$-tame almost complex structure in $X$. 
Fix $k\in\mathbb N$ and $\beta\in\pi_2(X,L)$ with $(k,\beta)\neq (0,0), (1,0)$. We consider the moduli space
$\mathcal M_{k+1,\beta}(J,L)$ of the equivalence classes of $(k+1)$-boundary-marked $J$-holomorphic genus-zero stable maps with one boundary component in $L$ in the class $\beta$.
Let $\mathbb M(J)$ denote the collection of the moduli spaces $\mathcal M_{k+1,\beta}( J, L)$ for all $(k,\beta)$.
We call $\mathbb M( J )$ a \textit{moduli system}. Using the virtual techniques \cite{Kuranishi_book}, it gives rise to a chain-level $A_\infty$ algebra denoted by $(\OL, \check \m^{J,L})$.
Suppose $J$ is semipositive in the sense that there does not exist any $J$-holomorphic disk in $\pi_2(X,L)$ with negative Maslov index. Then, one can check that $\check \m=\check \m^{J,L}$ is an object in $\UD=\UD(L)$.

The family Floer theory only works with the cohomology-level $A_\infty$ algebras; see e.g. Remark \ref{W_minimal_model_rmk}. We need to perform the homological perturbation to the above $\check \m^{J,L}$.
Let $g$ be a metric.
By the Hodge decomposition, one can find two cochain maps $i(g): \HL\to \OL$, $\pi(g):\OL\to \HL$ of degree zero and a map $G(g):\OL\to \OL$ of degree $-1$ such that
$i(g)\circ \pi(g)-\id= d_L\circ G+G\circ d_L$, $\pi(g)\circ i(g)=\id$, $G(g)\circ G(g)=0$, $G(g)\circ i(g)=0$, and $\pi(g)\circ G(g)=0$.
Now, we call $\con(g):=(i(g),\pi(g),G(g))$ the \textit{harmonic contraction}.
Applying the homological perturbation with this
$\con(g)$, the $A_\infty$ algebra $\check \m^{J,L}$ induces
its minimal model denoted by
$
(H^*(L),\m^{g,J,L})$.
It is an object in $\UD$ and also accompanied by a morphism $\mi^{g,J,L}: \ (H^*(L), \m^{g,J,L}) \to (\Omega^*(L), \check \m^{J,L})$ in $\UD$.
Next, for a path of metrics $\mg=(g_s)_{s\in\oi}$, there is a parameterized harmonic contraction $\con(\mg)$; further using the parameterized moduli spaces for a path $\pmb J=(J_s)_{s\in\oi}$, we can naturally produce an $A_\infty$ homotopy equivalence $\mC: (\HL, \m^{g_0,J_0,L}) \to (\HL, \m^{g_1,J_1,L})$ in $\UD$ such that $\mC_{1,0}=\id$.
It always comes from a pseudo-isotopy, and it does not depend on various choices up to the ud-homotopy relation.

\subsubsection{Analytic coordinate change maps}
We take an adjacent Lagrangian submanifold $\tilde L$ in a Weinstein neighborhood $\mathcal U_L$ of $L$
such that there exists a small isotopy $F$ supported near $L$ such that $F(L)=\tilde L$.
Choose a tautological 1-form $\lambda$ on $\mathcal U_L$ such that $\lambda$ vanishes exactly on $\tilde L$ and $\omega=d\lambda$. Then, one can use the Stokes' formula to show $E(\tilde \beta)=E(\beta)-\partial\beta\cap \lambda|_L$ for any $\beta\in \pi_2(X,L)$ and $\tilde\beta:=F_*\beta\in\pi_2(X,\tilde L)$.
Note that $\lambda:=\lambda|_L$ is a closed one-form on $L$.
Conventionally, we distinguish between $\pi_2(X,L)$ and $\pi_2(X,\tilde L)$ due to the energy change, but we may feel free to identify $\pi_1(L)$ with $\pi_1(\tilde L)$ for adjacent fibers.
%

By \cite{FuCyclic}, Fukaya's trick essentially refers to the observation of a natural correspondence between a $J$-holomorphic curve $u$ with boundary on the Lagrangian $L$ and an $F_*J$-holomorphic curve $F \circ u$ with boundary on $\tilde{L}$.

\begin{convention} \label{convention_Fuk_trick}
	We say $\mC: (\HL,\m)\to (H^*(\tilde L), \tilde \m)$ is an \textit{$A_\infty$ homotopy equivalence} \textit{up to the Fukaya's trick}, if it is an $A_\infty$ homotopy equivalence in $\Mor \UD$ from $(H^*( L), \m )$ to $(H^*( L), \tilde \m^F)$ such that $\mC_{1,0}=F^{-1*}$ and the $A_\infty$ algebra $\tilde\m^F$ is defined in view of Fukaya's trick by
	\[
	\tilde\m^F_{k,\beta} =  F^{ *} \tilde\m_{k,\tilde\beta} (F^{-1*},\dots, F^{-1*})
	\]
In practice, such a $\mC$ is obtained in almost the same way as \S \ref{sss_moduli_space_A_infinity} but further using the Fukaya's trick.
\end{convention}

Suppose $\mC: (\HL,\m)\to (H^*(\tilde L), \tilde \m)$ is an $A_\infty$ homotopy equivalence in $\Mor\UD$ up to the Fukaya's trick as above.
We write $W, Q, \ia$ and $\tilde W, \tilde Q, \tilde \ia$ as in (\ref{W_Q_eq}) for $\m$ and $\tilde \m$ respectively. We write
\begin{equation}
	\label{mC_F_notation_eq}
	\pmb {\mathfrak F}_{\mC} (Y)= \sum T^{E(\gamma)} Y^{\partial\gamma} \mC_{0,\gamma}  \ \ \in \Lambda[[\pi_1(L)]] \hat\otimes H^*(L)
\end{equation}
By the semipositive condition, we know any nonzero $\mC_{0,\gamma}\in H^1( L)$. By the gappedness, $\mC_{0,0}=0$.
The following lemma describes the wall-crossing phenomenon precisely.

\begin{lem}[\cite{Yuan_I_FamilyFloer}]
	\label{wall_crossing_coordiante_change_lem}
	The assignment $\phi: \tilde Y^{\alpha} \mapsto T^{\langle \alpha, \lambda\rangle} Y^{\alpha} \exp \langle \alpha, \pmb {\mathfrak F}_{\mC} \rangle$ defines an isomorphism
	\begin{equation}
		\label{phi_eq}
		\phi=\phi_\mC: \Lambda[[\pi_1(\tilde L)]] / \tilde \ia \to \Lambda[[\pi_1(L)]] / \ia
	\end{equation}
	such that $\phi(\tilde W)=W$.
	Moreover, this map $\phi$ only depends on the ud-homotopy class of $\mC$.
	Under Assumption \ref{assumption_weak_MC}, we actually have $\ia=\tilde\ia=0$ and
	\[
	\phi=\phi_\mC: \Lambda[[\pi_1(\tilde L)]]   \to \Lambda[[\pi_1(L)]]
	\]
\end{lem}

We use the formal symbols $Y^\alpha$ and $\tilde Y^\alpha$ to distinguish the monomials in $\Lambda[[\pi_1(L)]]$ and $\Lambda[[\pi_1(\tilde L)]]$ respectively, where $\alpha\in\pi_1(L)\cong \pi_1(\tilde L)$. 
From a different perspective, the coordinate change map $\phi$ corresponds to a rigid analytic map
\begin{equation}
	\label{phi_*_eq}
	\phi: \mathcal U\subset H^1(L;  \Lambda^*) \to H^1(\tilde L ;\Lambda^*) 
\end{equation}
where the domain $\mathcal U$ is usually a proper subset but always contains $H^1(L;U_\Lambda)$ by Gromov's compactness.
When $H^1(L;\Lambda^*)$ is identified with $(\Lambda^*)^n$, and we often take the domain $\mathcal U$ to be a polytopal affinoid domain $\trop^{-1}(\Delta)$ for a rational polyhedron $\Delta$ in $\mathbb R^n$, c.f. (\ref{trop_eq}).
Concretely, the image point $\tilde{\mathbf y} := \phi (\mathbf y)$ in $H^1(\tilde L; \Lambda^*)$ is such that $\tilde {\mathbf y}^\alpha = T^{\langle \alpha,\lambda\rangle} \mathbf y^\alpha \exp\langle \alpha, \pmb {\mathfrak F}_\mC (\mathbf y) \rangle$ for $\alpha\in\pi_1(L)\cong \pi_1(\tilde L)$, where the $\mathbf y^\alpha$ denotes the image of $( \mathbf y, \alpha)$ under the pairing $H^1(L;\Lambda^*) \times \pi_1(L) \to \Lambda^*$.
Further, chosen a basis, this means $\tilde y_i= T^{c_i} y_i\exp \big( \mathfrak F_i(y_1,\dots, y_n) \big)$. One can check that whenever $\mathbf y \in H^1(L; U_\Lambda)$, we have $\tilde {\mathbf y}\in H^1(\tilde L; U_\Lambda)$.
The above lemma tells that $\tilde W(\tilde {\mathbf y})=W(\mathbf y)$.
Conventionally, the $\phi$ will be called a (family Floer) \textit{analytic coordinate change}.

\section{Self Floer cohomology for the category $\UD$}
\label{s_self_HF}
%

\subsection{Non-curved $A_\infty$ structures}
\label{ss_non_curved_A_inf_algebra}

\subsubsection{Definition}
Fix a cohomology-level $A_\infty$ algebra $(H^*(L),\m)\in \Obj\UD$ (\S \ref{sss_UD}).
We define
\begin{equation}
	\label{CF_m_k_eq}
	 \mathbf m_k:=\sum_{\beta\in\pi_2(X,L)} T^{E(\beta)} Y^{\partial\beta} \m_{k,\beta} \ :   H^*(L)^{\otimes k} \to H^*(L)\hat\otimes \Lambda[[\pi_1(L)]]
\end{equation}
where we can further extend each $\m_{k,\beta}$ by linearity so that the $\mathbf m_k$ is $\Lambda[[\pi_1(L)]]$-linear.
The $\mathbf m_k$ is actually $\Lambda_0[[\pi_1(L)]]$-linear, since $E(\beta)\ge 0$ whenever $\m_{k,\beta}\neq 0$.
Also, the semipositive condition (\S \ref{sss_UD}) tells that $\mu(\beta)\ge 0$ if $\m_{k,\beta}\neq 0$ for some $k$.
Remark that since we really use higher Maslov indices like $\mu(\beta)>2$ here, there may be more information extracted from $\UD$ than in \cite{Yuan_I_FamilyFloer}.

Be cautious that $\mathbf m_0= W\cdot \one +Q$ by (\ref{W_Q_eq}) is only available in the cohomology level (Remark \ref{W_minimal_model_rmk}).
To develop a non-curved $A_\infty$ algebra, we need to exclude $\mathbf m_0$ in our consideration. 
Notice that the coefficient ring $\Lambda[[\pi_1(L)]]$ may be replaced by some quotient ring. Although the ideal $\ia$ vanishes under Assumption \ref{assumption_weak_MC}, we retain it to preserve a more general framework. The vanishing of $\ia$ becomes relevant only when considering the derivatives and critical points of $W$.

\begin{prop}
	\label{A_infinity_non_curved_prop}
	The collection $\mathbf m=\{ \mathbf m_k\}_{k\ge 1}$ forms a (non-curved) $A_\infty$ algebra over $\Lambda[[\pi_1(L)]]/\ia$.	
\end{prop}

\begin{proof}
The $A_\infty$ associativity of $\m=(\m_{k,\beta})$ deduces that $(\m\{\m\})_{k,\beta}=0$ for all pairs $(k,\beta)$. It implies that $\sum_\beta T^{E(\beta)} Y^{\partial\beta} (\m\{\m\})_{k,\beta} (x_1,\dots, x_k)=0$ for any fixed $k$. Therefore,
	\begin{align*}
		&	\sum_{\substack{\lambda+\mu+\nu=k \\ \lambda,\mu\ge 0,\nu\ge 1}} \mathbf m_{\lambda+\mu+1} \big(x_1^\#,\dots, x_\lambda^\#, \mathbf m_\nu(x_{\lambda+1},\dots, x_{\lambda+\nu}), \dots, x_k\big) \\
		&= -\sum_\lambda \sum_\beta T^{E(\beta)} Y^{\partial\beta}  \m_{k+1,\beta} (x_1^\#,\dots, x_\lambda^\#,  \mathbf m_0, x_{\lambda+1},\dots, x_k)  
		=0 \mod \ia
	\end{align*}
 The reasons are as follows. First, all of $\m_{0,\beta}$-terms are collected on the right, which forms $\mathbf m_0=W\cdot \one+Q$. Recall the ideal $\ia$ is generated by the components of $Q$, so $\mathbf m_0\equiv W\cdot \one$ (mod $\ia$).
 Finally, the unitality of $\m$ deduces that $\sum_{0\le\lambda\le k} \m_{k+1,\beta}(x_1^\#,\dots, x_\lambda^\#, \one ,x_{\lambda+1},\dots, x_k)=0$ for every $\beta$. 
\end{proof}

\begin{cor}
	\label{m_123_cor}
	$\mathbf m_1\circ\mathbf m_1=0$,
	$\mathbf m_1 \big(\mathbf m_2(x,y) \big) + \mathbf m_2( x^\#, \mathbf m_1(y)) + \mathbf m_2(\mathbf m_1(x), y)=0$, and
	\[
	\Scale[0.85]{	\mathbf m_2(\mathbf m_2( x,y ), z)+ \mathbf m_2(x^\#, \mathbf m_2(y,z))+ \mathbf m_1(\mathbf m_3(x,y,z)) + \mathbf m_3(\mathbf m_1(x),y,z)+\mathbf m_3(x^\#, \mathbf m_1(y), z) +\mathbf m_3(x^\#, y^\#, \mathbf m_1(z))=0}
	\]
\end{cor}

\begin{defn}
		The \textbf{self Floer cohomology of $L$} (associated to $\m$) is defined to be the $\mathbf m_1$-cohomology:
	\begin{equation}
		\label{HF_Lambda_eq}
		\HF(L,\m) := H^* \big( H^*(L)\hat\otimes \Lambda[[\pi_1(L)]] / \ia , \mathbf m_1 \big)
	\end{equation}
Moreover, $[x]\cdot [y]= [\mathbf m_2( -x^\#, y)]$ defines a ring structure on $\HF(L,\m)$ with a unit $[\one]$.
\end{defn}


Note that we may also replace the formal power series ring $\Lambda[[\pi_1(L)]]$ by any polyhedral affinoid algebra contained in it.
Since $\deg \m_{1,\beta}=1-\mu(\beta)\equiv 1 \ (\mathrm{mod} \ 2)$, the $\HF(L,\m)$ is at least $\mathbb Z_2$-graded.


Suppose $\mC: (\HL,\m)\to (H^*(\tilde L), \tilde \m)$ is an $A_\infty$ homotopy equivalence up to the Fukaya's trick in the sense of Convention \ref{convention_Fuk_trick} with respect to an isotopy $F$ such that $F(L)=\tilde L$.
Consider the two non-curved $A_\infty$ algebras $\mathbf m$ and $\tilde{\mathbf m}$ obtained as in Proposition \ref{A_infinity_non_curved_prop} using the two \textit{gapped} $A_\infty$ algebras $\m$ and $\tilde \m$ respectively.
By virtue of Lemma \ref{wall_crossing_coordiante_change_lem},
there is an $\mC$-induced isomorphism $\phi=\phi_{\mC}:\Lambda[[\pi_1(\tilde L)]]/\tilde \ia\cong\Lambda[[\pi_1(L)]]/\ia$.
Similar to (\ref{CF_m_k_eq}), we define
\begin{equation}
	\label{mC_k_eq}
	\mathbf C_k:= \mathbf C^{(F)}_k :=  F^{-1*} \sum_\beta T^{E(\beta)} Y^{\partial\beta} \mC_{k,\beta}  : \big( H^*(L) \hat \otimes \Lambda[[\pi_1(L)]]/\ia \big)^{\otimes k}  \to H^*(\tilde L)\hat\otimes \Lambda[[\pi_1(\tilde L)]] / \tilde \ia
\end{equation}
where all the formal power series coefficients in $Y$ will be afterward transformed to those in $\tilde Y$ via the isomorphism $\phi=\phi_\mC$.
Note that $\mathbf C_0=F^{-1*} \pmb {\mathfrak F}_{\mC}$, c.f. (\ref{mC_F_notation_eq}). Recall that the collections $\mathbf m$ and $\tilde {\mathbf m}$ exclude $\mathbf m_0$ and $\tilde { \mathbf m}_0$. Similarly, we set $\mathbf C=\{\mathbf C_k\}_{k\ge 1}$ excluding $\mathbf C_0$ as well.

\begin{prop}
	\label{A_infinity_morphism_non_curved_prop}
	The $\mathbf C=\mathbf C^{(F)}$ gives a (non-curved) $A_\infty$ homomorphism from $\mathbf m$ to $\tilde{\mathbf m}$.
\end{prop}

\begin{proof}
The $A_\infty$ associativity equations tells $(\tilde\m^F \diamond \mC)_{k,\beta}= (\mC \{ \m\})_{k,\beta}$ (c.f. Convention \ref{convention_Fuk_trick}). Hence,
	\begin{equation}
		\label{mC_k_non_curved_A_inf_eq}
		\sum_\beta T^{E(\beta)} Y^{\partial\beta} (\tilde \m^F \diamond \mC)_{k,\beta} (x_1,\dots, x_k) = \sum_\beta T^{E(\beta)} Y^{\partial \beta}  (\mC\{ \m\})_{k,\beta}(x_1,\dots, x_k)
	\end{equation}
for any fixed $k\ge 1$.
	On the left side, the terms involving $\mC_{0,\beta}$ form the series $\pmb {\mathfrak F}_{\mC}$ as in (\ref{mC_F_notation_eq}). Besides, since the components of $\pmb {\mathfrak F}_\mC$ are contained in $H^1(L)$, the divisor axiom of $\tilde \m^F$ can be applied after we use Lemma \ref{val=0-lem}.
	Therefore, the left side of (\ref{mC_k_non_curved_A_inf_eq}) becomes:
	\begin{align*}
		\sum_{\substack{ \ell\ge 1, \beta', \beta_1,\dots,\beta_\ell \\ 0 = j_0 < j_1<\dots < j_{\ell-1} < j_\ell=k  }}
		T^{E(\beta')}Y^{\partial\beta'}
		\exp \langle \partial\beta',  \pmb {\mathfrak F}_\mC\rangle \
		\tilde \m^F_{\ell,\beta'} \big(
		T^{E(\beta_1)}Y^{\partial\beta_1} \mC_{j_1-j_0,\beta_1}\otimes \cdots \otimes 
		T^{E(\beta_\ell)}Y^{\partial\beta_{\ell}}\mC_{j_\ell-j_{\ell-1},\beta_\ell} \big) (x_1,\dots, x_k)	
	\end{align*}
	On the other hand, using Lemma \ref{wall_crossing_coordiante_change_lem} yields that:
	\begin{equation*}
		\label{tilde_m_ell_eq}
		\sum T^{E(\beta)} Y^{\partial\beta} \exp \langle \partial\beta, \pmb {\mathfrak F}_{\mC}\rangle\tilde\m_{\ell,\beta} =\sum T^{E(\tilde\beta)} \tilde Y^{\partial\tilde\beta}\tilde \m_{\ell,\beta}= \tilde {\mathbf m}_\ell
	\end{equation*}
(Be careful to distinguish $Y$ from $\tilde Y$.) In summary, the left side of (\ref{mC_k_non_curved_A_inf_eq}) further becomes:
\begin{align*}
		\sum_{0 = j_0 < j_1<\dots < j_{\ell-1} < j_\ell=k}	
		F^* \tilde {\mathbf m}_\ell
		( \mathbf C^{(F)}_{j_1-j_0} \otimes \cdots \otimes \mathbf C^{(F)}_{j_\ell-j_{\ell-1}} )
		(x_1,\dots, x_k)  
	\end{align*}
	Now, in the right side of (\ref{mC_k_non_curved_A_inf_eq}), we collect all those terms involving $\m_{0,\beta}$ first, and their sum is equal to
	\begin{align*}
		\sum_{\beta, 0\le \lambda\le k}
		T^{E(\beta)} Y^{\partial\beta} \mC_{k+1,\beta} 
		\big(
		x^\#_1,\dots, x^\#_\lambda,  \mathbf m_0, x_{\lambda+1},\dots, x_k 
		\big) 	
	\end{align*}
	But, $\mathbf m_0= W\cdot \one + Q\equiv W\cdot \one \ (\mathrm{mod} \ \ia)$ and $k\ge 1$, and the above summation vanishes by the unitality of $\mC$. To conclude, what remain on the right side of (\ref{mC_k_non_curved_A_inf_eq}) are precisely as follows:
	\[
	\sum_{\substack{\lambda,\mu\ge 0, \nu\ge 1 \\ \lambda+\mu+\nu=k}} F^* \mathbf C^{(F)}_{\lambda+\mu+1} 
	\big(x_1^\#, \dots, x_\lambda^\#, \mathbf m_\nu(x_{\lambda+1},\dots, x_{\lambda+\nu} ), \dots, x_k\big)
	\]
	Applying $F^{-1*}$ to the two equations above, we complete the proof.
\end{proof}

\begin{rmk}
	\label{rmk_selfHF_morphism}
		It is worth mentioning that the equation (\ref{mC_k_non_curved_A_inf_eq}) in the above proof assumes $k\ge 1$. In contrast, the gluing analytic coordinate change maps developed in \cite{Yuan_I_FamilyFloer} (i.e. Lemma \ref{wall_crossing_coordiante_change_lem}) essentially uses the same equation as (\ref{mC_k_non_curved_A_inf_eq}) but only for $k=0$. Therefore, we really use more information about the category $\UD$ here. Compare also Remark \ref{rmk_selfHF_homot_inv}.
\end{rmk}

\subsubsection{Invariance}
\label{sss_invariance_HF}
In particular, Proposition \ref{A_infinity_morphism_non_curved_prop} tells that $\tilde {\mathbf m}_1\circ \mathbf C_1 =\mathbf C_1\circ \mathbf m_1$ and $\tilde {\mathbf m}_2( \mathbf C_1 (x_1) , \mathbf C_1 (x_2) )= \mathbf C_1 ( \mathbf m_2(x_1,x_2))$.
Therefore, we know that $\mathbf C_1: H^*(L)\hat\otimes \Lambda[[\pi_1(L)]] /\ia \to H^*(\tilde L) \hat\otimes \Lambda[[\pi_1(\tilde L)]] / \tilde \ia$ is a cochain map and further induces a $\Lambda$-algebra homomorphism
\begin{equation}
	\label{eta_HF_mC_1}
	 \eta_{\mC} :=[\mathbf C_1] : \HF(L,\m) \to \HF (\tilde L, \tilde \m)
\end{equation}

\begin{prop}
	\label{inv_Floer_self_prop}
	The $\eta_{\mC}$ only relies on the ud-homotopy class of $\mC$ in $\UD$.
\end{prop}

\begin{proof}
	We basically use the ideas in \cite{Yuan_I_FamilyFloer}.
	Suppose $\mC,\mC'$ are two $A_\infty$ homotopy equivalences from $(\HL, \m)$ to $(H^*(\tilde L),\tilde \m)$ up to the Fukaya's trick in $\Mor \UD$ about some isotopy $F$ (Convention \ref{convention_Fuk_trick}). Suppose also that they are ud-homotopic to each other. Our task is to prove
	$\eta_{\mC}=\eta_{\mC'}$.
	
	By definition, there exist operator systems $\f_s=\{(\f_s)_{k,\beta}\}$ and $\h_s = \{ (\h_s)_{k,\beta} \}$ for $0\le s\le 1$ such that
	the conditions (a) (b) (c) (d) in \S \ref{sss_whitehead_homotopy_theory_UD} hold, where $\f_0=\mC$ and $\f_1=\mC'$.
	By the condition (b), we get
	\begin{align}
		\sum_\beta T^{E(\beta)} Y^{\partial\beta} \frac{d}{ds} (\f_s)_{1,\beta}( \pmb x)
		&= \notag
		\sum_{\beta_1, \beta_2} T^{E(\beta_1)} Y^{\partial\beta_1} (\h_s)_{1,\beta_1} ( T^{E(\beta_2)} Y^{\partial\beta_2}\m_{1,\beta_2} (\pmb x)) \\
		&+ \notag
		\sum_{\beta_1,\beta_2} T^{E(\beta_1)} Y^{\partial\beta_1}(\h_s)_{2,\beta_1} \big(\pmb x^\#, T^{E(\beta_2)} Y^{\partial\beta_2}\m_{0,\beta_2} \big) \\
		&+ 	
		\label{f_s_h_s_application_eq}
		\sum_{\beta_1,\beta_2} T^{E(\beta_1)} Y^{\partial\beta_1}(\h_s)_{2,\beta_1} \big( T^{E(\beta_2)} Y^{\partial\beta_2}\m_{0,\beta_2},  \pmb x \big) \\
		&+ \notag
		\sum_\beta T^{E(\beta)} Y^{\partial\beta}   
		\sum_{1\le i\le \ell, \ \beta_0+\cdots+\beta_\ell+\beta'=\beta}
		\CU[\tilde\m^F]_{\ell,\beta_0} \big(	(\h_s)_{0,\beta'} ;
		(\f_s)_{0,\beta_1},\dots, (\f_s)_{1,\beta_i}( \pmb x),\dots, (\f_s)_{0,\beta_\ell}
		\big)	 \\
		&+ \notag
		\sum_\beta T^{E(\beta)} Y^{\partial \beta} \sum_{1\le i\le \ell, \ \beta_0+\cdots+\beta_\ell=\beta}
		\tilde  \m^F_{\ell, \beta_0} 
		\Big(
		(\f_s)_{0,\beta_1},\dots, (\h_s)_{1,\beta_i}(\pmb x),\dots, (\f_s)_{0,\beta_\ell}
		\Big)
	\end{align}
	The second and third sums are all zero modulo $\ia$ due to the condition (c) and the decomposition (\ref{W_Q_eq}).
	The fourth sum vanishes
	because the semipositive condition ensures $\deg (\h_s)_{0,\beta'}=-\mu(\beta')=0$ and we can use the cyclical unitality.
	For the fifth sum, the semipositive condition also tells $\deg (\f_s)_{0,\beta_j}=1-\mu(\beta_j)=1$; by the divisor axiom and by Lemma \ref{val=0-lem}, we conclude that
	the fifth sum is equal to
	\[
	\sum_{\beta_0,\beta_1} T^{E(\beta_0)} Y^{\partial\beta_0} 
	\exp \langle \partial\beta_0, \pmb {\mathfrak F}_{\f_s}\rangle \
	\tilde \m^F_{1,\beta_0} \big(  T^{E(\beta_1)} Y^{\partial\beta_1} (\h_s)_{1,\beta_1} (\pmb x) \big)
	\]
	Then, since $\f_s\simud\mC\simud\mC'$, it follows from Lemma \ref{weakMC_observation_lem} that $\pmb {\mathfrak F}_{\f_s} \equiv \pmb {\mathfrak F}_{\mC}=\pmb {\mathfrak F}_{\mC'}$ (mod $\ia$), where we recall the notation
	in (\ref{mC_F_notation_eq}). Hence, by Lemma \ref{wall_crossing_coordiante_change_lem}, after we transform into $\tilde Y$-coefficients, this fifth sum is
	\[
	F^* \sum_{\beta_0,\beta_1} T^{E(\beta_0)} \tilde Y^{\partial\beta_0}
	\tilde \m_{1,\beta_0} \big(  T^{E(\beta_1)} Y^{\partial\beta_1}  F^{-1*}(\h_s)_{1,\beta_1} (\pmb x) \big)
	\]
	To conclude, only the first and fifth sum in (\ref{f_s_h_s_application_eq}) survive.
	Now, we set
$\xi =\sum_\beta T^{E(\beta)} Y^{\partial\beta} \textstyle 
		F^{-1*} \int_0^1 (\h_s)_{1,\beta}$.
	Taking the integration from $s=0$ to $s=1$ and applying $F^{-1*}$ to the both sides of the initial equation (\ref{f_s_h_s_application_eq}), we exactly obtain that
$
		{\mathbf C}'_1(\pmb x)-\mathbf C_1(\pmb x) = 
		\xi \circ \mathbf m_1+ \tilde {\mathbf m}_1 \circ \xi
$.
	In other words, the $\xi$ gives a cochain homotopy between the two cochain maps $\mathbf C_1$ and $\mathbf C'_1$.
	Particularly, the induced morphisms on the cohomology are the same $\eta_{\mC}=\eta_{\mC'}$. The proof is now complete.
\end{proof}

\begin{rmk}
	\label{rmk_selfHF_homot_inv}
	Similar to Remark \ref{rmk_selfHF_morphism}, we indicate that the equation (\ref{f_s_h_s_application_eq}) without the input $\pmb x$ is used in \cite{Yuan_I_FamilyFloer} to prove that the analytic coordinate change maps are well-defined. In contrast, adding the input $\pmb x$ here, we have extracted more information from the category $\UD$ than \cite{Yuan_I_FamilyFloer}.
\end{rmk}

\subsubsection{Composition}
\label{sss_composition_HF}

Now, we consider three adjacent Lagrangian submanifolds $L, \tilde L$, and $\tilde{\tilde L}$.
Let $F$ and $F'$ be small isotopies such that $F(L)=\tilde L$ and $F'(\tilde L)=\tilde {\tilde L}$. Suppose $\mC: (H^*(L),\m)\to (H^*(\tilde L),\tilde \m)$ and $\mC': (H^*(\tilde L),\tilde \m) \to (H^*(\tilde{\tilde L}), \tilde {\tilde \m})$ are $A_\infty$ homotopy equivalences in $\Mor \UD$ up to Fukaya's tricks.
By Proposition \ref{A_infinity_morphism_non_curved_prop}, they give rise to two non-curved $A_\infty$ homomorphisms $\mathbf C=\{\mathbf C_k\}_{k\ge 1}$ and $\mathbf C'=\{\mathbf C'_k\}_{k\ge 1}$.
By Proposition \ref{A_infinity_morphism_non_curved_prop} again, the composition gapped $A_\infty$ morphism $\mC'\diamond \mC$ in $\Mor\UD$ also induces a non-curved $A_\infty$ homomorphism, denoted by $\mathbf C'\diamond \mathbf C=((\mathbf C'\diamond \mathbf C)_k)_{k\ge 1}$. We call it the \textit{composition} of $\mathbf C'$ and $\mathbf C$.
The notation and term we use here can be justified as follows:

\begin{prop}
	The non-curved $A_\infty$ homomorphism $\mathbf C'\diamond \mathbf C=((\mathbf C'\diamond \mathbf C)_k)_{k\ge 1}$ satisfies that
	\[
	(\mathbf C'\diamond \mathbf C)_k=\sum_{\ell\ge 1, \ 0=j_0<j_1<\dots <j_{\ell-1}<j_\ell=k} \mathbf C'_\ell  \big(\mathbf C_{j_1-j_0},\dots, \mathbf C_{j_\ell-j_{\ell-1}}\big)
	\]
\end{prop}

\begin{proof}
	Fix $k\ge 1$. By definition, we first know $(\mathbf C'\diamond \mathbf C)_k=\sum_\beta T^{E(\beta)} Y^{\partial\beta} (\mC'\diamond \mC)_{k,\beta}$.
	Expand each $(\mC'\diamond \mC)_{k,\beta}$ by (\ref{composition_Gerstenhaber_eq}), and all the terms involving $\mC_{0,\beta}$ form the series $\pmb {\mathfrak F}_{\mC}$.
	Then, as in the proof of Proposition \ref{A_infinity_morphism_non_curved_prop}, the divisor axiom can be applied after we use Lemma \ref{val=0-lem}, and we finally obtain
\begin{align*}
	\Scale[0.9]{(\mathbf C'\diamond \mathbf C)_k}
	=
	\sum_{\substack{ \ell\ge 1, \beta', \beta_1,\dots,\beta_\ell \\ 0 = j_0 < j_1<\dots < j_{\ell-1} < j_\ell=k  }}
	\Scale[0.9]{
		T^{E(\beta')}Y^{\partial\beta'} \exp\langle \partial\beta', \pmb {\mathfrak F}_\mC \rangle \mC'_{\ell,\beta'} \big(
	T^{E(\beta_1)}Y^{\partial\beta_1} \mC_{j_1-j_0,\beta_1}\otimes \cdots \otimes 
	T^{E(\beta_\ell)}Y^{\partial\beta_{\ell}}\mC_{j_\ell-j_{\ell-1},\beta_\ell} \big)}
\end{align*}
Recall that $\tilde Y^{\partial\beta'}=T^{E(\beta')}Y^{\partial\beta'} \exp\langle \partial\beta', \pmb {\mathfrak F}_\mC \rangle$ by Lemma \ref{wall_crossing_coordiante_change_lem}. The proof is now complete.
\end{proof}

In particular, the above proposition infers that $(\mathbf C'\diamond \mathbf C)_1=\mathbf C'_1\circ \mathbf C_1$ is also a cochain map; hence,
\begin{equation}
	\label{eta_composition_eq}
\eta_{\mC'\diamond \mC}=\eta_{\mC'}\circ \eta_{\mC}: \HF(L,\m)\to \HF(\tilde L,\tilde \m) \to \HF(\tilde {\tilde L},\tilde{\tilde \m} )
\end{equation}

\begin{cor}
	\label{inv_Floer_self_cor}
	In the context of Proposition \ref{inv_Floer_self_prop}, the $\eta_\mC$ is an isomorphism.
\end{cor}

\begin{proof}
	The $\mC$ admits a ud-homotopy inverse $\mC^{-1}$ so that $\mC^{-1}\diamond \mC\simud \id$ and $\mC\diamond \mC^{-1} \simud\id$ (up to Fukaya's tricks).
	By Proposition \ref{inv_Floer_self_prop} and (\ref{eta_composition_eq}), we get $\eta_{\mC^{-1}}\circ \eta_{\mC}= \eta_{\mC^{-1}\diamond \mC} =\eta_\id= \id$ and vice versa.
\end{proof}

\subsection{Evaluation of self Floer cohomology}
\label{ss_evaluation_HF}

\subsubsection{Background and review}
In the literature, the self Floer cohomology is defined by a different way. 
Here let's assume $\pi_1(L)$ has no torsion part.
Instead of $H^*(L)\hat\otimes \Lambda[[\pi_1(L)]]  / \ia $ in Proposition \ref{A_infinity_non_curved_prop}, we previously work with
$\CF(L; \Lambda) := H^*(L)\hat\otimes \Lambda$ and consider a \textit{weak bounding cochain} $b\in H^{\mathrm{odd}}(L; \Lambda_+)$, i.e.
$
\textstyle \sum_{k\ge 0} \sum_\beta T^{E(\beta)} \m_{k,\beta}(b,\dots, b) \in \Lambda_+\cdot \one
$
for the constant-one function $\one\in H^0(L)$. Then, we define
\[
\mathbf m_k^b (x_1,\dots , x_k) = \sum T^{E(\beta)} \m_{\bullet,\beta} (b,\dots,b , x_1, b,\dots \ \dots, b, x_k, b, \dots, b)
\]
By condition, the $\mathbf m_0^b\in \Lambda_+\cdot \one$ will be eliminated.
It is known in the literature that (see e.g. \cite[Proposition 3.6.10]{FOOOBookOne}):
\textit{The collection $\{\mathbf m_k^b\}_{k\ge 1}$ forms a (non-curved) $A_\infty$ algebra on $H^*(L)\hat\otimes \Lambda$.}
Then, the self Floer cohomology is usually defined by
$\HF(L, b; \Lambda) := H^* ( \CF(L;\Lambda) ;  \mathbf m_1^b)
$.

\subsubsection{Evaluation at a weak bounding cochain}

For the family Floer theory and SYZ picture, we should mainly focus on those degree-one weak bounding cochains $b$ in $H^1(L; \Lambda_+)$. Indeed, if $L$ is a Lagrangian torus fiber in the SYZ framework, its dual torus fiber is expected to be $H^1(L; U(1))$, at least at the topological level.
Now, we want to make connections with $\HF(L,\m)$ in (\ref{HF_Lambda_eq}).
In fact, since $\deg b=1$, the divisor axiom of $\m$ implies that
\[
\mathbf m_k^b(x_1,\dots, x_k) =\sum T^{E(\beta)} e^{\partial\beta\cap b} \m_{k,\beta}(x_1,\dots, x_k) 
\]
and we can even actually require $b\in H^1(L; \Lambda_0)$ instead of $H^1(L; \Lambda_+)$ without violating the convergence.
There is a quotient map 
$\mathsf q: H^1(L; \Lambda_0) \to H^1(L; U_\Lambda)$ induced by $\exp: \Lambda_0 \to U_\Lambda$.
For a basis of $H^1(L)$, we write $b=(x_1,\dots, x_m)$ for $x_i\in\Lambda_0$; the image point $\mathbf y=\mathsf q(b)$ is given by $(y_1,\dots, y_m)$ for $y_i=e^{x_i}\in U_\Lambda$.
In general, let $\mathbf y^\alpha$ be the image of $(\alpha, \mathbf y)$ under the natural pairing $\pi_1(L)\times H^1(L; \Lambda^*)\to \Lambda^*$. We also pick a basis so that $\pi_1(L)\cong\mathbb Z^m$. For $\alpha=(\alpha_1,\dots, \alpha_m)\in\pi_1(L)$ and $\mathbf y=(y_1,\dots, y_m)\in H^1(L; \Lambda^*)$, we have $\mathbf y^\alpha= y_1^{\alpha_1}\cdots y_m^{\alpha_m}$.
Now, we can introduce a slight variation of the above $\mathbf m_k^b$ as follows:
\begin{equation}
	\label{m_y_k_eq}
\mathbf m_k^{\mathbf y}(x_1,\dots, x_k) = \sum T^{E(\beta)} \mathbf y^{\partial\beta} \m_{k,\beta} (x_1,\dots, x_k)
\end{equation}
Note that $\mathbf m^{\mathbf y}_0\equiv W(\mathbf y)\cdot \one +Q(\mathbf y)$ (\ref{W_Q_eq}).
By the divisor axiom and under the semipositive condition, $Q(\mathbf y)=0$ if and only if the $\mathbf y$ admits a lift $b$ which is a weak bounding cochain.
Hence, by the transformation $y_i=e^{x_i}$, it simply goes back to the previous $\mathbf m_k^b$. 
Nevertheless, the advantage of (\ref{m_y_k_eq}) is that we can further allow $\mathbf y$ to run over a small neighborhood of $H^1(L;U_\Lambda)\cong U_\Lambda^n$ in $H^1(L;\Lambda^*)\cong (\Lambda^*)^n$.

\begin{prop}
	\label{A_infinity_CF_y_non_curved_prop}
If $Q(\mathbf y)=0$, then $\mathbf m^{\mathbf y}=\{\mathbf m_k^{\mathbf y}\}_{k\ge 1}$ forms a non-curved $A_\infty$ algebra on $H^*(L)\hat\otimes \Lambda$.
\end{prop}

\begin{proof}
	We may argue just as Proposition \ref{A_infinity_non_curved_prop} by the $A_\infty$ associativity and the unitality of $\m$.
\end{proof}

Particularly, $\mathbf m_1^{\mathbf y}\circ \mathbf m_1^{\mathbf y}=0$, so we can define the \textbf{self Floer cohomology of $L$ at $\mathbf y$}:
\begin{equation}
	\label{HF_at_y_eq}
	\HF(L, \m, \mathbf y):=H^*( H^*(L)\hat\otimes \Lambda ,\mathbf m^{\mathbf y}_1)
\end{equation}
As before, we can also show that
	$[x_1]\cdot[x_2]=[\mathbf m^{\mathbf y}_2(-x_1^\#, x_2)]$ defines a ring structure on $\HF(L,\m,\mathbf y)$ with a unit $[\one]$.
Additionally, we can consider the \textit{evaluation map at $\mathbf y$}:
\begin{equation}
	\label{Eva_y_eq}
	\mathcal E_{\mathbf y}:  H^*(L)\hat\otimes \Lambda[[\pi_1(L)]] \to H^*(L)\hat\otimes \Lambda \qquad Y^{\alpha}\cdot x \mapsto  \mathbf y^\alpha \cdot x 
\end{equation}
where $\alpha\in\pi_1(L)$ and $x\in H^*(L)$.
For the sake of convergence, we should replace $\Lambda[[\pi_1(L)]]$ by some polyhedral affinoid algebra in it, and we also require each coefficient $\mathbf y$ is contained in the corresponding polyhedral affinoid domain. But, let's make this point implicit for clarity.
Here we also assume $\mathbf y$ admits a weak bounding cochain lift $b$ up to Fukaya's trick, namely, $Q(\mathbf y)=0$. Thus, the above $\mathcal E_{\mathbf y}$ also descends to a quotient map modulo $\ia$, which is still denoted by
$\mathcal E_{\mathbf y}:  H^*(L)\hat\otimes \Lambda[[\pi_1(L)]] /\ia \to H^*(L)\hat\otimes \Lambda$.
Further, we can extend the definition of $\mathcal E_{\mathbf y}$ by setting $(\mathcal E_{\mathbf y})_1=\mathcal E_{\mathbf y}$, $(\mathcal E_{\mathbf y})_k=0$, $k\neq 1$, and we can directly check the following:
\begin{lem}
	\label{Eva_y_HF_to_HF_lem}
	The $\mathcal E_{\mathbf y}$ gives a non-curved $A_\infty$ homomorphism from
	$\mathbf m$ to $\mathbf m^{\mathbf y}$.
	Moreover, the induced map $\mathcal E_{\mathbf y}=[\mathcal E_{\mathbf y}]:  \HF(L, \m) \to \HF(L,\m,\mathbf y)$ is a unital $\Lambda$-algebra homomorphism so that $\mathcal E_{\mathbf y}([\one])=[\one]$.
\end{lem}



Suppose $\mC: (\HL,\m)\to  ( H^*(\tilde L),\tilde \m)$ is an $A_\infty$ homotopy equivalence up to the Fukaya's trick for a small isotopy $F$ with $F(L)=\tilde L$ (Convention \ref{convention_Fuk_trick}).
For the analytic map $\phi$ in (\ref{phi_*_eq}), the point $\tilde {\mathbf y}=\phi(\mathbf y)$ also admits a weak bounding cochain lift $\tilde b$, i.e. $\tilde Q(\tilde {\mathbf y})=0$.
Then, we get two non-curved $A_\infty$ algebras $\{\m_k^{\mathbf y}\}_{k\ge 1}$ and $\{\tilde \m_k^{\tilde{\mathbf y}}\}_{k\ge 1}$ on $H^*(L)\hat\otimes \Lambda$ and $H^*(\tilde L)\hat\otimes \Lambda$ by Proposition \ref{A_infinity_CF_y_non_curved_prop}.
Just as (\ref{mC_k_eq}), we define
\begin{equation}
	\label{mC_y_k_eq}
	\mathbf C_k^{\mathbf y}: = F^{-1} \sum T^{E(\beta)} \mathbf y^{\partial\beta} \mC_{k,\beta} :  \big( H^*(L)\hat \otimes \Lambda \big)^k \to H^*(\tilde L) \hat\otimes \Lambda
\end{equation}

\begin{prop}
	The collection $\mathbf C^{\mathbf y}=\{\mathbf C_k^{\mathbf y}  \}_{k\ge 1}$ forms a (non-curved) $A_\infty$ homomorphism from
	$\mathbf m^{\mathbf y}$ to
	 $\tilde {\mathbf m}^{\tilde{\mathbf y}}$ so that $\mathcal E_{\tilde {\mathbf y}} \diamond \mathbf C =\mathbf C^{\mathbf y}\diamond \mathcal E_{\mathbf y}$. In particular, we have $\tilde {\mathbf m}_1^{\tilde {\mathbf y}} \circ  \mathbf C_1^{\mathbf y}=\mathbf C_1^{\mathbf y}\circ \mathbf m_1^{\mathbf y}$ and $\tilde {\mathbf m}_2^{\tilde {\mathbf y}}  (\mathbf C_1^{\mathbf y} (x_1), \mathbf C_1^{\mathbf y}(x_2) ) = \mathbf C_1^{\mathbf y} (\mathbf m_2^{\mathbf y} (x_1,x_2))$.
\end{prop}

\begin{proof}
	The proof is just the same as that of Proposition \ref{A_infinity_morphism_non_curved_prop}.
\end{proof}

\begin{prop}
	\label{inv_Floer_self_prop_y}
	The $\mathbf C_1^{\mathbf y}$ similarly induces a $\Lambda$-algebra homomorphism
	$\eta_\mC^{\mathbf y}= [\mathbf C_1^{\mathbf y}] : \HF(L,\m, \mathbf y) \to \HF(L, \tilde \m, \tilde{\mathbf y})$ which only relies on the ud-homotopy class of $\mC$ in $\UD$. In particular, if $\mC$ admits a ud-homotopy inverse, then $\eta^{\mathbf y}$ is an isomorphism.
\end{prop}

\begin{proof}
	The proof is just the same as that of Proposition \ref{inv_Floer_self_prop} and Corollary \ref{inv_Floer_self_cor}.
\end{proof}

The proposition tells that $\HF(L,\m,\mathbf y)$ is invariant under the various analytic coordinate change $\phi$ (Lemma \ref{wall_crossing_coordiante_change_lem}). From the viewpoint of Theorem \ref{Main_theorem_thesis}, we should view $\mathbf y$ and $\tilde {\mathbf y}=\phi(\mathbf y)$ as the same point in the mirror analytic space $X^\vee$ but referring to different local charts.

\subsection{Critical points of the superpotential}
\label{ss_critical_points_nonvanishing}

Take a formal power series
$F=\sum_{j=1}^\infty c_j Y^{\alpha_j}$ in $\Lambda[[\pi_1(L)]]$ (where $c_j\in\Lambda$ and $\alpha_j\in \pi_1(L)$).
Given $\theta \in H^1(L) \cong \Hom (\pi_1(L),\mathbb R)$, we define the \textit{logarithmic derivative along $\theta$} of $F$ by
\begin{equation}
	\label{D_theta_eq_defn}
\textstyle D_\theta F= \sum_{j=1}^\infty c_j  \langle \alpha_j, \theta\rangle Y^{\alpha_j} 
\end{equation}
One can easily check the following properties:  (a) $D_\theta (F\cdot G)= F\cdot D_\theta G + G\cdot D_\theta F$; (b) $D_\theta(\exp(F)) =\exp (F) \cdot D_\theta F$.
Given an ideal $\mathcal I$ in $\Lambda[[\pi_1(L)]]$, we denote by
\begin{equation}
	\label{D_ideal_eq}
	D \mathcal I = \ \text{the ideal generated by} \  \{ D_\theta F\mid \forall  \ \theta\in H^1(L;\mathbb Z), \,   \forall \ F\in \mathcal I\} 
\end{equation}

We note that $D\mathcal I=0$, whenever $\mathcal I =0$.

Suppose we have an $A_\infty$ homotopy equivalence $\mC:  (\HL, \m) \to (H^*(\tilde L), \tilde \m)$ in $\Mor \UD$ up to the Fukaya's trick about a small isotopy $F$ such that $F(L)=\tilde L$ (Convention \ref{convention_Fuk_trick}). Recall that by Lemma \ref{wall_crossing_coordiante_change_lem}, the analytic coordinate change map 
\[
\phi: \Lambda[[\pi_1(\tilde L)]]/\tilde\ia \to \Lambda[[\pi_1(L)]]/\ia, \qquad 
\tilde Y^{\alpha} \mapsto T^{\langle \alpha, \lambda\rangle} Y^{\alpha} \exp \langle \alpha, \pmb {\mathfrak F}_{\mC}(Y)\rangle
\]
is an isomorphism such that $\phi(\tilde W)=W$. Recall that $W, Q, \ia$ and $\tilde W, \tilde Q, \tilde \ia$ are given as in (\ref{W_Q_eq}) for $\m$ and $\tilde \m$ respectively. 
Regard $W$ or $\tilde W$ as a principal ideal, we can define $DW$ or $D\tilde W$ just like (\ref{D_ideal_eq}).

\begin{prop}
	\label{DW_ideal_thm}
The analytic coordinate change map $\phi$ match the ideal $D\tilde W$ with $DW$.
\end{prop}

\begin{proof}
	Suppose $\theta\in H^1(L)$ and $\tilde\theta \in H^1(\tilde L)$ are $F$-related.	
	First, we compute the commutator of $\phi$ and $D$ applied to a monomial $\tilde Y^{\alpha}$:
\[
	D_\theta( \phi(\tilde Y^\alpha)) -\phi (D_{\tilde\theta} \tilde Y^\alpha) = T^{\langle \alpha, \lambda\rangle}  Y^\alpha \exp \langle \alpha, \pmb {\mathfrak F}_{\mC}(Y)\rangle D_\theta \langle \alpha, \pmb {\mathfrak F}_{\mC}(Y)\rangle
\]
Note that $\pmb {\mathfrak F}_{\mC} (Y)= \sum T^{E(\gamma)} Y^{\partial\gamma} \mC_{0,\gamma}$ is contained in $\Lambda[[\pi_1(L)]]\hat\otimes H^1(L)$ by (\ref{mC_F_notation_eq}). Next, we compute
	\begin{align*}
		D_\theta W -\phi (D_{\tilde \theta} \tilde W)
		&
		=
		D_\theta( \phi(\tilde W)) -\phi (D_{\tilde \theta} \tilde W) \\
		&
		=  	\sum_{\mu(\beta)=2} T^{E(\beta)} Y^{\partial\beta} \exp \langle \partial\beta, \pmb {\mathfrak F}_{\mC} (Y)\rangle D_\theta  \langle \partial\beta, \pmb {\mathfrak F}_{\mC}(Y)\rangle \cdot \tilde \m_{0,\beta}		\\
		&
		=	\sum_{\mu(\beta)=2} T^{E(\beta)} Y^{\partial\beta} \exp \langle \partial\beta, \pmb {\mathfrak F}_{\mC} (Y)\rangle \Big(
		\sum_{\gamma\neq 0, \mu(\gamma)=0} T^{E(\gamma)} \langle \partial\gamma, \theta\rangle Y^{\partial\gamma} \langle \partial\beta, \mC_{0,\gamma} \rangle
		\Big) \tilde \m_{0,\beta}	\\
		&
		=
		\sum_{\gamma\neq 0,\mu(\gamma)=0} T^{E(\gamma)}
		\langle \partial\gamma, \theta\rangle Y^{\partial\gamma} \sum_{\mu(\beta)=2} 
		\langle \partial\beta, \mC_{0,\gamma}\rangle    \cdot
		T^{E(\beta)} 
		 Y^{\partial\beta} \exp\langle \partial\beta, \pmb {\mathfrak F}_\mC(Y)\rangle  \cdot\tilde \m_{0,\beta}
	\end{align*}
For a fixed $\gamma$, we observe that the second sum over $\beta$ can be further simplified as follows:
\begin{align*}
	 \Scale[0.9]{ 
	 	\displaystyle 
\sum  \langle \partial \beta,   \mC_{0,\gamma}\rangle \ \phi\big( T^{E(\tilde\beta)} Y^{\partial\tilde\beta} \tilde\m_{0,\beta} \big)
=	
	 \phi \Big(
	\sum  \langle \partial\tilde\beta, F^{-1*} \mC_{0,\gamma}\rangle  T^{E(\tilde\beta)} \tilde Y^{\partial\tilde\beta} \tilde\m_{0,\beta}
	\Big) 
	=
	\phi \Big( D_{{}_{F^{-1*}\mC_{0,\gamma}} } \tilde W \Big)}
\end{align*}
%
Thus,
$D_\theta W$ is contained in the (extension) image ideal $\phi(D\tilde W)$, so $DW \subset \phi(D\tilde W)$.
By Lemma \ref{wall_crossing_coordiante_change_lem}, the inverse $\phi^{-1}$ can be defined by the same pattern as (\ref{phi_eq}) using a ud-homotopy inverse $\mC^{-1}$ of $\mC$.
Hence, we can similarly obtain that $D\tilde W\subset \phi^{-1}(DW)$. Putting them together, we conclude $\phi(D\tilde W)=DW$.
\end{proof}

\begin{defn}
	\label{critical_point_W_defn}
We call $\mathbf y\in H^1(L; \Lambda^* )$ a \textbf{critical point} of $W$ if $\mathbf y$ is contained in the convergence domain\footnote{When we substitute $\mathbf{y}$ into $W$, the result $W(\mathbf{y})$ is a series in the field $\Lambda$. We require that it converges with respect to the non-Archimedean norm; see \cite[Page 10, Lemma 3]{BoschBook}.}
and $D_{\theta} W(\mathbf y)=0$ for any $\theta\in H^1(L)$.
\end{defn}

\begin{cor}
	\label{crit_point_inv_cor}
	$\mathbf y$ is a critical point of $W$ if and only if $\tilde{\mathbf y}:=\phi(\mathbf y)$ is a critical point of $\tilde W$.
	Thus, in Theorem \ref{Main_theorem_thesis}, the critical points of $W^\vee$ are well-defined points in the mirror analytic space $X^\vee$.
\end{cor}

\subsection{Non-vanishing of self Floer cohomology}

For the lemma below, the basic ideas already exist in the literature, e.g. \cite[Lemma 13.1]{FOOOToricOne}, \cite[Lemma 2.4.20]{FOOO_bookblue}, \cite[Theorem 2.3]{FOOO_toric_degeneration}. But, we want some slight modifications in our settings.
Let $(\HL, \m)\in \Obj\UD$ be a minimal model $A_\infty$ algebra obtained by the moduli space of holomorphic disks and by the homological perturbation. Recall that $\m_{1,0}=0$ and $\m_{2,0}$ agrees with the wedge product up to sign.

Define the $\mathbf m_k$ from $\m=(\m_{k,\beta})$ as in (\ref{CF_m_k_eq}).

\begin{lem}
	\label{key_lemma}
	Suppose the de Rham cohomology ring $H^*(L)$ is generated by $H^1(L; \mathbb Z)$. Let $\{\theta_1,\dots,\theta_n\}$ in $H^1(L;\mathbb Z)$ be generators.
	For any $x \in H^*(L)$, there exist $R_1,\dots, R_n \in H^*(L) \hat\otimes \Lambda_0[[\pi_1(L)]]$
	such that
	\[
	\mathbf m_1(x) = D_{\theta_1} W\cdot R_1+\cdots+ D_{\theta_n} W \cdot R_n \quad (\mathrm{mod} \ D \ia)
	\]
\end{lem}

By \cite[Proposition 4.1.4]{MS}, for a sufficient small number $\hbar>0$, any nontrivial holomorphic disk $u$ in $\pi_2(X,L)$ satisfies
\begin{equation}
	\label{hbar_eq}
	E(u) > \hbar >0
\end{equation}

In the first place, we prove a weaker statement as follows:

\begin{sublemma}
	\label{key_lemma_to_prove_lem}
	Assume the same conditions in Lemma \ref{key_lemma}.
	For any $x \in H^*(L)$, there exist $R_1,\dots, R_n, S\in H^*(L) \hat\otimes \Lambda_0[[\pi_1(L)]]$
	such that
	\[
	\mathbf m_1(x) =D_{\theta_1} W\cdot R_1+\cdots+ D_{\theta_n} W \cdot R_n + T^{\hbar} \mathbf m_1 (S) \quad (\mathrm{mod} \ D \ia)
	\]
\end{sublemma}

\begin{proof}
	We perform induction on $\deg x$.
	When $\deg x=0$, $x=c\cdot\one\in H^0(L)$, so $\mathbf m_1(x)=0$. When $\deg x=1$, we write $x=\theta\in H^1(L)$, and it follows from the divisor axiom of $\m$ that
	\begin{equation}
		\label{key_sublemma_m_1_theta}
		\textstyle \mathbf m_1(\theta) 
		=
		\sum_\beta T^{E(\beta)} Y^{\partial\beta} \langle \partial\beta,\theta\rangle \m_{0,\beta}= D_{\theta} W\cdot \one + D_\theta Q
		 \equiv D_{\theta} W\cdot \one \ \ (\mathrm{mod} \ D\ia)
	\end{equation}
	
	Suppose the lemma is true for degrees $\le k-1$.
	When $\deg x=k$, we may write $x=\theta\wedge x'$ for some $\theta\in H^1(L)$ by assumption, so $\deg x'=k-1$.
	Recall that $\m_{2,0} (x_1,x_2)=(-1)^{\deg x_1}x_1\wedge x_2$.
	Therefore, 
	\begin{equation}
		\label{key_sublemma_x_expansion_eq}
		x
		=\theta\wedge x'
		= \m_{2,0} (-\theta, x') 
		= \mathbf m_2(-\theta, x')
		+ \sum_{\beta\neq 0} T^{E(\beta)} Y^{\partial\beta} \m_{2,\beta}(\theta, x')
		\, \, =: \mathbf m_2(-\theta,x')+ T^{\hbar} S_0
	\end{equation}
	By (\ref{hbar_eq}), if $\m_{2,\beta}(\theta , x')\neq 0$, we must have $E(\beta) > \hbar$. Hence, $S_0\in\Lambda_0[[\pi_1(L)]]$.
	By Corollary \ref{m_123_cor},
	\begin{equation}
		\label{key_sublemma_m_1_m_2_eq}
		\mathbf m_1( \mathbf m_2(-\theta ,x')) = \mathbf m_2 (\mathbf m_1(\theta), x') + \mathbf m_2(\theta, \mathbf m_1(x'))
	\end{equation}
	Due to (\ref{key_sublemma_m_1_theta}) and the unitality of $\m$, the first term on the right side is given by
	\begin{equation}
		\label{key_sublemma_m_2_m_1_theta_i_eq}
		\mathbf m_2( \mathbf m_1(\theta), x') 
		=
		\mathbf m_2(D_\theta W\cdot \one , x')
		=
		 D_{\theta} W\cdot x'  \ \  \ (\mathrm{mod} \ D\ia)
	\end{equation}
	On the other hand, the induction hypothesis implies that there exist $R_1',\dots, R_n', S' \in H^*(L)\hat\otimes \Lambda_0[[\pi_1(L)]]$ such that
$\mathbf m_1(x')= D_{\theta_1} W\cdot R_1' +\cdots + D_{\theta_n} W\cdot R_n' + T^{\hbar} \mathbf m_1(S')$ modulo $D\ia$.
	Then, using the equations (\ref{key_sublemma_m_1_m_2_eq}, \ref{key_sublemma_m_2_m_1_theta_i_eq}) with $S'$ in place of $x'$ there, we deduce that the second term on the right side of (\ref{key_sublemma_m_1_m_2_eq}) is
	\begin{align*}
		\label{key_sublemma_induction_hypo_eq}
		\mathbf m_2(\theta, \mathbf m_1(x') ) 
		&=  D_{\theta_1}W\cdot \mathbf m_2(\theta,R'_1)+\cdots + D_{\theta_n} W\cdot \mathbf m_2(\theta, R'_n) + T^{\hbar} \mathbf m_2(\theta, \mathbf m_1(S'))  \\
		&=
		D_{\theta_1}W\cdot \mathbf m_2(\theta,R'_1)+\cdots + D_{\theta_n} W\cdot \mathbf m_2(\theta, R'_n) + T^{\hbar} \Big(\mathbf m_1(\mathbf m_2(-\theta, S')) -D_\theta W\cdot S' \Big)
		\ \ \ (\mathrm{mod} \ D\ia)	
	\end{align*}
	Putting things together, we get
	\begin{align*}
		&\mathbf m_1(x) = \mathbf m_1(\mathbf m_2(-\theta, x')+T^{\hbar} S_0) = \mathbf m_2(\mathbf m_1(\theta), x') +\mathbf m_2(\theta,\mathbf m_1(x')) +T^{\hbar} \mathbf m_1(S_0)\\
		=& 
		D_\theta W\cdot \big( x' -T^{\hbar} S'\big)+ D_{\theta_1}W\cdot  \mathbf m_2(\theta, R'_1)+\cdots + D_{\theta_n} W\cdot \mathbf m_2(\theta, R'_n)  
		+ T^{\hbar} \mathbf m_1 \big( S_0+\mathbf m_2(-\theta, S') \big) \ \ \ (\mathrm{mod} \ D\ia)
	\end{align*}
	In conclusion, if $\theta=c_1\theta_1+\cdots+c_n\theta_n$, we set $R_k:= c_k (x'-T^{\hbar}S')+\mathbf m_2(\theta, R_k')$ for $1\le k\le n$ and $S:= S_0+\mathbf m_2(-\theta, S')$. The induction is now complete.
\end{proof}

\begin{proof}[Proof of Sublemma \ref{key_lemma_to_prove_lem} $\implies$ Lemma \ref{key_lemma}]
	We note that $\mathbf m_1$ is actually $\Lambda_0[[\pi_1(L)]]$-linear by definition. Repeatedly using Sublemma \ref{key_lemma_to_prove_lem} implies that for various series $R_i^{(k)}, S^{(k)}\in H^*(L)\hat\otimes \Lambda_0[[\pi_1(L)]]$,
	\begin{align*}
		\mathbf m_1(x) 
		& =
		D_{\theta_1} W\cdot R_1^{(0)} + \cdots + D_{\theta_n} W\cdot R_n^{(0)} + T^{\hbar} \mathbf m_1( S^{(0)}) \\
		\mathbf m_1(S^{(0)})
		&=
		D_{\theta_1} W\cdot R_1^{(1)} + \cdots + D_{\theta_n} W\cdot R_n^{(1)} + T^{\hbar} \mathbf m_1( S^{(1)}) \\
		& \cdots \\
		\mathbf m_1(S^{(N-1)} )
		&=
		D_{\theta_1} W\cdot R_1^{(N)} + \cdots + D_{\theta_n} W\cdot R_n^{(N)} + T^{\hbar} \mathbf m_1(S^{(N)}) \\
		& \cdots
	\end{align*}
	modulo the ideal $D\ia$. Then, it follows that
	\[
	\mathbf m_1(x)=\sum_{i=1}^n D_{\theta_i}W\cdot \sum_{k=0}^\infty T^{k\hbar} R_i^{(k)} \ \ \ (\mathrm{mod} \ D\ia)
	\]
	Since $R_i^{(k)}\in \Lambda_0[[\pi_1(L)]]$, the summations are convergent for the adic topology.
\end{proof}


\begin{thm}\label{nonvanishing_HF_thm}
	Suppose the de Rham cohomology ring $H^*(L)$ is generated by $H^1(L; \mathbb Z)$. Under Assumption \ref{assumption_weak_MC}, $\HF(L,\m,\mathbf y)\neq 0$ if and only if $\mathbf y$ is a critical point of $W$.
\end{thm}

\begin{proof}
	The assumption tells that $\ia=D\ia=0$ (\ref{D_ideal_eq}).
	Suppose $\mathbf y$ is not a critical point: there exists $\theta\in H^1(L)$ with $D_\theta W(\mathbf y) \neq 0$. 
	Just like (\ref{key_sublemma_m_1_theta}), utilizing the divisor axiom of $\m$ yields
	\[
	\mathbf m_1^{\mathbf y} (\theta) = \sum T^{E(\beta)} \mathbf y^{\partial\beta} \m_{1,\beta}(\theta)= D_\theta W(\mathbf y) \cdot \one
	\]
	Thus, for $\theta' :=\frac{1}{D_\theta W(\mathbf y)} \theta\in H^1(L)\hat\otimes \Lambda$, we have $\mathbf m_1^{\mathbf y}(\theta')=\one$. The unit $[\one]=0$, so $\HF(L,\m,\mathbf y)=0$.
	
	Conversely, suppose $\HF(L,\m,\mathbf y)=0$. So, there exists $x_0\in H^*(L)\hat\otimes \Lambda$ such that $\mathbf m_1^{\mathbf y} (x_0)=\one$ in $H^*(L)\hat\otimes \Lambda\cong \Lambda^{\oplus \dim H^*(L)}$. 
	By Lemma \ref{key_lemma}, we apply $\mathbf m_1$ (instead of $\mathbf m_1^{\mathbf y}$) to $x_0$, so there exist $R_1,\dots, R_n\in H^*(L)\hat\otimes \Lambda[[\pi_1(L)]]$ such that the equation
	\[
	\mathbf m_1(x_0)=D_{\theta_1} W\cdot R_1+\cdots +D_{\theta_n} W \cdot R_n
	\]
	holds in $H^*(L)\hat\otimes \Lambda[[\pi_1(L)]]$.
	Further, applying the evaluation map $\mathcal E_{\mathbf y}$ in (\ref{Eva_y_eq}) to the both sides, we get
	\[
		0\neq  \one = \mathbf m_1^{\mathbf y} (x_0)=   D_{\theta_1} W (\mathbf y) \cdot R_1(\mathbf y) +\cdots +D_{\theta_n}W(\mathbf y)  \cdot R_n (\mathbf y)
	\]
in $H^*(L)\hat\otimes \Lambda$. Hence, at least one of $D_{\theta_i}W(\mathbf y)$ is nonzero, so the $\mathbf y$ cannot be a critical point of $W$.
\end{proof}

\begin{rmk}
	The above theorem has a `wall-crossing invariance' in the sense that it does not rely on the choice of $\m$.
	If $\tilde \m$ is another $A_\infty$ algebra in $\UD$ which admits an $A_\infty$ homotopy equivalence $\mC:\m\to \tilde \m$ such that $\mC_{1,0}=\id$. By Lemma \ref{wall_crossing_coordiante_change_lem}, we have an isomorphism $\phi$ defined by $\mC$ as in (\ref{phi_eq}), and the $\tilde W=\phi^* W$ is exactly the superpotential associated to $\tilde \m$. Let $\tilde{\mathbf y}=\phi (\mathbf y)$. It follows from Corollary \ref{crit_point_inv_cor} that $\tilde{\mathbf y}$ is a critical point of $\tilde W$ if and only if $\mathbf y$ is a critical point of $W$.
	Finally, by Proposition \ref{inv_Floer_self_prop_y}, $\HF(L,\m,\mathbf y) \neq 0$ if and only if $\HF(L,\tilde \m,\tilde{\mathbf y})\neq 0$.
\end{rmk}

\section{Reduced Hochschild cohomology}
\label{s_reduced_Hochschild}

\subsection{Quantum cohomology: Review}
\label{ss_quantum_cohomology}

\subsubsection{Preparation}
We will adopt the virtual language in \cite{Kuranishi_book}, but the readers may also adopt their own conventions.
By a \textit{smooth correspondence}, we mean a tuple $\mathfrak X=( \mathcal X,M,M_0,f,f_0)$ consisting of a compact metrizable space $\mathcal X$ equipped with a Kuranishi structure, two smooth manifolds $M_0$ and $M$, a strongly smooth map $f: \mathcal X\to M$ and a weakly submersive strongly smooth map $f_0:\mathcal X\to M_0$. We have a \textit{correspondence map} of degree $\dim M -\mathrm{vdim}(\mathcal X)$ between the space of differential forms:
\begin{equation}
	\label{corr_map_eq}
\Corr_{\mathfrak X} \equiv \Corr(\mathcal X; f, f_0): \Omega^*(M_0)\to \Omega^*(M)
\end{equation}
One can define the fiber product $\mathfrak X_{13}$ of two Kuranishi spaces $\mathfrak X_{12}$ and $\mathfrak X_{23}$, and the \textit{composition formula} \cite[Theorem 10.21]{Kuranishi_book} implies
\begin{equation}
	\label{corr_formula_eq}
	\Corr_{\mathfrak X_{23}} \big( \Corr_{\mathfrak X_{12}} (h_1)\times h_2 \big)= \Corr_{\mathfrak X_{13}}(h_1\times h_2)
\end{equation}
On the other hand, a smooth correspondence $\mathfrak X$ induces a boundary smooth correspondence $\partial\mathfrak X=(\partial \mathcal X, M, M_0, f|_{\partial \mathcal X}, f_0|_{\partial \mathcal X} )$. The Stokes' formulas in the Kuranishi theory \cite[Theorem 9.28]{Kuranishi_book} is
\begin{equation}
	\label{stokes_Kuranishi_eq}
	d_{M_0}\circ \Corr_{\mathfrak X} - \Corr_{\mathfrak X} \circ d_M =\Corr_{\partial\mathfrak X}
\end{equation}

\subsubsection{Definition}
We briefly recall the basic aspects of the (small) quantum cohomology. A standard reference is \cite{MS}.
Let $t$ be a formal symbol, and we define (c.f. \cite{mclean2020birational})
\begin{equation}
	\label{quantum_Nov_eq}
	\Lambda^X :=
	\left\{
	\sum_{i\in\mathbb N} b_i t^{A_i} \mid b_i   \in   \mathbb R, A_i \in H_2(X), \  0\le E(A_i)\equiv \omega(A_i)\to \infty
	\right\}
\end{equation}
We may also replace $H_2(X)\equiv H_2(X;\mathbb Z)$ by the image of the Hurewicz map $\pi_2(X)\to H_2(X)$.
We call $\Lambda^X$ the \textit{quantum Novikov ring} of $X$, since we reserve the term `Novikov ring' for the $\Lambda_0$.
Define
\begin{equation}
	\label{QH_two_defn_eq}
QH^*(X; \Lambda^X):= H^*(X)\hat\otimes \Lambda^X 
\quad \text{and}\quad 
QH^*(X; \Lambda)= H^*(X)\hat\otimes \Lambda
\end{equation}
Take the moduli space $\mathcal M_\ell(A)$ of genus-0 stable maps of homology class $A$ with $\ell$ marked points.
Let
\[
\Ev=(\Ev_0,\Ev_1,\dots,\Ev_{\ell-1}): \mathcal M_\ell(A)\to X^\ell
\]
be the evaluation maps at the marked points.
The quantum product of $QH^*(X; \Lambda^X)$ is defined as follows.
First, in the chain level, we define for $g_1, g_2 \in \Omega^*(X)$:
\[
(g_1 \pmb\ast g_2)_A =\Corr (\mathcal M_3(A) ;   (\Ev_1,\Ev_2), \Ev_0) ( g_1, g_2)
\]
and define
\[
g_1 \pmb\ast g_2 = \sum_{A\in H_2(X) }  (g_1\pmb \ast g_2)_A \cdot t^A
\]
Since the moduli $\mathcal M_\ell (A)$ has no codimension-one boundary, it follows from the Stokes' formula (\ref{stokes_Kuranishi_eq}) that the above assignment $\pmb\ast:\Omega^*(X)\otimes \Omega^*(X)\to \Omega^*(X)\hat\otimes \Lambda^X$ is a cochain map and hence induces a product map $\pmb\ast: H^*(X)\otimes H^*(X)\to H^*(X)\hat\otimes \Lambda^X$. Extending by linearity, we get the \textit{quantum product}:
\[
\pmb\ast: QH^*(X; \Lambda^X)\times QH^*(X; \Lambda^X)\to QH^*(X; \Lambda^X)
\]
Abusing the notations, over the Novikov field $\Lambda$, we also put
\[
g_1 \pmb\ast g_2 = \sum_{A\in H_2(X) }  (g_1\pmb \ast g_2)_A \cdot T^{E(A)}
\]
By the same discussion, it defines a quantum product
\[
\pmb\ast: QH^*(X;\Lambda) \times QH^*( X;\Lambda) \to QH^*(X; \Lambda)
\]
The Novikov field $\Lambda$ is algebraically closed \cite[Appendix A]{FOOOToricOne}, so we will work with the latter $QH^*(X;\Lambda)$ to study the eigenvalues. But, the  first $QH^*(X; \Lambda^X)$ is slightly more general and is still useful for our purpose.
In either cases, it is standard that the product $\pmb\ast$ gives a ring structure. The leading term $(g_1\pmb\ast g_2)_0$ of $g_1\pmb \ast g_2$ is just the standard wedge (cup) product. Besides, the constant-one function $\one_X\in H^0(X)$ is the \textit{unit} in the quantum cohomology ring \cite[Proposition 11.1.11]{MS}.

\subsubsection{An example: $\mathbb {CP}^n$} 
\label{sss_an_example_QH_CPn}

Let $\mathcal H \in H_2(\mathbb {CP}^n)$ be the standard generator represented by the line $\mathbb {CP}^1$. The cohomology $H^*({\mathbb {CP}^n})$ is a truncated polynomial ring generated by the class $c \in H^2(\mathbb {CP}^n)$ such that $c(\mathcal H)=1$ and $c^{n+1}=0$. It suffices to compute $(c^i\pmb\ast c^j)_{\ell\mathcal H}$ for $0\le i,j\le n$. When $\ell=0$, it corresponds to the usual cup product. When $\ell\ge 1$, it is standard that (see e.g. \cite[Example 11.1.12]{MS})
\[
\int_X (c^i\pmb \ast c^j)_{\ell \mathcal H} \cup c^k = \mathrm{GW}_{\ell\mathcal H, 3}^{\mathbb {CP}^n} (c^i, c^j, c^k)=
\begin{cases}
	1, \quad &\text{if} \ \ell=1 , i+j+k=2n+1\\
	0, \quad &\text{otherwise}
\end{cases}
\]
To sum up,
\[
(c^i \pmb \ast c^j)_{\ell \mathcal H}= 
\begin{cases}
	c^{i+j} 			& \text{if} \ \ell=0 ,  \  0\le i+j\le n \\
	c^{i+j-n-1} 		&\text{if} \ \ell=1, \  n+1\le i+j\le 2n \\
	0 &\text{otherwise}
\end{cases}
\]
Namely, the quantum product $\pmb \ast$ on $QH^*(X;\Lambda)$ (resp. $QH^*(X; \Lambda^X)$) is given by
\[
c^i\pmb \ast c^j =\begin{cases}
c^{i+j}, &\text{if} \ 0\le i+j\le n \\
c^{i+j-n-1} \cdot T^{E(\mathcal H)},\quad (\text{resp. }  \, \, c^{i+j-n-1} \cdot t^{\mathcal H}) \  &\text{if} \ n+1\le i+j\le 2n
\end{cases}
\]
Recall that the first Chern class is $c_1=(n+1)\cdot c$. To study the eigenvalues of $c\pmb\ast$, it would be better to work with $QH^*(X;\Lambda)$ instead of $QH^*(X;\Lambda^X)$. The matrix of $c\pmb\ast$ is then given by
\[
\begin{bmatrix}
	0 & T^{E(\mathcal H)} \\
	I_{n} & 0
\end{bmatrix}
\]
where $I_n$ denotes the $n\times n$ identity matrix. So, the eigenvalues of $c\pmb\ast $ solve the equation $\lambda^n-T^{E(\mathcal H)}=0$. Since $c_1=(n+1)c$, we finally know the eigenvalues of $c_1\pmb \ast $ are
\[
\Xi_s:=(n+1) T^{E(\frac{\mathcal H}{n+1})} e^{\frac{2\pi i s}{n+1}} \qquad \text{for} \ s\in\{0,1,\dots, n\}
\]

\subsection{Reduced Hochschild cohomology}

\subsubsection{Label grading}
\label{sss_label_grading}
Fix a $\G$-gapped $A_\infty$ algebra $(C,\m)$. We study the Hochschild cohomologies in our labeled setting \S \ref{sss_label_group}.
The $\CC\equiv \CC_\G(C)$ is a subspace of the direct product $\prod_{k,\beta} \CC_{k,\beta}$ where each $\CC_{k,\beta}$ is just a copy of $\Hom(C^{\otimes k}, C)$ with an extra label $\beta$.
Note that every $\CC_{k,\beta}$ naturally embeds into $\CC$.
Since $\deg\m_{k,\beta}=2-k-\mu(\beta)$ involves $\beta$, the operator system $\m$ in $\CC$ only has a well-defined degree modulo 2.
To settle this, we introduce the following degree on $\CC$.

\begin{defn}
	\label{label_grading_defn}
	The \textit{label degree} $|\cdot|$ of a $k$-multilinear map $\phi$ in $\CC_{k,\beta}$ is defined to be
	\[
	|\phi|= \deg'\phi +\mu(\beta) \ \in \mathbb Z
	\]
	where $\deg' \phi$ is the shifted degree as a multilinear map (see \S \ref{sss_sign_degree}). This is called the \textit{label grading}.
\end{defn}

Since $\mu(\beta)\in 2\mathbb Z$, $|\phi| \equiv \deg'\phi$ ($\mathrm{mod} \ 2$). But, if $\m=(\m_{k,\beta})$ is a gapped $A_\infty$ algebra on $C$, then $\deg \m_{k,\beta}=2-k-\mu(\beta)$, $\deg' \m_{k,\beta}=1-\mu(\beta)$, and so $|\m_{k,\beta}|=1$ is homogeneous.
For a $\G$-gapped $A_\infty$ homomorphism $\f=(\f_{k,\beta})$, one can similarly check $|\f_{k,\beta}|=0$.

%

\subsubsection{Brace operations}
\label{sss_brace}
The brace operation for $n=2$ is defined by Gerstenhaber \cite{Gerstenhaber_1963cohomology}, and the higher braces are defined by
Kadeishvili \cite{Kadeishvili_1988structure}
and 
Getzler \cite{Getzler_1993cartan}.
We can adapt the definitions to our labeled setting.
Given homogeneously-graded $\g, \f_1,\dots, \f_n\in \CC_\G(C)$, we define an element $\g\{ \f_1,\dots, \f_n\}$ in $\CC_\G(C)$ by the formula:
\begin{equation*}
	\Scale[0.92]{(\g\{\f_1,\dots,\f_n\})_{k,\beta}
	=
	\sum_{\Scale[0.75]{\substack{r_0+\cdots +r_{n}=r\\ r+t_1+\cdots +t_n=k \\ \bar\beta+\beta_1+\cdots +\beta_n=\beta}
		}
	}
	\g_{r+n,\bar\beta}
	\left(
	(\id^{\# s_0})^{\otimes r_0} \otimes (\f_1^{\# s_1})_{t_1,\beta_1}\otimes (\id^{\# s_2})^{\otimes r_2} \otimes \cdots \otimes (\id^{\# s_{n-1}})^{\otimes r_{n-1}} \otimes (\f_n^{\# s_n})_{t_n,\beta_n} \otimes (\id^{\# s_n})^{\otimes r_n}
	\right)}
\end{equation*}
where we denote $s_m=|\f_{m+1}|+\cdots +|\f_n|$ for $0\le m\le n$ (so $s_n=0$).
Moreover, note that the gappedness conditions (b) (c) in \S \ref{sss_label_group} ensure the above sum is finite, but the condition (a) is not necessary and will be omitted in \S \ref{sss_reduced_Hochschild_cohomology}.
The brace operation has degree zero for the label grading in the sense that for homogeneously-graded elements, we have
\begin{equation}
	\label{brace_degree_eq}
	|\g\{\f_1,\dots, \f_n\}|= |\g|+|\f_1|+\cdots+ |\f_n|
\end{equation}
Recall that we already adopt the brace operations to denote the Gerstenhaber product in \S \ref{sss_sign_degree}.
Also, we adopt the convention that 
\[
\g\{\}=\g
\]
By routine calculation, we can show that: (see \cite[Page 51]{getzler1994operads} and \cite[Proposition 2.3.2]{tamarkin2005ring})

\begin{prop}
	\label{brace_property_prop}
	The brace operations satisfy the following property:
	\[
	\h\{\g_1,\dots,\g_m\} \{\f_1,\dots , \f_n\}
	=
	\sum (-1)^\epsilon \
	\h \big\{\f_1,\dots, \f_{i_1}, \g_1\{\f_{i_1+1},\dots, \f_{j_1}\}, \f_{j_1+1}, \dots, \f_{i_m}, \g_m\{ \f_{i_m+1},\dots, \f_{j_m}\}, \f_{j_m+1},\dots, \f_n\big\}
	\]
	where the summation is over ${0\le i_1\le j_1\le \cdots \le i_m\le j_m\le n}$ and $\epsilon =\sum_{k=1}^m |\g_k| \cdot \sum_{\ell=1}^{i_k}|\f_\ell|$.
\end{prop}

The sign can be described briefly as follows: for each $\g_k$, we gather all $\f_j$'s that appear to the left of $\g_k$, then compute the sum of the degrees of these $\f_j$'s multiplied by the degree of $\g_k$.
A few special cases of Proposition \ref{brace_property_prop} are as follows:
\begin{align*}
	\h\{\g\}\{ \f\}
	&
	=
	\h\{ \g\{\f\}\}+\h\{\g\{\},\f\}+(-1)^{|\g||\f|} \h\{ \f,\g\{\}\} \\
	\h\{\g\} \{\f_1,\f_2\}
	&
	=
	\h\{ \g\{\f_1,\f_2\}\} + \h\{\g\{\f_1\}, \f_2 \} + (-1)^{|\g||\f_1|} \h\{\f_1,\g\{\f_2\} \} \\
	&
	+
	\h\{\g\{\}, \f_1,\f_2\} + (-1)^{|\g||\f_1|} \ \h\{ \f_1, \g\{\},\f_2\} + (-1)^{|\g|(|\f_1|+|\f_2|)} \ \h\{ \f_1, \f_2 , \g\{\} \}
\end{align*} 
where we also recall the convention that $\g\{\}=\g$.

On the other hand, we define the \textit{Gerstenhaber bracket} in our labeled setting as follows:
\begin{equation}
	\label{Gerstenhaber_bracker_eq}
	[\f,\g]:=\f \{ \g\} - (-1)^{|\f||\g| } \g \{ \f\}
\end{equation}
Due to (\ref{brace_degree_eq}), $|[\f,\g]|=|\f|+|\g|$.
One can easily check the \textit{graded skew-symmetry}:
$
	[\f,\g]=-(-1)^{|\f||\g|}[\g,\f]$.
A tedious but straightforward computation yields the \textit{graded Jacobi identity}:
\begin{equation}
	\label{Jacobi_eq}
	[\f,[\g,\h]]=[[\f, \g],\h] + (-1)^{|\f||\g|} [\g, [\f,\h]]
\end{equation}

\subsubsection{Hochschild cohomology as usual}
Given a gapped $A_\infty$ algebra $(C,\m)$, we define
\begin{equation}
	\label{Hochschild_diff_m_eq}
	\delta_\m:= [\m, \cdot ] : \CC_\G(C)\to \CC_\G(C)
\end{equation}
Clearly, $\delta_\m\circ \delta_\m=0$, and it is called a \textit{Hochschild differential}; thereafter, we may call an operator system in $\CC_\G(C)$ a Hochschild cochain.
Since $|\m|=1$, we have $|\delta_\m|=1$, i.e. $|\delta_\m\f|=|\f|+1$.
In contrast, $\deg'\delta_\m$ is only defined in $\mathbb Z_2$; this is also a reason why we introduce the label grading above.
In addition, one can use (\ref{Jacobi_eq}) to show the \textit{graded Leibniz rule}:
\begin{equation}
	\label{Leibniz_DGLA_eq}
	\delta_\m[\f,\g]=[\delta_\m \f,\g] + (-1)^{|\f|} [\f,\delta_\m\g]
\end{equation}
In summary, $(\CC(C), \delta_\m, [,])$ is a differential graded Lie algebra.
Now, the \textit{Hochschild cohomology} of a gapped $A_\infty$ algebra $(C,\m)$ is defined by 
\[HH^*(C,\m) := H^*( \CC(C), \delta_\m)\]
Next, we define the following Yoneda product on the cochain complex $\CC(C)$: (c.f. \cite{Ganatra_thesis})
\begin{equation}
	\label{cup_product_eq}
	\f\cup_\m \g= (-1)^{|\f|+1} \m\{\f,\g\}
\end{equation}
Then, we can carry out a few computations using Proposition \ref{brace_property_prop}:
\begin{align*}
	0= \m \{ \m\} \{\f ,\g\} & =  \m \{\m,\f,\g\} + (-1)^{|\f|} \m \{\f,\m,\g\} +(-1)^{|\f|+|\g|} \m\{\f,\g,\m\}  \\
	&+
	\m\{ \m\{\f\},\g\} + (-1)^{|\f|}\m\{\f,\m\{\g\}\} + \m \{\m\{\f,\g\}\}
\end{align*}
\[
\m\{\f ,\g \} \{\m\}= (-1)^{|\f|+|\g|} \m\{\m,\f,\g\} + (-1)^{|\g|} \m \{\f,\m,\g\} + \m\{\f,\g,\m\} +(-1)^{|\g|}\m\{\f\{\m\},\g\}+ \m\{\f,\g\{\m\}\}
\]
Hence, a direct calculation yields that
$
\delta_\m (\f\cup\g)=\delta_\m\f \cup \g -(-1)^{|\f|} \f\cup \delta_\m\g$, and we get a well-defined \textit{Yoneda cup product} on the Hochschild cohomology:
\[
\cup=\cup_\m: HH^*(C,\m)\times HH^*(C,\m)\to HH^*(C,\m)
\]
From (\ref{brace_degree_eq}) it follows that the label degree is $|\cup|=1$, namely, 
$|\f\cup\g|=|\f|+|\g|+1$.

\begin{prop}
	\label{cup_product_ring_prop}
	The cup product $\cup_\m$ defines a ring structure on the Hochschild cohomology $HH^*(C,\m)$.
\end{prop}

\begin{proof}
We aim to check the cup product is associative. In fact,
by Proposition \ref{brace_property_prop} again, we obtain
\begin{align*}
	0=\m\{\m\}\{\f,\g,\h\} &= \m\{\m,\f,\g,\h\}+ (-1)^{|\f|} \m\{\f,\m,\g,\h\} +(-1)^{|\f|+|\g|} \m\{\f,\g,\m,\h\} +(-1)^{|\f|+|\g|+|\h|} \m\{\f,\g,\h,\m\} \\
	&+ \m\{\m\{\f\},\g,\h\} +(-1)^{|\f|} \m \{\f,\m\{\g\}, \h\} +(-1)^{|\f|+|\g|} \m\{\f,\g,\m\{\h\}\} \\
	&+ \m\{\m\{\f,\g\},\h\} +(-1)^{|\f|} \m\{\f, \m\{\g,\h\}\} +\m\{\m\{\f,\g,\h\}\}
\end{align*}
\begin{align*}
	\m\{\f,\g,\h\}\{\m\} 
	&= (-1)^{|\f|+|\g|+|\h|} \m\{\m,\f,\g,\h\} + (-1)^{|\g|+|\h|} \m\{\f,\m,\g,\h\} +(-1)^{|\h|} \m\{\f,\g,\m,\h\}+ \m\{\f,\g,\h,\m\} \\
	&+ (-1)^{|\g|+|\h|} \m\{\f\{\m\},\g,\h\} + (-1)^{|\h|} \m\{\f,\g\{\m\},\h\} +\m\{\f,\g,\h\{\m\}\}
\end{align*}
It follows that
\begin{align*}
	(-1)^{|\g|} \big((\f\cup\g)\cup \h - \f\cup(\g\cup\h) \big) =\m\{\delta_\m \f,\g,\h\} +(-1)^{|\f|} \m\{\f,\delta_\m\g,\h\} + (-1)^{|\f|+|\g|} \m\{\f,\g,\delta_\m\h\} + \delta_\m \big(\m\{\f,\g,\h\}\big)
\end{align*}
Descending to the cohomology completes the proof.
\end{proof}

\subsubsection{Reduced Hochschild cohomology}
\label{sss_reduced_Hochschild_cohomology}

In practice, we need a slight variant.
The space $\CC_\G(C)$ consists of those operator systems $\mathfrak t'=(\mathfrak t'_{k,\beta})$. By the gappedness condition (a) in \S \ref{sss_label_group}, we have to assume $\mathfrak t'_{0,0}=0$. 
But, we must allow nontrivial component for $(k,\beta)=(0,0)$ later.
Thus, we define
\begin{equation}
	\label{rCC_eq}
\rCC (C)\subset \CC_{0,0}(C)\times \CC_\G(C)
\end{equation}
to be the space of the operator systems $\mathfrak t=(\mathfrak t_{k,\beta})_{k\in\mathbb N,\beta\in\G}$ such that $\mathfrak t-\mathfrak t_{0,0}\equiv (\mathfrak t_{k,\beta})_{(k,\beta)\neq (0,0)}$ is contained in $\CC_\G(C)$ and
\begin{equation}
	\label{rCC_condition_eq}
	\mathfrak t_{k,\beta}(\dots, \one,\dots)=0 \ \ \quad (\forall  \ k\ge 1, \ \forall \ \beta ) 
\end{equation}
We keep using the label grading on $\CC_\G(C)$, and we further define the grading on $C\equiv \CC_{0,0}(C)$ by
\begin{equation}
	\label{rCC_grading_00_eq}
	|\mathfrak t_{0,0}|=\deg' \mathfrak t_{0,0}
\end{equation}

\begin{rmk}
	Algebraically, the above version is needed for the unital ring structures.
	Geometrically, this is also necessary, because the operator $\q_{\ell,0,0}$ (for the moduli space of constant maps with one boundary and $\ell$ interior markings as in \S \ref{ss_from_QH_to_rHH}) should correspond to a term with label $(k,\beta)=(0,0)$.
\end{rmk}

	\begin{lem}
		\label{operations_on_rCC}
		The brace operations can be defined on $\rCC(C)$.
	\end{lem}

\begin{proof}

	Suppose $\g,\f_1,\dots,\f_n\in\rCC(C)$ for $n\ge 1$.
	We define $\h:=\g\{\f_1,\dots, \f_n\}$ exactly as in \S \ref{sss_brace}, and there will be at most finite extra terms from the $\CC_{0,0}$.
Depending on the place of $\one$, we decompose
\[
\h_{k,\beta}(\dots ,\one,\dots) = \sum \pm \g_{s,\bar\beta} ( \dots, (\f_i)_{t_i,\beta_i}(\dots,\one,\dots),\dots) + \sum \pm \g_{s,\bar\beta} (\dots, (\f_i)_{t_i,\beta_i} ,\dots,\one,\dots, (\f_{i+1})_{t_{i+1},\beta_{i+1}},\dots) 
\]
Hence, the condition (\ref{rCC_condition_eq}) of $\h$ just follows from that of $\g$ and $\f_i$'s.
\end{proof}

In particular, we can define $\g\{\f\}$ and $[\g,\f]=\g\{\f\} \pm \f\{ \g\}$ on $\rCC$.
By comparison, $\g\diamond\f$ is not defined for $\f,\g\in \rCC(C)$, since there may be infinite extra terms in (\ref{composition_Gerstenhaber_eq}) from the $\CC_{0,0}$.
We introduce $\rCC$ mainly for the sake of closed-open operators, whereas we must still use $\CC$ rather than $\rCC$ to develop the homotopy theory of $A_\infty$ algebras for the category $\UD$, e.g. the Whitehead theorem \ref{whitehead_thm}. 
Furthermore, the condition (\ref{rCC_condition_eq}) does not exactly match the unitality of an $A_\infty$ algebra $\m$ or that of an $A_\infty$ homomorphism $\f$. Indeed, we only have $\m-\m_{2,0}\in \rCC(C)$ and $\f-\f_{1,0} \in \rCC(C)$ in general.
But, according to our trial-and-error, the condition (\ref{rCC_condition_eq}) is indeed the correct one as suggested below:

\begin{lem}
	\label{rHH_defn_previous_lem}
	Suppose $(C,\m)\in \Obj \UD$. Then,
	
	\begin{enumerate}
		\item the Hochschild differential $\delta_\m=[\m,\cdot]$ still gives a differential map on $\rCC(C)$;
\item the Yoneda cup product $\cup_\m$ induces a map $\rCC(C)\times \rCC(C)\to \rCC(C)$.
\end{enumerate}

\end{lem}

\begin{proof}
We only know $\m-\m_{2,0}\in \rCC(C)$, but we recall that $\m_{2,0}(\one, x)= (-1)^{\deg' x+1} \m_{2,0}(x,\one)=x$.

\textbf{(1).} Given $\f \in \rCC(C)$, it suffices to show $\delta_\m \f\in \rCC(C)$.
By Lemma \ref{operations_on_rCC}, we have $[\m-\m_{2,0},\f]\in\rCC(C)$.
	Thus, it suffices to check the condition (\ref{rCC_condition_eq}) for $\delta_{\m_{2,0}}(\f)=[\m_{2,0}, \f]$.
	Specifically, fix $k\ge 1$ and $1\le i\le k$; we aim to show 
	\[
	\Delta:=([\m_{2,0}, \f])_{k,\beta}(x_1,\dots, x_{i-1}, \one, x_i,\dots, x_{k-1})
	\]
	vanishes all the time. We check it by cases. Recall that we denote $x^\#=(-1)^{\deg' x} x$ (\ref{sign_operator_eq}). 
	
	\begin{itemize}
		\item [$\ast$] 
	When $ i\neq 1$ or $k$, using the condition (\ref{rCC_condition_eq}) of $\f$ deduces that
	\begin{align*}
	\Delta = & -(-1)^{\deg' \f} \ \f_{k-1,\beta} (x_1^\#,\dots, x_{i-2}^\#, \m_{2,0} (x_{i-1},\one), x_i, x_{i+1}, \dots, x_{k-1}) \\
	&-(-1)^{\deg' \f} \ \f_{k-1,\beta} (x_1^\#,\dots, x_{i-2}^\#, x_{i-1}^\# , \m_{2,0} (\one, x_i), x_{i+1}, \dots, x_{k-1})
	\end{align*}
	This vanishes as $\m_{2,0}(x,\one)=(-1)^{\deg' x+1} x=-x^\#$ and $\m_{2,0}(\one, x)=x$. 
	
	\item [$\ast$] 
	When $i=1$, using the condition (\ref{rCC_condition_eq}) of $\f$ deduces that
	\begin{align*}
		\Delta =	& \m_{2,0} ( \one^{\# \deg' \f}, \f_{k-1,\beta}(x_1,\dots, x_{k-1})) -(-1)^{\deg'\f} \  \f_{k-1,\beta} (\m_{2,0}(\one, x_1),x_2,\dots, x_{k-1})
	\end{align*}
Note that $\m_{2,0}(\one, x)=x$ and $\one^{\# \deg' \f}=(-1)^{\deg' \f} \one$. Thus, we also get $\Delta=0$.
	
	\item [$\ast$]
	When $i=k$, we can similarly compute the sign. Using the condition (\ref{rCC_condition_eq}) of $\f$ deduces that
			\begin{align*}
				\Delta &= \m_{2,0} (\f_{k-1,\beta}(x_1,\dots, x_{k-1}), \one) -(-1)^{\deg' \f} \f_{k-1,\beta} (x^\#_1,\dots, x^\#_{k-2}, \m_{2,0}( x_{k-1},\one)) \\
				&= 
				\Scale[0.8]{\big( (-1)^{\deg' \f +\deg' x_1+\cdots +\deg' x_{k-1}+1} - (-1)^{\deg' \f+\deg' x_1+\cdots +\deg'x_{k-2} + (\deg' x_{k-1}+1)} \big)}
				\cdot \f_{k-1,\beta}(x_1,\dots, x_{k-1})=0
			\end{align*}
	\end{itemize}
	
\textbf{(2).} Given $\f,\g\in \rCC(C)$, we aim to show $ \m\{\f,\g\} \in \rCC(C)$.
Since $\m-\m_{2,0}\in \rCC(C)$, applying Lemma \ref{operations_on_rCC} first implies that $(\m-\m_{2,0}) \{\f,\g\}\in\rCC(C)$.
However, the conditions (\ref{rCC_condition_eq}) of $\f$ and $\g$ also imply that
$
\big((\m_{2,0})\{\f,\g\} \big)_{k,\beta}( x_1,\dots, x_{i-1},\one,x_i,\dots, x_{k-1})
=0
$ for any $k\ge 1$ and $1\le i\le k$.
Hence, $(\m_{2,0})\{\f,\g\}\in \rCC(C)$.
So, the proof of (2) is now complete.
 (Remark that actually one can show that $\m\{\f_1,\dots,\f_n\}\in\rCC(C)$ for any $n\ge 2$ and $\f_1,\dots, \f_n\in\rCC(C)$.)
\end{proof}

Combining Lemma \ref{operations_on_rCC} and \ref{rHH_defn_previous_lem} with the proof of Proposition \ref{cup_product_ring_prop} yields the following valid definition:

\begin{defn}
	\label{rHH_defn}
	The \textbf{reduced Hochschild cohomology} for $(C,\m)\in\Obj\UD$ is defined by
	\[
	\rHH(C,\m) := H^*(\rCC(C),\delta_\m)
	\]
	Moreover, the cup product $\cup_\m$ also gives rise to a ring structure on $\rHH(C,\m)$.
\end{defn}

\subsubsection{Unitality}
A major advantage of the reduced Hochschild cohomology is the unitality:

\begin{lem}
	\label{unit_rHH_lem}
	Fix $(C,\m) \in \Obj\UD$. Then, $\rHH(C,\m)$ is a unital ring such that the \textbf{unit} is given by
	\begin{equation}
		\label{unit_rHH_eq}
	\mathfrak e =(\mathfrak e_{k,\beta})_{k\in\mathbb N, \beta\in \G} \in \rCC(C)
	\end{equation}
where $\mathfrak e_{0,0}=\one$ and $\mathfrak e_{k,\beta}=0$ for any $(k,\beta)\neq (0,0)$. Moreover, its label degree is $|\mathfrak e|=-1$.
\end{lem}

\begin{proof}
Clearly, $\delta_\m(\mathfrak e)=0$. By (\ref{rCC_grading_00_eq}), the degree of $\mathfrak e$ is
$
|\mathfrak e|=\deg' \one  =  -1
$.
Fix $[\f]\in \rHH(C,\m)$. 
Using the unitality of $\m$ implies $(\mathfrak e\cup_\m\f)_{k,\beta}= \m_{2,0} (\mathfrak e_{0,0}, \f_{k,\beta})=\f_{k,\beta}$
Similarly, we can check $\f\cup_\m \mathfrak e=\f$.
\end{proof}


\subsubsection{$\Lambda^X$-module}
\label{sss_Lambda_X_module}
From now on, we assume the label group is $\G=\pi_2(X,L)$.
The quantum Novikov ring $\Lambda^X$ in (\ref{quantum_Nov_eq}) acts on $\rCC$ by $(t^A\cdot \f)_{k,\beta}= \f_{k,\beta-A}$ for any $A\in \pi_2(X)$. That is, the $\rCC$ has a natural $\Lambda^X$-module structure.
Specifically, if $\f$ is supported in a single component $\CC_{k,\gamma}$,
the $t^A\cdot \f$ is simply the same multilinear map for a different label, supported in $\CC_{k,\gamma+A}$ instead. This will not violate (\ref{rCC_condition_eq}) or any gappedness conditions since $E(A)\ge 0$. Notice that the label degree (\S \ref{sss_label_grading}) will be changed by
\begin{equation}
	\label{tA_degree_label_eq}
	|t^A\cdot \f|=|\f|+ 2c_1(A) 
\end{equation}
since $\mu(\gamma+A)=\mu(\gamma)+2c_1(A)$. Particularly, we always have $(-1)^{|t^A\cdot \f|}=(-1)^{|\f|}$.

\begin{prop}
	The $\Lambda^X$-action is compatible with the brace operation, namely, $t^A\cdot \big(\g\{\f_1,\dots,\f_n\}\big) = (t^A\cdot \g)\{\f_1,\dots,\f_n\}=\g\{\f_1,\dots, (t^A\cdot \f_i),\dots,\f_n\}$ for $1\le i\le n$.
\end{prop}

\begin{proof}
	We only consider $n=1$, and the others are similar. By linearity, we may assume $\g$ and $\f$ are supported in $\CC_{k+1,\gamma}$ and $\CC_{\ell,\eta}$ respectively, and we may also assume $|\f|=p$ and $|\g|=q$.
	Without the labels, we simply have $\g\{\f\}(x_1,\dots, x_{k+\ell})=\sum_i \g( x^{\# p}_1,\dots, x_i^{\# p}, \f(x_{i+1},\dots, x_{i+\ell}),\dots, x_{k+\ell})$. Then, one can easily check that $t^A\cdot (\g\{\f\})$, $\g\{ t^A\cdot \f\}$, and $(t^A\cdot \g)\{\f\}$ are the same multilinear map as above, supported in $\CC_{k+\ell,\gamma+\eta+A}$ with the same label. The signs are also the same thanks to (\ref{tA_degree_label_eq}).
\end{proof}

\begin{cor}
	$\delta_\m (t^A\cdot \f)= t^A\cdot \delta_\m \f$.
\end{cor}

\begin{cor}
	\label{Lambda_X_module_rHH_cor}
	The (reduced) Hochschild cohomology is a $\Lambda^X$-module.
\end{cor}

\subsection{From quantum cohomology to reduced Hochschild cohomology}
\label{ss_from_QH_to_rHH}
We study the moduli spaces of pseudo-holomorphic curves.
Let $L$ be an oriented relatively-spin Lagrangian submanifold in a symplectic manifold $X$. Let $J$ be an $\omega$-tame almost complex structure.
Given $\ell,k\in\mathbb N$ and $\beta\in\pi_2(X,L)$, we denote by $\mathcal M_{\ell, k+1,\beta}(L,\beta)$ the moduli space of all equivalence classes $[\uu,\mathbf z^+, \mathbf z]$ of $J$-holomorphic stable maps $(\uu,\mathbf z^+,\mathbf z)$ of genus-0 with one boundary component in $L$ such that $[\uu]=\beta$ and $\mathbf z^+=(z^+_1,\dots, z^+_\ell)$, $\mathbf z=(z_0,z_1,\dots, z_k)$ are the interior and boundary marked points respectively. Moreover, the moduli space is simply not defined when $(\ell,k,\beta)\neq (0,0,0), (0,1,0)$ for the sake of the stability. Note that with the notation in \S \ref{sss_moduli_space_A_infinity}, we have 
\begin{equation}
	\label{moduli_ell=0_eq}
	\mathcal M_{0, k+1,\beta}(J,L)=\mathcal M_{k+1,\beta}(J,L)
\end{equation}
The moduli space admits natural evaluation maps corresponding to the marked points:
\[
(\ev^+,\ev_0,\ev)=(\ev_1^+,\dots, \ev_\ell^+ ;\ev_0; \ev_1, \dots, \ev_k) : \mathcal M_{\ell,k+1,\beta}(J,L)\to X^\ell\times L^{k+1}
\]
where $\ev^+_i([\uu, \mathbf z^+,\mathbf z])= \uu (z^+_i)$ and $\ev_i([\uu, \mathbf z^+,\mathbf z])= \uu (z_i)$.
The codimension-one boundary of $\mathcal M_{\ell,k+1,\beta}(J,L)$ in the sense of Kuranishi structure and smooth correspondence (\ref{corr_map_eq}) is given by the following union of the fiber products:
\begin{equation}
	\label{moduli_boundary_eq}
\partial \mathcal M_{\ell,k+1,\beta}(J,L) = \bigcup \mathcal M_{\ell_1,k_1+1,\beta_1}(J,L)  \ \ \ {}_{\ev_0}\times_{\ev_i} \mathcal M_{\ell_2,k_2+1,\beta_2}(J,L)
\end{equation}
where the union is taken over all the $(\ell_1,\ell_2)$-shuffles of $\{1,\dots,\ell\}$ with $\ell_1+\ell_2=\ell$, $k_1+k_2=k$, $\beta_1+\beta_2=\beta$, and $1\le i\le k_2$.
Here a \textit{$(\ell_1,\ell_2)$-shuffle} is a permutation $(\mu_1,\dots,\mu_{\ell_1},\nu_1,\dots,\nu_{\ell_2})$ of $(1,2,\dots, \ell)$ such that $\{ \mu_1<\cdots<\mu_{\ell_1} \}$ and $\{ \nu_1<\cdots<\nu_{\ell_2}\}$.

Consider the smooth correspondence $\mathfrak X=( \mathcal X,M,M_0,f,f_0)$ given by $M_0= L$, $M= X^\ell \times L^{k}$, $f_0=\ev_0$, $f= (\ev^+, \ev_1,\dots, \ev_k)$, and $\mathcal X=\mathcal M_{\ell,k+1,\beta}(J,L)$.
For $(\ell,\beta)\neq (0,0)$ and for $g_i\in \Omega^*(X)$ and $h_i\in \OL$, we define
\[
\q^{J,L}_{\ell,k,\beta} (g_1,\dots, g_\ell, h_1,\dots, h_k) = \tfrac{1}{\ell !} \Corr (\mathcal M_{\ell,k+1,\beta}(J,L); (\ev^+,\ev),  \ev_0 ) ( g_1\wedge \cdots g_\ell \wedge h_1\wedge \cdots \wedge h_k)
\]
by (\ref{corr_map_eq}). Alternatively, it may be also instructive to adopt the following notation:
\[
\q^{J,L}_{\ell,k,\beta}(g_1,\dots, g_\ell ; h_1,\dots, h_k)= \pm \tfrac{1}{\ell !} (\ev_0)_!  (  \ev^{+*}_1 g_1\wedge\cdots \ev^{+*}_\ell g_\ell \wedge \ev^*_1 h_1\wedge\cdots \wedge \ev^*_k h_k)
\]
If $(\ell,\beta)=(0,0)$, then we exceptionally define $\q^{J,L}_{0,k,0}=0$ for $k\neq 1,2$, $\q^{J,L}_{0,1,0}(h) =d h$, and $\q^{J,L}_{0,2,0}(h_1,h_2)= (-1)^{\deg h_1} h_1\wedge h_2$.

\begin{prop}
	\label{operator_q_prop}
	The operators $\q_{\ell,k,\beta}:=\q^{J,L}_{\ell,k,\beta}$ satisfy the following properties:
	
	\begin{enumerate}[label=(\roman*)]
		\item   $\deg \q_{\ell,k,\beta}=2-2\ell-k-\mu(\beta)$. The $\q_{\ell,k,\beta}\neq 0$ only if $E(\beta)\ge 0$.
		\item  For any permutation $\sigma\in \mathcal S_\ell$, we have
		\[
		\q_{\ell,k,\beta}(g_1,\dots, g_\ell; h_1,\dots, h_k)=(-1)^a \q_{\ell,k,\beta}( g_{\sigma(1)}, \dots, g_{\sigma(\ell)} ; h_1,\dots, h_k)
		\]
		where $a=\sum_{i<j; \sigma(i)>\sigma(j)} \deg g_i \deg g_j$.
		\item  When $\ell=0$, we have $\q_{0,k,\beta}=\check\m^{J,L}_{k,\beta}$ (\S \ref{sss_moduli_space_A_infinity}).
		\item When $\ell\neq 0$ and $\beta=0$, we have $\q_{\ell,k,0}=0$ except $\q_{1,0,0}=\iota^*: \Omega^*(X)\to \OL$ where $\iota:L\to X$ is the inclusion.
		\item  Let $\one \in \Omega^0(L)$ be the constant-one function. Then, $\q_{\ell,k,\beta}(g_1,\dots, g_\ell; h_1,\dots, h_{i-1}, \one, h_i,\dots, h_k)=0$ except
		$
		\q_{0,2,0}(\one, h)=(-1)^{\deg h} \q_{0,2, 0} (h,\one)=h
		$.
		\item Given $g_1,\dots, g_\ell\in \Omega^*(X)$ and $h_1,\dots, h_k\in \OL$,
		\begin{align*}
			& \sum_{1\le i\le \ell} (-1)^{\dagger} \q_{\ell,k,\beta}( g_1,\dots, dg_i,\dots, g_\ell; h_1,\dots, h_k) + \sum_{\substack{ 
					\Scale[0.6]{\beta_1+\beta_2=\beta} \\ \Scale[0.6]{\lambda+\mu+\nu=k} \\ 
					\Scale[0.6]{\ell_1+\ell_2=\ell} }
				}
			\sum_{ \substack{ \Scale[0.6]{(\ell_1,\ell_2)\text{-shuffles}} \\ \Scale[0.6]{(\mu_1,\dots,\mu_{\ell_1}, \nu_1,\dots,\nu_{\ell_2})}} } \\
			& (-1)^\ast \q_{\ell_1,\lambda+\mu+1,\beta_1} ( g_{\mu_1},\dots, g_{\mu_{\ell_1}} ; h_1,\dots, h_\lambda, \q_{\ell_2,\nu,\beta_2} ( g_{\nu_1},\dots, g_{\nu_{\ell_2}}; h_{\lambda+1},\dots, h_{\lambda+ \nu} ), h_{\lambda+\nu+1},\dots, h_k) =0
		\end{align*}
	where $\dagger= \sum_{j=1}^{i-1} \deg g_j$ and $\ast= \sum_{j=1}^{\ell_1} \deg g_{\mu_j} + \sum_{j=1}^\lambda (\deg h_j-1) \cdot \sum_{j=1}^{\ell_2} (\deg g_{\nu_j}-1)$.
	\end{enumerate}

\end{prop}

\begin{proof}[Sketch of proof]
	Nothing is new here, and we only give rough explanations; see the various literature for more details: \cite{Solomon_Diff_survey}, \cite[Theorem 3.8.32]{FOOOBookOne}, \cite[\S 2.3]{FOOO_bookblue}, and \cite[Chapter 17]{FOOOSpectral}.
	First, the items (i) (ii) (iii) should be clear.
	We explain (iv). If $\beta=0$, the moduli spaces consist of constant maps into $L$. Hence, $\mathcal M_{\ell,k+1,\beta} (J,L)\cong L\times \mathcal M_{\ell,k+1}$, and the evaluation maps just become the projection maps $\ev_i \equiv \pr : L\times \mathcal M_{\ell,k+1} \to L$ or the projection maps pre-composed with the inclusion $\ev^+_i \equiv  \iota \circ \pr : L\times \mathcal M_{\ell,k+1} \to L \to X$. Then, $\q_{\ell,k,0}(g_1,\dots, g_\ell; h_1,\dots, h_k)=
	\pm \pr_! \pr^* \big( \iota^*g_1 \wedge \cdots \iota^* g_\ell \wedge h_1\wedge \cdots \wedge h_k \big)$.
	Note that $\pr_! \pr^* \neq 0$ only if the fibers of $\pr$ has dimension zero, namely, $2\ell+k-2=\dim \mathcal M_{\ell,k+1}=0$. When $\ell\neq 0$, the only possibility is that $\ell=1, k=0$. In this case, $\pr=\id_L$, and $\q_{1,0,0}=\iota^*$.
	Next, we can check the item (v) by studying the forgetful map of the marked point of $\one$ in a similar way, c.f. \cite[\S 6]{Yuan_I_FamilyFloer}.
	Eventually, we explain (vi). For the smooth correspondence $\mathfrak X$ as above, applying the Stokes' formulas (\ref{stokes_Kuranishi_eq}) to (\ref{moduli_boundary_eq}) implies that:
$d \circ \q_{\ell, k,\beta}(g_1,\dots, g_\ell; h_1,\dots h_k)  \pm \sum_i \q_{\ell,k,\beta}(g_1,\dots, dg_i,\dots, g_\ell; h_1,\dots, h_k) \pm \sum_i \q_{\ell,k,\beta}(g_1,\dots, g_\ell; h_1,\dots, dh_i,\dots, h_k) 
	= \pm \sum_{(\ell_1,\beta_1)\neq 0\neq (\ell_2,\beta_2)}  \\
	\q_{\ell_1,\lambda+\mu+1,\beta_1} (g_{\mu_1},\dots, g_{\mu_{\ell_1}}; h_1\dots, h_i, \q_{\ell_2,\nu,\beta_2} (g_{\nu_1},\dots, g_{\nu_{\ell_2}} ; h_{\lambda+1},\dots, h_{\lambda+\nu}), h_{\lambda+\nu+1},\dots, h_k)$.
Now, as $\q_{0,1,0}=\check\m_{1,0}^{J,L}=d$, the equation in (vi) then follows.
\end{proof}

For our purpose, we focus on the case $\ell=1$, then the equation in (vi) becomes the following:
\begin{align}
	\Scale[0.85]{\q_{1,k,\beta}(dg ;h_1,\dots, h_k) }
	&\Scale[0.85]{\displaystyle + \sum_{\substack{\beta_1+\beta_2=\beta  \\ \lambda+\mu+\nu=k}} (-1)^{\deg g +\sum_{j=1}^{\lambda}(\deg h_j-1)}
	\q_{1,\lambda+\nu+1,\beta_1} (g; h_1,\dots, h_\lambda,  \check \m^{J,L}_{\nu,\beta_2} (h_{\lambda+1},\dots, h_{\lambda+\nu}), h_{\lambda+\nu+1},\dots, h_k)  } \notag \\
	&\Scale[0.85]{\displaystyle + \sum_{\substack{\beta_1+\beta_2=\beta  \\ \lambda+\mu+\nu=k}} (-1)^{\sum_{j=1}^{\lambda} (\deg h_j-1) \cdot (\deg g-1)}
	\check \m^{J,L}_{\lambda+\mu+1,\beta_1} (h_1,\dots, h_\lambda, \q_{1,\nu,\beta_2} (g; h_{\lambda+1},\dots, h_{\lambda+\nu}),h_{\lambda+\nu+1},\dots, h_k)=0}
\label{q_1k}
\end{align}

Accordingly, we can define a linear map
\begin{equation}
	\label{hat_q_eq}
	\hat \q:  \Omega^*(X) \to \rCC(\OL)
\end{equation}
by setting $\hat \q(g )_{k,\beta}= \q_{1,k,\beta}(g; \cdots)$.
Here one can directly check that the condition (\ref{rCC_condition_eq}) by Proposition \ref{operator_q_prop} (v).
Moreover, for the label grading (\S \ref{sss_label_grading}), the item (i) implies
\begin{equation}
	\label{degree_hat_q_eq}
|\hat\q(g)|=|\q_{1,k,\beta}(g)|=\deg g -1 
\end{equation}
By \S \ref{sss_Lambda_X_module}, we can extend it linearly and thus obtain a $\Lambda^X$-linear map $\Omega^*(X)\hat\otimes \Lambda^X\to \rCC(\OL)$, still denoted by $\hat \q$.
Then, the equation (\ref{q_1k}) exactly implies
\[
\hat \q (dg) + \delta_{\check\m^{J,L}}  \big(\hat \q(g)\big)=0
\]
Consequently, the $\hat \q$ induces a $\Lambda^X$-linear map to the reduced Hochschild cohomology of $\check \m^{J,L}$:
\begin{equation}
	\label{[hat_q]_eq}
	[\hat \q]:  QH^*(X; \Lambda^X)\equiv H^*(X)\hat\otimes \Lambda^X   \to \rHH(\OL ,\check \m^{J,L})
\end{equation}

	Recall that the ud-homotopy theory of $A_\infty$ algebras for $\UD$ must work on $\CC$ in (\ref{CC_eq}). 
	In contrast, as $\hat\q(g)_{0,0}=\iota^*g$ can be nonzero by Proposition \ref{operator_q_prop} (iv),
	we must work on $\rCC$ in (\ref{rCC_eq}) here.

The reduced Hochschild cohomology ring $\rHH(\OL,\check \m^{J,L})$ (Definition \ref{rHH_defn}) admits a unital ring structure. The unit is the $\mathfrak e$ that is supported in $\CC_{0,0}$ (Lemma \ref{unit_rHH_lem}), and its label degree is $-1$. 
Meanwhile, $QH^*(X; \Lambda^X)\equiv H^*(X)\hat\otimes \Lambda^X$ has a unital ring structure with the unit $\one_X$ of degree $0$.

Our goal is to show the $\hat\q$ gives a unital ring homomorphism.
But, as we mentioned in the introduction, the main complication is not Theorem \ref{[hat_q]_ring_homo_thm} below but that the moduli space geometry \textit{cannot} make a ring homomorphism into the reduced Hochschild cohomology of the minimal model $A_\infty$ algebras. We will resolve this issue soon later by using several techniques about $\UD$.

\begin{thm}
	\label{[hat_q]_ring_homo_thm}
	The map $[\hat \q]$ is a \textit{unital} ring homomorphism of degree $-1$ such that $\hat\q (\one_X)= \mathfrak e$.
\end{thm}

\begin{proof}[Sketch of proof]
The argument here is standard and well-known to many experts, and we simply rewrite it in our notations and languages. 
But, we will moreover discover that our notion of the reduced Hochschild cohomology fits perfectly as well. In short, Proposition \ref{operator_q_prop} (v) corresponds to (\ref{rCC_condition_eq}).

First off, the degree matches by (\ref{degree_hat_q_eq}).
Next, we consider the forgetful map $\forget: \mathcal M_{2,k+1,\beta}(J,L)\to \mathcal M_{2,1}$ which forgets the maps and all the incoming boundary marked points and then shrinks the resulting unstable domain components if any.
Given $\mathfrak x\in \mathcal M_{2,1}$, we denote the fiber of $\forget$ over $\mathfrak x$ by 
\[
\mathcal M^{\mathfrak x}_{2,k+1,\beta} (J,L) := \forget^{-1}(\mathfrak x)
\]

Notice that the $\mathcal M_{2,1}$ is homeomorphic to the closed unit disk $\mathbb D$ and has a stratification (see \cite[Figure 2.6.1]{FOOO_bookblue}) which consists of two copies of $(-1,1)$, $\mathrm{int}(\mathbb D)\setminus\{0\}$, and three distinguished moduli points $[\Sigma_0]$, $[\Sigma_{12}]$, $[\Sigma_{21}]$.
Geometrically, we may assume one interior marking is at $0\in\mathbb D$, and the other is denoted by $z$. Then, the three moduli points corresponding to the situations when $z\to 0, \pm 1$ respectively.
To be specific, we have the following fiber products:
\begin{equation}
\label{M_2_k+1_Sigma_012_eq}
\mathcal M^{[\Sigma_0]}_{2,k+1,\beta}(J,L)= \bigcup_{\gamma+A=\beta} \mathcal M_3(A)  \ \ \ {}_{{}_{\Ev_0}}\times_{\Scale[0.66]{\ev_1^+}} \mathcal M_{1,k+1,\gamma}(J,L)
\end{equation}
\[
\mathcal M^{[\Sigma_{12}]}_{2,k+1,\beta}(J,L) =  \bigcup_{\substack{1\le i< j\le m+2 \\ 
m'+m''+m=k \\ \gamma'+\gamma''+\gamma=\beta
} }
\mathcal M_{1, m'+1,\gamma'}(J,L)  \ \ \ {}_{\ev_0}\times_{\ev_i} \ \mathcal M_{0,m+3,\gamma}(J,L) \ \ \ {}_{\ev_j}\times_{\ev_0} \ \mathcal M_{1, m''+1,\gamma''} (J,L)
\]
Now, we take a path $\mathfrak y: \oi \to \mathcal M_{2,1}$ from $[\Sigma_0]$ to $[\Sigma_{12}]$ and define
\[
\mathcal M^{\mathfrak y}_{2,k+1,\beta}(J,L ) := \bigcup_{t\in\oi} \{t \} \times \mathcal M^{\mathfrak y(t)}_{2,k+1,\beta}(J,L )
\]
Its codimensional-one boundary is given by
the union of $\mathcal M^{[\Sigma_0]}_{2,k+1,\beta}(J,L)$, $\mathcal M^{[\Sigma_{12}]}_{2,k+1,\beta}(J,L)$, and
\begin{equation}
	\label{moduli_M_2_k+1_path_boundary_eq}
	\bigcup_{\substack{m'+m''=k \\ \beta'+\beta''=\beta}}
	 \mathcal M^{\mathfrak y}_{2,m'+2,\beta'} (J,L)  \ \ \  {}_{\ev_i}\times_{\ev_0} \mathcal M_{0,m''+1,\beta''}(J,L)
\end{equation}
\begin{equation*}
	\bigcup_{\substack{m'+m''=k \\ \beta'+\beta''=\beta}}
	 \mathcal M_{0,m'+2,\beta'} (J,L)  \ \ \  {}_{\ev_i}\times_{\ev_0} \mathcal M^{\mathfrak y}_{2,m''+1,\beta''}(J,L)
\end{equation*}
\begin{figure}[h]
	\centering
	\includegraphics[width=9cm]{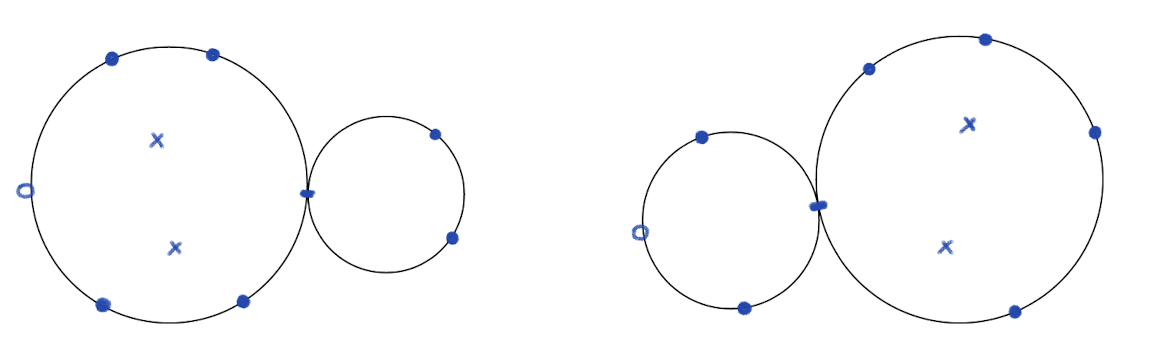}
	\caption{\footnotesize The figure for Equation (\ref{moduli_M_2_k+1_path_boundary_eq}) where the open circle indicates the outgoing marked point and the filled circles are the other boundary marked points.}
\end{figure}

Given $g_1,g_2\in \Omega^*(X)$ and $h_1,\dots, h_k\in \Omega^*(L)$, we define
\[
\mathfrak Q_{2,k,\beta}(g_1,g_2 ; h_1,\dots, h_k)
:=
\Corr(\mathcal M^{\mathfrak y}_{2,k+1,\beta}(J,L); (\ev^+_1,\ev^+_2,\ev),\ev_0 ) (g_1,g_2,h_1,\dots, h_k)
\]
by (\ref{corr_map_eq}).
Applying the Stokes' formula (\ref{stokes_Kuranishi_eq}) yields that
\begin{align*}
& \mathfrak Q_{2,k,\beta}(dg_1,g_2; h_1,\dots, h_k) + \mathfrak Q_{2,k,\beta}(g_1,dg_2; h_1,\dots, h_k)   \\
&=
\sum_{\beta'+\beta''=\beta} \pm \mathfrak Q_{2,m'+1,\beta'}(g_1,g_2; h_1,\dots, h_\lambda, \check \m^{J,L}_{m'',\beta''} (h_{\lambda+1},\dots, h_{\lambda+\nu} ), h_{\lambda+\nu+1},\dots, h_k) \\
&+
\sum_{\beta'+\beta''=\beta}  \pm \check \m^{J,L}_{m'+1,\beta'} (h_1,\dots, h_\lambda, \mathfrak Q_{2, m'', \beta''} ( g_1, g_2 ; h_{\lambda+1},\dots, h_{\lambda+\nu}) , h_{\lambda+\nu+1},\dots, h_k)  \\
&+
\sum_{\gamma+A=\beta} \pm \q_{1,k,\gamma} ( (g_1\pmb\ast g_2 )_A ; h_1,\dots, h_k)  \\
&+
\sum_{\gamma'+\gamma''+\gamma=\beta} \pm \ \check\m^{J,L}_{m+2,\gamma} ( h_1,\dots, h_{i-1}, \q_{1,m',\gamma'}( g_1 ; h_i,\dots, h_{i+m'-1}),\dots, \q_{1,m'',\gamma''}(g_2; h_{j+m'}, \dots ), \dots, h_k) 
\end{align*}
In reality, the left side together with those terms with $(m'',\beta'')=(1,0)$ in the first sum of the right side corresponds to the left side of the Stokes' formula (\ref{stokes_Kuranishi_eq}); all others correspond to the right side of (\ref{stokes_Kuranishi_eq}): the first and second summations correspond to (\ref{moduli_M_2_k+1_path_boundary_eq}), and the third and forth summations correspond to (\ref{M_2_k+1_Sigma_012_eq}).
As in (\ref{hat_q_eq}), we define a map
\[
\mathfrak {\hat Q} : \Omega^*(X) \otimes \Omega^*(X)  \to \rCC(\OL)
\]
by setting $\big(\mathfrak {\hat Q}(g_1,g_2 )\big)_{k,\beta} = \mathfrak Q_{2,k,\beta}(g_1,g_2; \cdots)$.
By Proposition \ref{operator_q_prop} (v), the image of $\mathfrak {\hat Q}$ satisfies the condition (\ref{rCC_condition_eq}) for $\rCC(\OL)$.
Now, the above relation can be concisely written as follows:
\[
\mathfrak {\hat Q} ( dg_1, g_2) +\mathfrak {\hat Q} (g_1, dg_2)
=
\delta_{\check \m^{J,L}} ( \mathfrak {\hat Q} (g_1,g_2))
\pm \big( \hat \q  (  g_1\pmb\ast g_2 ) - \hat\q(g_1) \cup \hat \q(g_2) \big)
\]
Accordingly, the map $ [\hat\q]$ in (\ref{[hat_q]_eq}) satisfies that for $g_1,g_2\in H^*(X)$, we have
\[
[\hat \q] ( g_1\pmb\ast g_2)= [\hat\q] (g_1) \cup [\hat\q] (g_2)
\]
Ultimately, it remains to show the unitality.
Let $\one_X\in \Omega^0(X)$ be the constant and $\mathfrak e$ be given in (\ref{unit_rHH_eq}). Then, we aim to show $\hat\q (\one_X)= \mathfrak e$.
In fact, suppose $S:=\hat \q (\one_X)_{k,\beta}(h_1,\dots, h_k)=\q_{1,k,\beta}(\one_X; h_1,\dots, h_k)$ is nonzero.
If $\beta=0$, then using Proposition \ref{operator_q_prop} (iv) implies that $k=0$ and $S=\q_{1,0,0}(\one_X)=\one_L$.
	If $\beta\neq 0$, it is defined by the moduli space $\mathcal M_{1,k+1,\beta}(J,L)$. We consider the forgetful map $\forget: \mathcal M_{1,k+1,\beta}(J,L)\to \mathcal M_{k+1,\beta}(J,L)$.
	To produce a nonzero term, the dimension of the fiber of $\forget$ needs to be zero just as what we discussed before. But this is impossible when $\beta\neq 0$. See also \cite{FuCyclic}.
\end{proof}

\section{Closed-open maps with quantum corrections}
\label{s_closed_open}

\subsection{The maps $\Theta$ and $\mathbb P$}

Recall that the moduli space system of $J$-holomorphic disks bounding $L$ gives rise to a chain-level $A_\infty$ algebra $(\OL, \check \m^{J,L})$. A metric $g$ induces the so-called harmonic contraction for which applying the homological perturbation to $\check\m^{J,L}$ yields the minimal model $A_\infty$ algebra $(\HL, \m^{g,J,L})$ along with an $A_\infty$ homotopy equivalence $\mi^{g,J,L}: \m^{g,J,L}\to \check \m^{J,L}$ (see \S \ref{sss_moduli_space_A_infinity})

For simplicity, we set $\check \m =\check \m^{J,L}$, $\m=\m^{g,J,L}$, and $\mi=\mi^{g,J,L}$. Note that they are contained in $\UD$.
The Whitehead theorem can be generalized to $\UD$; namely, there is a ud-homotopy inverse $\mi^{-1}$ of $\mi$ in the sense of Theorem \ref{whitehead_thm}.
Concerning (\ref{rCC_eq}), we define the following map:
\begin{equation}
	\label{map_Theta_eq}
	\Theta: \rCC(\OL)\to \rCC(\HL) \qquad \varphi \mapsto (\mi^{-1} \{ \varphi\}) \diamond \mi
\end{equation}
We apologize for the lack of motivation, but it is truly just found out by trial-and-error.
We may assume $\varphi=(\varphi_{k,\beta})_{k\ge 0, \beta\in\pi_2(X,L)}$ is homogeneously-graded, say $|\varphi|=p$. Then, we explicitly have
\begin{align*}
\Theta(\varphi)_{k,\beta} 
=
\displaystyle \sum_{\Scale[0.68]{\substack{\lambda,\mu,\nu\ge 0 \\ \sum r_i+\sum k_i+\sum s_i=k \\
\beta'+\beta''+\sum\alpha_i+\sum\beta_i+\sum\gamma_i=\beta}}
}
 \mi^{-1}_{\lambda+\mu+1,\beta'} \circ \Big(
\mi^{\# p}_{r_1,\alpha_1}\otimes \cdots \otimes \mi^{\# p}_{r_\lambda, \alpha_\lambda} \otimes \varphi_{\nu,\beta''}(\mi_{k_1,\beta_1}\otimes \cdots \otimes \mi_{k_\nu,\beta_\nu} )\otimes \mi_{s_1,\gamma_1}\otimes \cdots \otimes \mi_{s_\mu,\gamma_\mu}
\Big)
\end{align*}
Briefly, this reads $\Theta(\varphi)=\sum \mi^{-1} (\mi^{\# p} \cdots \mi^{\# p} \ \varphi (\mi \cdots \mi) \ \mi \cdots \mi)$.
One can check that $\Theta$ is $\Lambda^X$-linear (\S \ref{sss_Lambda_X_module}). Further, given $\varphi\in \rCC(\OL)$, the image $\Theta(\varphi)$ satisfies (\ref{rCC_condition_eq}) and so lies in $\rCC(\HL)$, since $\mi$ and $\mi^{-1}$ are unital.
The label degree is $|\Theta|=0$, i.e.  $|\Theta(\varphi)|=|\varphi|$. Moreover, for the $\mathfrak e$ in (\ref{unit_rHH_eq}), we have
\begin{equation}
	\label{Theta_mathfrak_e_eq}
	\Theta(\mathfrak e)=\mathfrak e
\end{equation}
as $\Theta(\mathfrak e)_{0,0}= \mi^{-1}_{1,0} (\mathfrak e_{0,0})=\mi^{-1}_{1,0}(\one)=\one=\mathfrak e_{0,0}$. (Here we slightly abuse the notation $\mathfrak e$.)

As said in the introduction, it is generally \textit{incorrect} that the $\Theta$ would induce a ring homomorphism between the two reduced Hochschild cohomologies $\rHH(\OL,\check \m)$ and $\rHH(\HL, \m)$, but it is not far from so.
Heuristically, if we could pretend that $\mi^{-1}\diamond \mi= \id$,
then the $\Theta$ would induce an honest ring homomorphism.
But, $\mi^{-1}\diamond \mi$ is only ud-homotopic to $\id$, especially when the nontrivial Maslov-0 disks are allowed.
Shortly, the above says that the gap for producing an honest ring homomorphism can be controlled by the ud-homotopy relations in $\UD$. Anyway, we start with some useful computations:

\begin{lem}
	\label{Theta_computation_lem}
	Given $\varphi\in \rCC(\OL)$, we have
	\begin{equation*}
		\Scale[0.92]{\Theta(\delta_{\check \m} \varphi)=  \big( (\m\diamond \mi^{-1}) \{ \varphi\} \big) \diamond \mi -(-1)^{|\varphi|} \Theta(\varphi) \{ \m\}} 
	\end{equation*}
Given $\varphi,\psi\in \rCC(\OL)$, we have
	\begin{equation*}
			\Scale[0.92]{\Theta(\varphi\cup_{\check\m} \psi)  + \mi^{-1}\{\delta_{\check\m} \varphi,\psi\} \diamond \mi + (-1)^{|\varphi|} \mi^{-1} \{\varphi, \delta_{\check\m} \psi\} \diamond \mi
		=
		\big((\m \diamond \mi^{-1}) \{\varphi,\psi\}\big) \diamond \mi - (-1)^{|\varphi|+|\psi|}  \big(\mi^{-1} \{\varphi,\psi\} \diamond \mi \big) \{ \m\}}
	\end{equation*}
\end{lem}

\begin{proof}
	(i) May assume $\varphi$ is homogeneously-graded and $|\varphi|=p$. 
	First, by Proposition \ref{brace_property_prop}, we obtain $\mi^{-1}\{ \delta_{\check\m}\varphi \}= \mi^{-1}\{\check\m\{\varphi\}\}-(-1)^p \mi^{-1}\{\varphi\{\check\m\}\}= \mi^{-1}\{\check\m\} \{\varphi\} - (-1)^p \mi^{-1}\{\varphi\}\{ \check\m\}$.
	Because $\mi^{-1}$ is an $A_\infty$ homomorphism from $\check \m$ to $\m$, we have $\m\diamond \mi^{-1}=\mi^{-1} \{\check\m\}$. Moreover, using the fact that $\mi: \m\to \check \m$ is an $A_\infty$ homomorphism, we have $\big( \mi^{-1}\{\varphi\}\{\check\m\} \big) \diamond \mi = \big( \mi^{-1}\{\varphi\} \diamond \mi \big) \{ \m\}=\Theta(\varphi)\{\m\}$.
	
	(ii) May assume $\varphi$ and $\psi$ are homogeneously-graded and $|\varphi|=p, |\psi|=q$. By Proposition \ref{brace_property_prop},
	\[
	\mi^{-1}\{\check\m\} \{\varphi,\psi\} -(-1)^{p+q} \mi^{-1}\{\varphi,\psi\} \{\check\m\} = \mi^{-1}\{\delta_{\check\m} \varphi,\psi\} + (-1)^p \mi^{-1} \{\varphi, \delta_{\check\m} \psi\} +  \mi^{-1} \{\varphi\cup_{\check\m}\psi\}
	\]
	Applying $- \diamond\mi$ on both sides, and using the $A_\infty$ equations of both $\mi^{-1}$ and $\mi$, we complete the proof.
\end{proof}

	
Finally, we define the length-zero projection map:
	\begin{equation}
		\label{map_P_eq}
	\mathbb P: \rCC(\HL) \to   H^*(L)\hat\otimes \Lambda[[\pi_1(L)]] / \ia    \qquad  \mathfrak f \mapsto \sum_\beta T^{E(\beta)} Y^{\partial\beta} \mathfrak f_{0,\beta}
	\end{equation}
One can check that $\mathbb P$ is $(\Lambda, \Lambda^X)$-linear in the sense that $\mathbb P(t^A\cdot \f)= T^{E(A)} \mathbb P(\f)$.
For (\ref{unit_rHH_eq}), we have
\begin{equation}
	\label{P_e=one_eq}
	\mathbb P(\mathfrak e)= T^0 Y^0 \mathfrak e_{0,0}= \mathfrak e_{0,0}=\one
\end{equation}

\subsection{From reduced Hochschild cohomology to self Floer cohomology}


There is really no hope to cook up a unital ring homomorphism from $\rHH(\OL,\check \m)$ to $\rHH(\HL,\m)$. Notwithstanding, the following result will be sufficient for our purpose:

\begin{thm}
	\label{from_HH_to_HF_thm}
The composite map $\mathbb P  \Theta$ induces a unital ring homomorphism
\[
\Phi:=[\mathbb P  \Theta]: \rHH(\OL,\check \m) \to \HF(L, \m)
\]
\end{thm}

\begin{proof}
	We notice $\mathbb P  \Theta (  \mathfrak e) =\one$ by (\ref{Theta_mathfrak_e_eq}) and (\ref{P_e=one_eq}).
	The following two items are what we need:

\begin{note}
	\textbf{\emph{(a).}} $\mathbb P   \Theta (\delta_{\check \m} \varphi) =  \mathbf m_1 \big(\mathbb P  \Theta(\varphi) \big)$
\end{note}
\begin{note}
	\textbf{\emph{(b).}} If $\varphi,\psi \in \rCC(\OL)$ are $\delta_{\check\m}$-closed, then
	\[
	\mathbb P \Theta( \varphi \cup_{\check \m} \psi) 
	= \mathbf m_2
	\big(\mathbb P   \Theta(\varphi), \mathbb P   \Theta(\psi) \big)
	+ \mathbf m_1 \big( \mathbb P( (\mi^{-1}\{\varphi,\psi\})\diamond \mi ) \big)
	\]
\end{note}

First, may assume $\varphi\in \rCC(\OL)$ is homogeneously-graded and $|\varphi|=p$.
Concerning the first half of Lemma \ref{Theta_computation_lem}, we compute that $\mathbb P \left( (\m\diamond \mi^{-1}) \{ \varphi\} ) \diamond \mi \right)$ equals to
\begin{align*}
	\sum_{\Scale[0.77]{\substack{ \lambda,\mu \ge 0 \ ; \  \beta',\beta'',\alpha_1,\dots,\alpha_\lambda, \gamma_1,\dots,\gamma_\mu	}}
	}
T^{E(\beta')} Y^{\partial\beta'}  \m_{\lambda+\mu+1,\beta'} \Big(
T^{E(\alpha_1)} Y^{\partial \alpha_1} (\mi^{-1}\diamond \mi)^{\# p}_{0 , \alpha_1} ,\dots, T^{E(\alpha_\lambda)} Y^{\partial \alpha_\lambda} (\mi^{-1}\diamond \mi)^{\# p}_{0,\alpha_\lambda} , \\
T^{E(\beta'')}  Y^{\partial\beta''} \big( ( \mi^{-1} \{ \varphi\})\diamond \mi \big)_{0,\beta''} ,  T^{E(\gamma_1)}Y^{\partial\gamma_1} (\mi^{-1}\diamond \mi)_{0,\gamma_1},\dots, T^{E(\gamma_\mu)} Y^{\partial\gamma_\mu} (\mi^{-1}\diamond \mi)_{0,\gamma_\mu}
	\Big)
\end{align*}
Moreover, since $\mi^{-1}\diamond\mi \simud \id$, we know $\sum_\alpha T^{E(\alpha)} Y^{\partial\alpha} (\mi^{-1}\diamond \mi)_{0,\alpha}\equiv \sum_\alpha T^{E(\alpha)}Y^{\partial\alpha} (\id)_{0,\alpha} \equiv 0 \  (\mathrm{mod} \ \ia)$ due to Lemma \ref{weakMC_observation_lem}.
Accordingly, those terms with $\lambda\neq 0$ or $\mu\neq 0$ all vanish modulo $\ia$. So, we have:
\[
\mathbb P \left( \big(\m\diamond \mi^{-1}) \{ \varphi\} \big)\diamond \mi \right) 
=
\sum_{\beta'+\beta''=\beta} T^{E(\beta')}Y^{\partial\beta'} \m_{1,\beta'} \big(T^{E(\beta'')} Y^{\partial\beta''}  \big( (\mi^{-1}\{\varphi\})\diamond \mi \big)_{0,\beta''} \big) = \mathbf m_1( \mathbb P  \Theta(\varphi))
\]
By the first half of Lemma \ref{Theta_computation_lem}, it remains to prove
\[
	\mathbb P ( \Theta(\varphi)\{ \m\}) 
	=
	\sum T^{E(\beta')}Y^{\partial\beta'} \big(\Theta (\varphi)\big)_{1,\beta'} \big(T^{E(\beta'')}Y^{\partial\beta''} \m_{0,\beta''} \big)
	\equiv 0 \ (\mathrm{mod} \ \ia)
\]
Indeed, we first note that $\sum T^{E(\beta'')}Y^{\partial\beta''} \m_{0,\beta''} =W\cdot \one$ modulo $\ia$. Besides, all of $\mi^{-1} , \varphi, \mi$ satisfy the condition (\ref{rCC_condition_eq}) except $\mi_{1,0}(\one)=\one$ and $\mi^{-1}_{1,0}(\one)=\one$, so any term in the expansion of $\Theta(\varphi)_{1,\beta'}(\one)\equiv \big( (\mi^{-1}\{\varphi\})\diamond\mi\big)_{1,\beta'}(\one)$ must vanish.
This completes the proof of (a).

Second, we assume $\varphi, \psi \in \rCC(\OL)$ are $\delta_{\check\m}$-closed and homogeneously-graded, say $|\varphi|=p$ and $|\psi|=q$.
	For the second half of Lemma \ref{Theta_computation_lem}, we compute that $		\mathbb P
	\left(
	( (\m\diamond \mi^{-1})\{ \varphi,\psi \} )\diamond \mi
	\right)$ equals to
	\begin{align*}
&		\mathbb P \Big(
		\m \big( \dots, (\mi^{-1}\diamond \mi)^{\# (p+q)},\dots,  ((\mi^{-1}\{\varphi\})\diamond \mi )^{\# q}, \dots, (\mi^{-1}\diamond \mi)^{\# q},\dots , (\mi^{-1}\{\psi\} )\diamond \mi, \dots, \mi^{-1}\diamond \mi,\dots  
		\big) \Big) \\
+&  \mathbb P \Big(
\m \big( \dots, (\mi^{-1}\diamond \mi)^{\# (p+q)},\dots,  (\mi^{-1}\{\varphi,\psi\} ) \diamond \mi,  \dots, \mi^{-1}\diamond \mi,\dots  
\big) \Big)
\end{align*}
Just as above, by applying Lemma \ref{weakMC_observation_lem} to $\mi^{-1}\diamond \mi\simud \id$, we know it further equals to (modulo $\ia$)
\begin{align*}
&\textstyle
\sum_{\beta_0,\beta_1,\beta_2} 
T^{E(\beta_0)} Y^{\partial\beta_0} \m_{2,\beta_0} 
\Big(
T^{E(\beta_1)} Y^{\partial\beta_1} \big( (\mi^{-1}\{ \varphi\})\diamond \mi \big)_{0,\beta_1}, T^{E(\beta_2)}Y^{\partial\beta_2}  \big( (\mi^{-1}\{ \psi\})\diamond \mi\big)_{0,\beta_2}
\Big) \\
+&
\textstyle \sum_{\beta_0, \beta_1} T^{E(\beta_0)} Y^{\partial\beta_0} \m_{1,\beta_0} 
\Big( T^{E(\beta_1)} Y^{\partial\beta_1}  \big(\mi^{-1}\{\varphi, \psi\} \diamond \mi \big)_{0,\beta_1} \Big) 	 	
=
\mathbf m_2
\big(\mathbb P   \Theta(\varphi), \mathbb P   \Theta(\psi) \big) + \mathbf m_1 \big(\mathbb P\big( (\mi^{-1}\{\varphi,\psi\}) \diamond \mi \big) \big)
	\end{align*}

By Lemma \ref{Theta_computation_lem},
since $\delta_{\check \m}\varphi=\delta_{\check \m} \psi=0$,
it remains to prove that $\mathbb P \big( ( \mi^{-1} \{\varphi,\psi\}\diamond \mi)\{ \m\} \big) \equiv  0  \ (\mathrm{mod} \ \ia)$.
Indeed, it can be proved by almost the same way as in part (a): first, $\sum T^{E(\alpha)}Y^{\partial\alpha} \m_{0,\alpha}=W\cdot \one \ (\mathrm{mod} \ \ia)$; then, just use the conditions (\ref{rCC_condition_eq}) of $\mi^{-1},\varphi,\psi,\mi$. 
Putting things together, the proof is now complete.
\end{proof}

\subsection{Closed-open maps with Maslov-0 disks} \label{ss_closed_open}

By Theorem \ref{[hat_q]_ring_homo_thm} and Theorem \ref{from_HH_to_HF_thm}, the composition of $[\hat \q]$ and $\Phi\equiv [\mathbb P    \Theta]$:
\[
\xymatrix{ \Psi:
 & QH^*(X; \Lambda^X) \ar[rr]^{[\hat \q]\qquad } & & \rHH(\OL,\check \m)  \ar[rr]^{\Phi=[\mathbb P  \Theta]} & & \HF(L, \m)
}
\]
is a unital ring homomorphism.
Recall that we have defined two versions of quantum cohomologies $QH^*(X; \Lambda^X)=H^*(X)\hat\otimes \Lambda^X$ and $QH^*(X; \Lambda)=H^*(X)\hat\otimes \Lambda$ in (\ref{QH_two_defn_eq}).
As $\HF(L,\m)$ is a $\Lambda$-algebra, we can consider the $\Lambda$-linear extension of the restriction of $\Psi = \Phi\circ [\hat\q]$ on $H^*(X; \mathbb R)$, which gives a map:
\begin{equation}
	\label{CO_gJL_eq}
	\CO: \quad  QH^*(X; \Lambda)\to \HF(L,\m)
\end{equation}

\begin{prop}
	The $\CO$ is a unital $\Lambda$-algebra homomorphism and
$\CO(\one_X)= [\one]$.
\end{prop}

\begin{proof}
	Denote the quantum products in $QH^*(X;\Lambda^X)$ and $QH^*(X;\Lambda)$ by $\pmb \ast^X$ and $\pmb \ast $ respectively for a moment to distinguish. Recall that $\Theta$ and $\hat \q$ are $\Lambda^X$-linear, and $\mathbb P$ is $(\Lambda, \Lambda^X)$-linear. So, by construction, $\Psi$ is also $(\Lambda, \Lambda^X)$-linear, that is, $\Psi( t^A\cdot g)= T^{E(A)} \Psi(g)$. Since $\CO$ is the $\Lambda$-linear extension, we see that $\Psi(g {\pmb \ast^X}  h)=\CO(g\pmb \ast h)$ for any $g,h\in H^*(X;\mathbb R)$.
	Let $\pmb g, \pmb h\in QH^*(X; \Lambda)\equiv H^*(X)\hat\otimes \Lambda$, and write $\pmb g= \sum_{i=1}^\infty T^{\lambda_i} g_i$ and $\pmb h= \sum_{j=1}^\infty T^{\mu_j} h_j$, where $\lambda_i \nearrow \infty$, $\mu_j\nearrow \infty$, and $g_i,h_j\in H^*(X;\mathbb R)$. Now, we have
$
		\CO(\pmb g \pmb \ast \pmb h)= \sum_{i,j\ge 1} T^{\lambda_i+\mu_j} \CO (g_i\pmb \ast  h_j)=\sum_{i,j\ge 1} T^{\lambda_i+\mu_j} \Psi (g_i {\pmb \ast^X}  h_j)
		=
		\sum_{i,j\ge 1} T^{\lambda_i+\mu_j} \Psi (g_i) \cdot  \Psi(h_j)
		= \CO(\pmb g)\cdot \CO(\pmb h)
$.
Finally, $\CO(\one_X)=\Psi(\one_X)=\Phi \circ [\hat \q] (\one_X) =\Phi(\mathfrak e)= [\one]$.
\end{proof}

On the other hand, just like the Gromov-Witten theory, the operator $\q$ also satisfies the divisor axiom for the interior marked points. Specifically,
if $g_1\in Z^2(X)$ is a closed 2-form supported in $X\setminus L$, then
\begin{equation}
	\label{divisor_axiom_for_q_eq}
	\q_{\ell,k,\beta}(g_1,g_2,\dots, g_\ell; h_1,\dots, h_k) = \textstyle \int_\beta g_1 \cdot \q_{\ell-1,k,\beta}(g_2,\dots, g_\ell ;h_1,\dots, h_k)
\end{equation}

Denote by $c_1=c_1(X) \in H^2(X)$ the first Chern class of $X$. 
Note that there exists a cycle $Q$ in $X\setminus L$ with $\beta\cap Q=\mu(\beta)$ and $[Q]$ is Poincare dual to $2c_1$ \cite[\S 23.3]{FOOOSpectral}.
Together with Proposition \ref{operator_q_prop} (iii), we obtain
\begin{equation}
	\label{hat_q_c_1_eq}
	\big(\hat \q(c_1) \big)_{k,\beta}
	=
	\tfrac{\mu(\beta)}{2} \cdot  \check \m_{k,\beta} 
\end{equation}

\begin{prop}
	\label{CO_c_1_W_prop}
	$\CO(c_1)= W\cdot [\one]$.
\end{prop}

\begin{proof}
First, we compute in the cochain level:
\begin{align*}
	\mathbb P\circ \Theta \circ \hat\q (c_1) 
	&
	=
	\mathbb P \Big(	(\mi^{-1} \{ \hat \q(c_1)\})\diamond \mi	\Big) = \mathbb P\Big(	\mi^{-1}\big(\mi^\#,\dots, \mi^\#, \ \hat \q(c_1)\diamond \mi, \ \mi,\dots,\mi \big)	\Big) \\
	&
	=
	\sum_{s, t\ge 0, \B\in\pi_2(X,L)} \
	\sum_{\alpha+\beta+\sum \gamma_i+\sum \eta_i=\B}
	T^{E(\B)} Y^{\partial\B} \cdot
	 \mi^{-1}_{s+t+1,\alpha} \big( \mi^\#_{0,\gamma_1}, \dots, \mi^\#_{0,\gamma_s}, \ \big(\hat \q(c_1)\diamond \mi \big)_{0,\beta}, \ \mi_{0,\eta_1},\dots, \mi_{0,\eta_t} \big)
\end{align*}
Since $\mi_{0,\gamma}\in \Omega^{1-\mu(\gamma)}(L)$, the semipositive condition implies that $\mu(\gamma)=0$ whenever $\mi_{0,\gamma}\neq 0$.
By this observation and by (\ref{hat_q_c_1_eq}), the term 
\[
\big(\hat \q(c_1)\diamond \mi \big)_{0,\beta}  
=
\sum_{\beta_0+\cdots+\beta_\ell=\beta} 
\hat \q(c_1)_{\ell,\beta_0}
(\mi_{0,\beta_1},\dots, \mi_{0,\beta_\ell})
=
\sum_{\beta_0+\cdots+\beta_\ell=\beta} 
\tfrac{\mu(\beta_0)}{2} \cdot \check \m_{\ell,\beta_0}
(\mi_{0,\beta_1},\dots, \mi_{0,\beta_\ell})
\]
is contained in $\Omega^{2-\mu(\beta_0)}(L)$. By the semipositive condition again, we may assume $\mu(\beta_0)=0$ or $2$. Hence, whenever it is nonzero, we may always assume $\mu(\beta_0)=2$ and $\mu(\beta_i)=0$ for other $1\le i \le \ell$, in particular, $\mu(\beta)=2$.
Then, using the $A_\infty$ associativity relation $\check\m \diamond \mi= \mi\{\m\}$ yields that
\begin{align*}
	\big(\hat \q(c_1)\diamond \mi \big)_{0,\beta}  
	=
	\sum_{\mu(\beta_0)=2} \check \m_{\ell,\beta_0} (\mi_{0,\beta_1},\dots, \mi_{0,\beta_\ell})  
	=
	\Big( \sum \check \m_{\ell,\beta_0} (\mi_{0,\beta_1},\dots, \mi_{0,\beta_\ell})  \Big) \Big|_{\Omega^0(L)}
	=
	\Big(
	\sum \mi_{1,\beta'} (\m_{0,\beta''} ) \Big) \Big|_{\Omega^0(L)}
\end{align*}
Notice that $\mu(\beta')+\mu(\beta'')=\mu(\beta)=2$.
Thus, 
\begin{align*}
	\sum_\beta T^{E(\beta)} Y^{\partial\beta}  \big(\hat \q(c_1)\diamond \mi \big)_{0,\beta}
	&
	=
	\sum_{\mu(\beta)=2} T^{E(\beta)} Y^{\partial\beta}  \big(\hat \q(c_1)\diamond \mi \big)_{0,\beta} \\
	&
=	\sum_{\mu(\beta')=0} T^{E(\beta')} Y^{\partial\beta'} \mi_{1,\beta'} \big( \sum_{\mu(\beta'')=2} T^{E(\beta'')} Y^{\partial\beta''} \m_{0,\beta''} \big)  \  
	= \mi_{1,0}(W\cdot \one)=W\cdot \one \quad (\mathrm{mod} \ \ia)
\end{align*}
%
where the terms with $\mu(\beta'')=0$ are killed modulo $\ia$.
Beware that we abuse the notation $\one$ to represent the constant-one functions in both $\Omega^0(L)$ and $H^0(L)$.
Finally, back to the calculation at the start, using the unitality of $\mi^{-1}$ further deduces that $s=t=0$ and $\alpha=0$ there. To conclude,
\[
\mathbb P\circ \Theta\circ \hat \q(c_1) 
= \mi^{-1}_{1,0}(W \cdot \one) = W\cdot \one \qquad (\mathrm{mod} \ \ia)
\]
Passing to the cohomology, we get $\CO(c_1)=\Psi(c_1)=W\cdot [\one]$.
The proof is now complete.
\end{proof}

Suppose $\mathbf y \in H^1(L; U_\Lambda)$ can lift to a weak bounding cochain $b\in H^1(L; \Lambda_0)$; or equivalently, suppose the ideal $\ia$ vanishes at $\mathbf y$, i.e. $Q(\mathbf y)=0$.
(Indeed, up to Fukaya's trick, we may even allow $\mathbf y$ to lie in a neighborhood of $H^1(L;U_\Lambda)$ in $H^1(L;\Lambda^*)$.)
By Lemma \ref{Eva_y_HF_to_HF_lem}, we get a unital $\Lambda$-algebra homomorphism $\mathcal E_{\mathbf y}: \HF(L,\m) \to \HF(L,\m,\mathbf y)$.
Accordingly, the composition
\begin{equation}
	\label{CO_Eva_y_eq}
	\xymatrix{
		\CO_{\mathbf y}:=\mathcal E_{\mathbf y} \circ \CO: &  QH^*(X; \Lambda) \ar[rr]^{\CO} & & \HF(L,\m) \ar[rr]^{\mathcal E_{\mathbf y}} &  & \HF(L,\m, \mathbf y)
	}
\end{equation}
is also a unital $\Lambda$-algebra homomorphism. In particular,
\begin{equation}
	\label{CO_y_1_X_eq}
	\mathbb {CO}_{\mathbf y} (\one_X)= [\one]
\end{equation}
Moreover, by Proposition \ref{CO_c_1_W_prop}, we also have
\begin{equation}
	\label{CO_y_c_1_eq}
	\mathbb {CO}_{\mathbf y} (c_1)= W(\mathbf y)\cdot [\one]
\end{equation}
After all the above preparations, there is no further obstacle for the proof of our local result:

\begin{thm}
	[Theorem \ref{Main_thm_c_1_general_thm}]
	\label{eigenvalue_general_thm}
	Suppose the cohomology ring $H^*(L)$ is generated by $H^1(L;\mathbb Z)$. Under Assumption \ref{assumption_weak_MC}, any critical value of $W$ is an eigenvalue of $c_1\pmb\ast :QH^*(X;\Lambda)\to QH^*(X;\Lambda)$.
\end{thm}

\begin{proof}
	Let $\mathbf y$ be a critical point of $W$ in the sense of Definition \ref{critical_point_W_defn}. We aim to show $W(\mathbf y)\in \Lambda$ is an eigenvalue of $c_1\pmb\ast$.
	Arguing by contradiction, suppose $a:=c_1-W(\mathbf y) \one_X$ was invertible in $QH^*(X; \Lambda)$. Then, there would exist $a^{-1}$ with $a{\pmb\ast} a^{-1}=\one_X$. 
	Notice that $\CO_{\mathbf y}(a)=\CO_{\mathbf y}(c_1)-W(\mathbf y) \CO_{\mathbf y}(\one_X)=0$ by (\ref{CO_y_1_X_eq}) and (\ref{CO_y_c_1_eq}).
	Since $\mathbb {CO}_{\mathbf y}$ is a unital $\Lambda$-algebra homomorphism, we conclude that
	\begin{align*}
		[\one]= \CO_{\mathbf y}(\one_X) = \CO_{\mathbf y} (a {\pmb\ast} a^{-1}) =\CO_{\mathbf y}(a) \cdot \CO_{\mathbf y} (a^{-1}) =0
	\end{align*}
Namely, we must have $\HF(L,\m,\mathbf y)=0$.
But, due to Theorem \ref{nonvanishing_HF_thm}, the $\mathbf y$ being a critical point of $W$ implies that $\HF(L,\m,\mathbf y)\neq 0$. This is a contradiction.
\end{proof}

\subsection{Global result and the family Floer program}
\label{ss_implication_family_Floer}

Lastly, we want to adapt Theorem \ref{eigenvalue_general_thm} to the non-archimedean SYZ picture, showing Conjecture \ref{conjecture_folklore} for the analytic mirror Landau-Ginzburg model in Theorem \ref{Main_theorem_thesis}. That is, we aim to show Theorem \ref{Main_thm_c_1}.


\subsubsection{General aspects}
A brief review of Theorem \ref{Main_theorem_thesis} is as follows.
First, using the moduli space of holomorphic disks and the homological perturbation, we obtain an $A_\infty$ algebra $(H^*(L),\m)$ in $\Obj \UD(L)$ for each Lagrangian torus fiber $L=L_q=\pi^{-1}(q)$.
Take a small rational polyhedron $\Delta$ in $B_0$ near $q$, and we can identify $\Delta$ with a subset in $\mathbb R^n$ via integral affine coordinates.
Define the $W$ and $\ia$ for $\m$ as in \S \ref{ss_weak_MC}, but now the ideal $\ia$ vanishes under Assumption \ref{assumption_weak_MC}.
Now, a local chart of $X^\vee$ is identified with a polyhedral affinoid domain $\trop^{-1}(\Delta)$ for the fibration map $\trop$ in (\ref{trop_eq}).
From the rigid analytic geometry, we know $\trop^{-1}(\Delta)$ can be recognized as the spectrum $\Sp\Lambda\langle \Delta\rangle$ of maximal ideals of the following affinoid algebra contained in $\Lambda[[\pi_1(L)]]\cong \Lambda[[Y_1^\pm,\dots, Y_n^\pm]]$:
\begin{equation}
	\label{affinoid_algebra_eq}
\Lambda\langle \Delta \rangle = \Big\{ \textstyle \sum_{\alpha\in\pi_1(L)} c_\alpha Y^\alpha 
\mid
\val (c_\alpha)+ \langle \alpha, v\rangle \to \infty, \, \, \text{for any} \ v\in \Delta
\Big\} 
\end{equation}
Recall that $\Lambda\langle \Delta\rangle$ consists of all those formal power series in $\Lambda[[\pi_1(L)]]$ which converge on $\trop^{-1}(\Delta)$.
Using the reverse isoperimetric inequalities \cite{ReverseI} and shrinking $\Delta$ if necessary, we may require the $\mathbf m_k\equiv \sum_\beta T^{E(\beta)} Y^{\partial\beta}\m_{k,\beta}$ maps $H^*(L)\hat\otimes \Lambda\langle \Delta\rangle$ to $H^*(L)\hat\otimes \Lambda\langle \Delta\rangle$ in (\ref{CF_m_k_eq}). In particular, the $W\equiv \mathbf m_0$ lies in $\Lambda\langle \Delta\rangle$ or equivalently converges on $\trop^{-1}(\Delta)$.

The affinoid space $\trop^{-1}(\Delta)$ embeds into the mirror analytic space $X^\vee$ as a local chart. 
In reality, the transition maps $\phi$ among all these local charts are given by restricting the homomorphisms in (\ref{phi_eq}) to some affinoid algebra like (\ref{affinoid_algebra_eq}).
Eventually, by \cite{Yuan_I_FamilyFloer}, the various local charts glue to the mirror rigid analytic space $X^\vee$.

\subsubsection{Proof of the global result}

Now, we are ready to prove Theorem \ref{Main_thm_c_1}.
Let $\lambda: = W^\vee(\mathbf y)$ be a critical value of the mirror Landau-Ginzburg superpotential $W^\vee$ at a point $\mathbf y$ in the mirror analytic space $X^\vee$.
In view of Corollary \ref{crit_point_inv_cor}, we may choose a specific local chart of $X^\vee$ to present the critical point $\mathbf y$.
Actually, since $X^\vee=\bigcup_q H^1(L_q; U_\Lambda)$ as a set, there is a Lagrangian fiber $L=L_q$ so that its dual fiber $H^1(L; U_\Lambda)$ contains the point $\mathbf y$. Remark that one may also take any other adjacent Lagrangian fiber $L'$ thanks to the Fukaya's trick.
Next, we find a nearby local chart which corresponds to an $A_\infty$ algebra $(\HL, \m)\in \Obj\UD$ associated to $L$. As said before, choosing integral affine coordinates near $q$, this local chart can be identified with a polyhedral affinoid domain $\trop^{-1}(\Delta)$, and the dual fiber $H^1(L;U_\Lambda)$ is then identified with the central fiber $\trop^{-1}(0)\equiv U_\Lambda^n$. The critical point can be also viewed as a point in $\trop^{-1}(\Delta) \subset (\Lambda^*)^n$, say $\mathbf y=(y_1,\dots, y_n)$.
Let $W$ be the local expression of $W^\vee$ in this local chart. Recall that the nonvanishing of $\HF(L,\m, \mathbf y)$ does not rely on the local chart, and so does
the critical value $\lambda= W(\mathbf y)$.
Now, it follows from Theorem \ref{eigenvalue_general_thm} that $\lambda$ is an eigenvalue of $c_1$.
Ultimately, applying the same argument to all critical points of $W^\vee$, we complete the proof of Theorem \ref{Main_thm_c_1}.

\paragraph{\pmb{Acknowledgment.}}
The author thanks Kenji Fukaya for his constant encouragement, and Ezra Getzler and Boris Tsygan for helpful comments on Hochschild cohomology. The author is also grateful to Mohammed Abouzaid, Denis Auroux, Andrew Hanlon, Mark McLean, Yuhan Sun, and Eric Zaslow for their interest and stimulating discussions. Special thanks to the anonymous referee for the meticulous review and invaluable suggestions.

\appendix

\bibliographystyle{alpha}
\bibliography{mybib_critical}

\end{document}